\documentclass[reqno,twoside,12pt]{article}
 \usepackage[english]{babel}
 \usepackage[T1]{fontenc}
 \usepackage{times}

 \usepackage[
    paper=a4paper,
    portrait=true,
    textwidth=425pt,
    textheight=650pt,
    tmargin=3cm,
    marginratio=1:1
            ]{geometry}
 \usepackage[all]{xy}       
 \usepackage[centertags]{amsmath}
 \usepackage{amsfonts}
 \usepackage{amsthm}
 \usepackage{amssymb}
 \usepackage{enumerate}
 \usepackage{mathrsfs}      
 \usepackage[active]{srcltx}    
 \usepackage{color}

 \theoremstyle{plain}
 \newtheorem{Thm}{Theorem}[section]
 \newtheorem*{Thm*}{Theorem}
 \newtheorem{Cor}[Thm]{Corollary}
 \newtheorem{Lemma}[Thm]{Lemma}
 \newtheorem{Prop}[Thm]{Proposition}

 \theoremstyle{definition}
 \newtheorem{Rem}[Thm]{Remark}
 \newtheorem{Ex}[Thm]{Example}
 \newtheorem{Defi}[Thm]{Definition}

 \numberwithin{Thm}{section}
 \numberwithin{equation}{section}

\def\medno{\medbreak\noindent}
\def\text#1{\;\;\;\;{\rm \hbox{#1}}\;\;\;\;}
\def\qquad{\quad\quad}

\def\itema{\item[{\rm (a)}]}
\def\itemb{\item[{\rm (b)}]}
\def\itemc{\item[{\rm (c)}]}
\def\itemd{\item[{\rm (d)}]}

\def\msy#1{{\mathbb #1}}
\def\C{{\msy C}}
\def\N{{\msy N}}
\def\Z{{\msy Z}}
\def\R{{\msy R}}
\def\D{{\msy D}}

\def\ga{\alpha}

\def\gd{\delta}
\def\geps{\varepsilon}
\def\gf{\varphi}
\def\gg{\gamma}

\def\gl{\lambda}

\def\gs{\sigma}

\def\gD{\Delta}
\def\Deltach{\delta}

\def\gS{\Sigma}

\def\fa{{\mathfrak a}}

\def\fg{{\mathfrak g}}
\def\fh{{\mathfrak h}}

\def\fk{{\mathfrak k}}
\def\fl{{\mathfrak l}}
\def\fm{{\mathfrak m}}
\def\fn{{\mathfrak n}}

\def\fp{{\mathfrak p}}
\def\fq{{\mathfrak q}}
\def\fs{{\mathfrak s}}
\def\ft{{\mathfrak t}}
\def\fu{{\mathfrak u}}
\def\fv{{\mathfrak v}}

\def\implies{\Rightarrow}
\def\to{\rightarrow}
\def\Re{\mathrm{Re}\,}

\def\inp#1#2{\langle#1\,,\,#2\rangle}

\def\Ad{\mathrm{Ad}}
\def\End{\mathrm{End}}

\def\Hom{\mathrm{Hom}}

\def\ad{\mathrm{ad}}
\def\after{\,{\scriptstyle\circ}\,}

\def\pr{\mathrm{pr}}

\def\tr{\mathrm{tr}\,}
\DeclareMathOperator*{\Res}{Res}

\def\iq{{\mathrm q}}
\def\iC{{\scriptscriptstyle \C}}

\def\cA{{\mathscr A}}

\def\cC{{\mathscr C}}
\def\cD{{\mathscr D}}

\def\cH{{\mathscr H}}

\def\cL{{\mathscr L}}
\def\cM{{\mathscr M}}

\def\cP{{\mathscr P}}

\def\cS{{\mathscr S}}

\def\cV{{\mathscr V}}
\def\cW{{\mathscr W}}

\def\cZ{{\mathscr Z}}

\def\col{\,:\,}
\def\faq{\fa_\fq}
\def\faqc{\fa_{\fq\iC}}
\def\faqd{\faq^*}
\def\faqdc{\fa_{\fq\iC}^*}
\def\fadc{\fa_{\iC}^*}
\def\fad{\fa^*}
\def\fah{\fa_\fh}
\def\fahd{\fa_\fh^*}

\def\oC{\cA_{M,2}}
\def\oCtau{\oC(\tau)}

\def\Aq{A_\fq}
\def\supp{\mathop{\rm supp}}

\def\Ind{\mathrm{Ind}}

\def\Cartan{\theta}

\def\Vtau{V_\tau}

\def\Aqp{A_\iq^+}

\def\GL{\mathrm{GL}}
\def\embeds{\hookrightarrow}

\def\dotvar{\, \cdot\,}

\def\bs{\backslash}
\def\bp{{}^\backprime}

\def\GammaQ{\Gamma(Q)}

\def\Ft{\mathfrak{F}}
\def\nFt{\mathcal{F}}
\def\eFt{{\mathcal F}}

\def\Ht{\mathcal{H}}
\def\It{\mathcal{I}}
\def\Kt{\mathcal{K}}
\def\Rt{\mathcal{R}}

\def\diag{\mathrm{diag}}
\def\spn{\mathrm{span}}
\def\loc{\mathrm{loc}}
\def\Cen{Z}
\def\Nor{N}
\def\Acomp{\mathfrak{A}}
\def\ch{\mathrm{conv}}
\def\1{\mathbf{1}}

\def\cusp{\mathrm{cusp}}
\def\ds{\mathrm{ds}}
\def\temp{\mathrm{temp}}
\def\mc{\mathrm{mc}}

\def\fahd{\fa_{\fh}^*}
\def\Ua{U}
\def\vH{{}^v\!H}
\def\Etau{E}
\def\ctau{C}
\def\Mzerohat{\widehat M_0}

\def\barV{{\bar V}}
\def\inj{i^\#}
\def\proj{p^\#}
\def\xiM{\xi_M}
\def\SO{{\rm SO}}
\def\faqdp{\fa_\fq^{*+}}
\def\DGH{\D(G/H)}
\def\Restau{{\rm Res}_\tau}
\def\Resone{{\rm Res}_1}
\def\LH{H_L}
\def\nicefrac#1#2{#1/#2}
\def\sign{{\rm sign}\,}
\def\Nt{{\mathcal N}}
\def\oneK{1}
\def\subseteqq{\subseteq}
\def\cPH{{\cP_\fh}}
\def\findim{\cV}
\def\res{{\rm res}}
\def\colsep{\colon}
\def\HypQtau{{\rm Hyp}(Q,\tau)}
\def\HypRQtau{{\rm Hyp}_\R(Q,\tau)}

\def\nbhood{U}

 \title{Cusp forms for reductive symmetric spaces of split rank one}
 \author{Erik~P.~van den Ban and Job~J.~Kuit\footnote{Supported by the Danish National Research Foundation through the Centre for Symmetry and Deformation (DNRF92) and the ERC Advanced Investigators Grant HARG 268105.}}
 \date{\today}

\begin{document}
 \maketitle
\begin{abstract}
For reductive symmetric spaces $G/H$ of split rank  one
 we identify a class of minimal parabolic subgroups for which certain
  cuspidal integrals of Harish-Chandra -- Schwartz functions are absolutely convergent. Using
these integrals we  introduce a notion of cusp forms  and investigate its relation with representations of the discrete series for $G/H$.
\end{abstract}
\tableofcontents
\section*{Introduction}\addcontentsline{toc}{section}{Introduction}
\setcounter{section}{9}
\renewcommand{\thesection} {\Alph{section}}
In this article we aim to develop a notion of cusp forms for reductive symmetric spaces.
More precisely, we generalize Harish-Chandra's notion of cusp forms for reductive Lie groups to a notion for  reductive symmetric spaces of split rank one.
Furthermore, we investigate the relation of this
notion with representations of the discrete series for the spaces considered.

Let $G$ be a real reductive Lie group of the Harish-Chandra class and let $\cC(G)$ be the Harish-Chandra  space of $L^2$-Schwartz functions on $G.$
In \cite{Harish-Chandra_HarmonicAnalyisOnRealReductiveGroupsI} Harish-Chandra proved that for every parabolic subgroup $P=M_{P}A_{P}N_{P}$ of $G$, every $\phi\in\cC(G)$ and every $g\in G$ the integral
\begin{equation}\label{eq int_N phi(gn) dn}
\int_{N_{P}}\phi(gn)\,dn
\end{equation}
is absolutely convergent. In analogy with the theory of automorphic forms,  he then defined a cusp form on $G$ to be a function $\phi\in\cC(G)$
 such that
 the integral (\ref{eq int_N phi(gn) dn}) vanishes for every proper parabolic subgroup
 $P$ of $G$  and every $g\in G$.
Let $\cC_{\cusp}(G)$ be the space of cusp forms and let $\cC_{\ds}(G)$ be the closed span of the $K$-finite matrix coefficients of the representations from
the discrete series. Harish-Chandra established the fundamental result that
\begin{equation}
\label{e: intro cusp is ds for group}
\cC_{\cusp}(G)=\cC_{\ds}(G).
\end{equation}
See \cite{HC_ds_2}, \cite[Thm.\ 10]{HC_harman_70} and \cite[Sect.\ 18 \& 27]{Harish-Chandra_HarmonicAnalyisOnRealReductiveGroupsI}; see also \cite[Thm.\  16.4.17]{Varadarajan_HarmonicAnalysisOnRealReductiveGroups}.
\par
For the more general class of real reductive symmetric spaces
$G/H,$ the main problem one encounters when trying to define cusp forms, is convergence of the integrals involved. The naive idea would be to use the class of
$\gs$-parabolic subgroups, as they appear in the general Plancherel theorem
as obtained by P. Delorme \cite{Delorme_Planch}  and independently,
H.\  Schlichtkrull and the first named author, \cite{vdBanSchlichtkrull_MostContinuousPart}, \cite{vdBan_Schlichtkrull_Plancherel1}. This approach fails
however for two reasons; firstly the integrals need not always converge, see
\cite[Lemma 4.1]{AndersenFlenstedJensenSchlichtkrull_CuspidalDiscreteSerieseForSemisimpleSymmetricSpaces}  and secondly,
the notion differs from Harish-Chandra's for the group.

Around 2000,  M.\  Flensted-Jensen proposed a notion of cusp forms for symmetric spaces that does generalize Harish-Chandra's notion.
This notion makes use of minimal parabolic subgroups for the group $G,$ which
are in a certain position relative to the Lie algebra $\fh$ of $H;$ in a sense they are as far away from $\gs$-parabolic subgroups as possible; in the present paper such minimal parabolic subgroups are called $\fh$-extreme,
see Definition \ref{d: extreme parabolics}.

In \cite{AndersenFlenstedJensenSchlichtkrull_CuspidalDiscreteSerieseForSemisimpleSymmetricSpaces} the new notion was tested for real hyperbolic spaces.
In that setting the space $\cC_\cusp(G/H)$ of cusp forms in the Schwartz space $\cC(G/H)$
is contained in the discrete part $\cC_\ds(G/H),$ but in contrast with the case of the group, the inclusion may be proper.
The aim of this paper is to understand such and other properties of cusp forms in the more general context of
reductive symmetric spaces of split rank one.

Our approach to the convergence problem is indirect, and heavily based on the
available tools from the harmonic analysis leading to the Plancherel formula.
In an earlier paper, \cite{vdBanKuit_EisensteinIntegrals},  we prepared for the present one by developing (without restriction on the rank) a notion of minimal Eisenstein integrals
for $G/H$  in terms of minimal parabolics of the group $G.$
For the case of the group
viewed as a symmetric space, Harish-Chandra's (minimal) Eisenstein integrals can then be recovered by making the special choice of
$\fh$-extreme minimal parabolic subgroups.

Somewhat surprisingly,  it appears that for the convergence of the cuspidal integrals another condition on the minimal parabolic subgroup involved is needed, which we call $\fh$-compatibility, see  Definition  \ref{d: fh compatible}.  The set
$\cP_\fh$ of such minimal parabolic subgroups is non-empty;
for the group, it actually consists of {\em all} minimal parabolic subgroups. For the real hyperbolic
spaces the class of $\fh$-extreme minimal parabolic subgroups turns out to coincide with the class
of $\fh$-compatible ones.

Let $G/H$ be of split rank one.
In Theorem \ref{Thm convergence}
we prove that for each $Q \in \cP_\fh$
and every Schwartz function $\phi \in \cC(G/H)$ the following
Radon transform integral,
\begin{equation}
\label{e: intro Rt Q}
\Rt_Q\phi (g):= \int_{N_Q/N_Q \cap H} \phi(gn)\; dn        \qquad (g \in G)
\end{equation}
is absolutely convergent and defines a smooth function of $g \in G;$ here, $N_Q$ denotes the nilpotent radical of $Q.$
A function $\phi \in \cC(G/H)$ is said to be a cusp form if for all $Q \in \cP_\fh$ the
Radon transform $\Rt_Q\phi$ is identically zero.
It turns out that for this to be valid, it is enough to require vanishing of $\Rt_Q\phi$
for all $\fh$-extreme parabolic subgroups in $\cP_\fh,$ see Lemma
\ref{l: minimal defi cusp form}. Thus, for both the case of the group and for the real hyperbolic spaces, our notion coincides with the existing ones.
Let $\cC_\cusp(G/H)$ denote the space of cusp forms. Under the assumption of
split rank one, we show that
\begin{equation}
\label{e: cusp G/H in ds}
\cC_\cusp(G/H) \subseteq \cC_\ds(G/H),
\end{equation}
see Theorem \ref{t: inclusion cusp in ds}.
Let $K$ be a $\gs$-stable maximal compact subgroup of $G$ and $\tau$ a finite dimensional unitary
representation of $K.$ In Theorem \ref{t: spherical cCres generated by Res} we establish that the space
$\cC_{\ds}(G/H:\tau)$ admits an $L^2$-orthogonal decomposition
\begin{equation}
\label{e: K-finite deco cusp and res}
\cC_\ds(G/H:\tau)  \ = \cC_\cusp(G/H:\tau ) \oplus \cC_{\res}(G/H : \tau),
\end{equation}
where $\cC_\res(G/H:\tau)$ is spanned by certain residues of Eisenstein
integrals defined in terms of $\fh$-compatible, $\fh$-extreme parabolic
subgroups. Furthermore, in Theorem
\ref{t: K fixed ds cusp forms implies all} we give the following remarkable criterion
for the analogue of (\ref{e: intro cusp is ds for group}) to be valid,
\begin{equation}
\label{e: criterion ds is cusp}
\cC_\res(G/H)^K = 0 \;  \implies\;  \cC_\cusp(G/H) = \cC_\ds(G/H).
\end{equation}
Finally, we establish, in Theorem \ref{t: characterisation ds by Radon}, a characterisation
of the subspace $\cC_\ds(G/H)$ of $\cC(G/H)$ in terms of the behavior of the
Radon transforms $\Rt_Q \phi,$ for $Q \in \cP_\fh.$

We will now give a more detailed outline of the structure of our paper. In the first part we
work in the generality of an arbitrary reductive symmetric space of the Harish-Chandra class.
Let $\Cartan$ be the Cartan involution associated with $K,$ and
$\fg = \fk \oplus \fp$ the associated Cartan decomposition of the Lie algebra of $G.$
Let $\fq$ be the $-1$-eigenspace of the infinitesimal involution
$\gs$ and let $\fa$ be a maximal abelian subspace
of $\fp$ such that $\faq:= \fa \cap \fq$ is maximal abelian in $\fp \cap \fq.$
Furthermore, let $A := \exp \fa$ and $\Aq : = \exp \faq.$ The (finite) set
of minimal parabolic subgroups $Q \subseteq G$ containing $A$
is denoted by $\cP(A)$ and the subset of $\fh$-compatible ones by $\cP_\fh(A).$
After necessary preparations in Section 1,
we define Radon transforms for $\phi \in L^1(G/H)$ as in (\ref{e: intro Rt Q}). By a Fubini type argument combined
with the Dixmier-Malliavin theorem on smooth vectors,
we show, in Proposition \ref{Prop int_N_Q/(N_Q cap H) phi integrable for phi smooth L^1 vector} that for $\phi \in L^1(G/H)^\infty$ the integral (\ref{e: intro Rt Q}) is absolutely convergent, and defines a smooth function on $G/N_Q.$

To make the connection with harmonic analysis,
we define, in Section \ref{s: Harish-Chandra transforms}, a
Harish-Chandra transform $\Ht_{Q}$, which maps a function $\phi\in L^1(G/H)^{\infty}$ to the smooth function on $MA:= Z_G(\fa)$
given by
$$
\Ht_{Q}\phi(l)
=\Deltach_{Q}(l)\Rt_{Q}\phi(l)\qquad  (l\in MA).
$$
Here $\Deltach_{Q}$ is a certain character on $MA$ that is chosen such that $\Ht_{Q}\phi$ is right $(MA \cap H)$-invariant and can therefore be viewed as a smooth function on $MA/(MA\cap H)$. We thus obtain a continuous linear map
\begin{equation}
\label{e: HC Q transform new}
\Ht_{Q}: L^1(G/H)^{\infty}\;  \longrightarrow\; C^{\infty}(MA/MA\cap H).
\end{equation}
It is then shown, that associated with $Q$ there exists a certain $P \in \cP(A)$ such that
$\gd_P^{-1} \Ht_Q\gf$ vanishes at infinity on $MA/(MA\cap H)$ for all $\phi \in L^1(G/H)^\infty.$
It is a consequence of this result that
$\Rt_Q$ vanishes on $L^1(G/H) \cap L^2_\ds(G/H),$ see Theorem
\ref{t: vanishing Rt Q on Lone}.

The next goal is to find a condition on the minimal parabolic subgroup
$Q$ to ensure that  (\ref{e: HC Q transform new})
extends continuously from $L^1(G/H)^\infty$  to the larger space  $\cC(G/H).$

Our strategy  is to first prove, in Section \ref{s: Schwartz functions}, that every Schwartz function can be dominated by a non-negative $K$-fixed function from $\cC(G/H),$
see Proposition \ref{Prop domination by K-inv Schwartz functions}.
Based on this, we show that for the convergence of the integral
(\ref{e: intro Rt Q}) for $\phi\in\cC(G/H)$ it suffices to prove that the restriction of $\Ht_{Q}$ to $C_{c}^{\infty}(G/H)^{K}$ extends continuously to $\cC(G/H)^{K}$, see Proposition \ref{Prop If H_Q extends to C(G/H)^K, then R_Q extends to C(G/H) with conv integrals}.

In Section \ref{s: Fourier transforms} we use the Eisenstein integrals associated with $Q\in \cP(A)$ and a finite dimensional representation $\tau$ of $K,$ introduced in \cite{vdBanKuit_EisensteinIntegrals}, to define a Fourier transform $\Ft_{Q,\tau}.$
For a compactly supported smooth $\tau$-spherical function $\phi\in C_c^\infty(G/H:\tau),$
the Fourier transform $\Ft_{Q,\tau}\phi$  is a meromorphic
function of a spectral parameter $\gl \in \faqdc.$

In Section \ref{s: tau spherical HC transform}
we introduce a $\tau$-spherical version of the transform (\ref{e: HC Q transform new}),
$$
\Ht_{Q,\tau} : C_c^\infty (G/H:\tau) \to C^\infty(\Aq) \otimes \oC(\tau).
$$
Here, $\oC(\tau)$ is a certain finite dimensional Hilbert space, which appears in the
description of the most continuous part of the Plancherel formula for $G/H,$
as a parameter space for the Eisenstein integrals involved.
The transform $\Ht_{Q,\tau}$ applied to a compactly
supported smooth $\tau$-spherical function $\phi \in C_c^\infty(G/H:\tau)$
gives a function whose Euclidean Fourier-Laplace transform coincides with
$\Ft_{Q,\tau}\phi,$ see Proposition \ref{Prop Ft_(Q,tau)=F_A circ H_Q}.
At the end of the section, we discuss the relation of the Harish-Chandra transform
with invariant differential operators on $G/H.$

Section \ref{section Extension to Harish-Chandra-Schwartz functions} is devoted
to the extension of the Harish-Chandra transform to the Schwartz space. First,
for a function $\phi \in \cC(G/H:\tau),$ the transform
$\Ht_{Q,\tau} \phi$
can be expressed as a Euclidean inverse Fourier transform of $\Ft_{Q,\tau}\phi$
which involves a contour integral over a translate of $i\faqd$ in the spectral parameter space $\faqdc,$
see Lemma \ref{l: HC transform as contour integral}. The idea is then to shift the contour integral towards the tempered part of the Plancherel spectrum, corresponding to $i\faqd,$
and to analyze the appearing residues.

At this point we restrict to spaces $G/H$ with $\dim \Aq = 1,$ in order to
be able to handle the appearing residues. The shift then results in the sum of a so-called tempered term and a so-called residual term, which essentially is  a sum of residues of the Fourier transform $\Ft_{Q,\tau}\phi.$ By its close relation with the most continuous part of the Plancherel formula, the tempered term can be shown to extend continuously to the Schwartz space. On the other hand, for $\tau$ the trivial
representation, the residual term can be shown to come from testing with
matrix coefficients of the discrete series, which arise from residues of the Eisenstein integral
for $Q.$  It is for drawing this conclusion that the condition of $\fh$-compatibility on $Q$ is needed.
Accordingly, for such a $Q,$ the transform $\Ht_{Q,1}$
extends continuously to all of $\cC(G/H)^K.$ As we indicated
above this implies that $\Rt_Q$ extends continuously to $\cC(G/H),$
see Theorem \ref{Thm convergence}.
In turn, this implies that the general $\tau$-spherical Harish-Chandra transform extends to
$\cC(G/H:\tau),$ so that the associated $\tau$-spherical residual term
must be tempered.

In Section \ref{section Cusp forms}, \S \ref{ss: kernel of I Q tau} and \S \ref{ss: residues arbitrary K types}, we apply a spectral analysis involving invariant differential
operators, to show that the $\tau$-spherical residual term consists of matrix coefficients
of discrete series representations. In the final subsection of the paper, \S \ref{ss: cusp forms}, we define the notion of cusp form as discussed above, and obtain the mentioned
results (\ref{e: cusp G/H in ds}), (\ref{e: K-finite deco cusp and res}) and
(\ref{e: criterion ds is cusp}), as well as the mentioned characterization of $\cC_\ds(G/H)$
in terms of Radon transforms.

Whenever possible, we develop the theory without restriction on
the split rank of $G/H.$ In fact, only in the subsections \ref{subsection Residues}, \ref{ss: residues trivial K type}, \ref{ss: Convergence}, \ref{ss: residues arbitrary K types} and \ref{ss: cusp forms} we require that
$\dim \Aq = 1.$ This restriction will always be mentioned explicitly.

In Remarks
\ref{r: convergence Radon in AFJS}, \ref{r: notion of cusp forms coincides with AFJS},
\ref{r: K finite cusp forms}, \ref{r: inclusion cusp in ds proper} and
\ref{r: non spherical residual part} and
\ref{r: finitely many non-cuspidal ds in AFJS}
we compare our results with the results of \cite{AndersenFlenstedJensenSchlichtkrull_CuspidalDiscreteSerieseForSemisimpleSymmetricSpaces}.
Finally, our results are consistent with the convergence of a certain integral transform
appearing in the proof of the Whittaker Plancherel formula given in \cite{Wallach_RealReductiveGroupsII}, but suggest that
the image space does not consist of Schwartz functions. This is confirmed
by an explicit calculation for ${\rm SL}(2,\R),$ see Example \ref{ex: Whittaker transform} and
Remark \ref{r: error in Wallach}.

{\bf Acknowledgements\ }
Both authors are very grateful to Mogens Flensted-Jensen and Henrik
Schlichtkrull for generously sharing their ideas in many enlightening
discussions.

\setcounter{section}{0}
\renewcommand{\thesection}{\arabic{section}}
\section{Notation and preliminaries}\label{section Notation}
Throughout the paper, $G$ will be a reductive Lie group of the Harish-Chandra class, $\sigma$ an involution of $G$ and $H$ an open subgroup of the fixed point subgroup for $\sigma$.
We assume that $H$ is essentially connected as defined in \cite[p.\ 24]{vdBan_ConvexityThm}.
The involution of the Lie algebra $\fg$ of $G$ obtained by deriving $\gs$ is denoted by the same symbol.
Accordingly, we write $\fg=\fh\oplus\fq$ for the decomposition of $\fg$ into the $+1$ and $-1$-eigenspaces for $\sigma$. Thus, $\fh$ is the Lie algebra of $H$. Here and in the rest of the paper, we adopt the convention to denote Lie groups
by Roman capitals, and their Lie algebras by the corresponding Fraktur lower cases.

 Given a subgroup $S$ of $G$ we agree to write
 $$
 H_{S}:= S\cap H.
 $$
We fix a Cartan involution $\theta$ that commutes with $\sigma$ and write $\fg=\fk\oplus\fp$ for the corresponding decomposition of $\fg$ into the $+1$ and $-1$ eigenspaces for $\theta$. Let $K$ be the fixed point subgroup of $\theta$. Then $K$ is a $\sigma$-stable maximal compact subgroup with Lie algebra $\fk$.   In addition, we fix a maximal abelian subspace $\faq$ of $\fp\cap\fq$ and a maximal abelian subspace $\fa$ of $\fp$ containing $\faq$. Then $\fa$ is $\sigma$-stable and
$$
\fa=\faq\oplus\fa_{\fh},
$$
where $\fa_{\fh}=\fa\cap \fh$.
This decomposition allows us to identify $\faqd$ and $\fahd$ with the subspaces
$(\fa/\fh)^*$ and $(\fa/\fq)^*$ of $\fad,$ respectively.

Let $A$ be the connected Lie group with Lie algebra $\fa$. We define $M$ to be the centralizer of $A$ in $K$ and write
$L$ for the group $MA.$ The  set of minimal parabolic subgroups containing $A$ is denoted by
$
\cP(A).
$

In general,  if $Q$ is a parabolic subgroup, then its nilpotent radical will be denoted by
$N_{Q}.$ Furthermore, we agree to write $\bar Q = \Cartan Q$ and $\bar N_Q = \Cartan N_Q.$
Note that if $Q\in\cP(A)$, then $L$ is a Levi subgroup of $Q$ and $Q=MAN_{Q}$ is
the Langlands decomposition of $Q$.

The root system of $\fa$ in $\fg$ is denoted by $\Sigma = \gS(\fg,\fa).$
For $Q\in\cP(A)$ we put
$$
\Sigma(Q): = \{\ga \in \gS :  \fg_\ga \subseteq     \fn_Q\}.
$$
Let $\Cen_\fg(\faq)$ denote the centralizer of $\faq$ in $\fg.$
We define the elements $\rho_Q $ and $\rho_{Q,\fh}$ of $\fad$ by
\begin{equation}
\label{e: defi rho 2x}
\rho_Q(\dotvar) = \frac12 \tr(\ad(\dotvar)|_{\fn_Q}), \qquad {\rm and}\quad
\rho_{Q,\fh}(\dotvar) = \frac12 \tr(\ad(\dotvar)|_{\fn_Q \cap \Cen_\fg(\faq)}).
\end{equation}
Let $m_\ga = \dim \fg_\ga,$ for $\ga \in \gS.$ Then it follows that
$$
\rho_{Q}
=\frac12\sum_{\alpha\in\Sigma(Q)}m_\ga \,\alpha,
\quad{\rm and}\quad
\rho_{Q,\fh}
=\frac12\sum_{\alpha\in\Sigma(Q)\cap\fa_{\fh}^{*}} m_\ga \,\alpha.
$$

For an involution $\tau$ of $\fg$ that stabilizes $\fa$ we write
$$
\Sigma(Q,\tau):= \Sigma(Q)\cap\tau\Sigma(Q).
$$
If $Q \in \cP(A)$ then $\gS(Q) \cap \fahd \subseteq \gS(Q,\gs)$ and
$\gS(Q) \cap \faqd \subseteq \gS(Q, \gs \Cartan).$ Furthermore,
$$
\gS(Q) = \gS(Q, \gs\Cartan) \sqcup \gS(Q, \gs)
$$
see \cite[Lemma 2.1]{vdBanKuit_EisensteinIntegrals}.
The following definition is consistent with \cite[Def.\ 1.1]{vdBanKuit_EisensteinIntegrals}.

\begin{Defi}
\label{d: extreme parabolics}
Let $Q \in \cP(A).$
\begin{enumerate}
\itema
The parabolic subgroup $Q$ is said to be $\fq$-extreme if $\Sigma(Q,\sigma)=\Sigma(Q)\cap\fa_{\fh}^{*}$.
\itemb
The group $Q$ is said to be $\fh$-extreme if $\Sigma(Q,\sigma\theta)=\Sigma(Q)\cap\faq^{*}$.
\end{enumerate}
\end{Defi}
\vspace{3pt}
We define the partial ordering $\preceq$ on $\cP(A)$ by
$$
Q \preceq P \iff\;\;
\Sigma(Q,\sigma\theta)\subseteq    \Sigma(P,\sigma\theta)\;\;{\rm and}\;\;
\Sigma(P,\sigma) \subseteq    \Sigma(Q,\sigma).
$$
The condition $Q\preceq P$  guarantees in particular that
$H\cap N_{P} \subseteq   H \cap N_{Q}.$ The latter implies that we have a natural surjective $H$-map $H/(H\cap N_{P}) \to H/(H \cap N_{Q})$.
\begin{Lemma}
\label{l: extremity and ordering}
Let $Q\in \cP(A).$ Then we have the following equivalences
\begin{enumerate}
\itema  $Q$ is $\fq$-extreme  $\iff\;$  $Q$ is $\preceq$-maximal;
\itemb  $Q$ is $\fh$-extreme  $\iff\;$  $Q$ is $\preceq$-minimal.
\end{enumerate}
\end{Lemma}

\begin{proof}
In both (a) and (b) the implications from left to right are obvious from the definitions.
The converse implications follow from \cite[Lemma 2.6]{vdBanKuit_EisensteinIntegrals}
and \cite[Lemma 2.6]{BalibanuVdBan_ConvexityTheorem}.
\end{proof}

We denote by $\cP_\gs(\Aq)$ the set of minimal $\Cartan\gs$-stable parabolic subgroups containing $\Aq.$
If $P_0 \in \cP_\gs(\Aq)$ then $A \subseteqq P_0$ and we write
$$
\gS(P_0):= \{\ga \in \gS: \fg_\ga \subseteqq\fn_{P_0}\}\;\;{\rm and}\;\;  \gS(P_0, \faq):= \{\ga|_{\faq} : \ga \in \gS(P_0)\}.
$$
Then $P_0 \mapsto \gS(P_0,\faq)$ is a bijection from $\cP_{\sigma}(A_{\fq})$
onto the collection of positive systems for the root system $\gS(\fg, \faq)$ of $\faq$ in $\fg.$

From \cite[Lemma 1.2]{vdBanKuit_EisensteinIntegrals} we recall that a
parabolic subgroup $P \in \cP(A)$ is $\fq$-extreme if and only if it is contained in a parabolic subgroup
$P_0 \in \cP_\gs(\Aq).$ Furthermore, in that case we must have
$$
\gS(P_0) = \gS(P, \gs\Cartan),
$$
showing that $P_0$ is uniquely determined.  In accordance with this observation, we agree to write
$$
\cP_\gs(A) =
\{P \in \cP(A) : \;\;P \;{\rm is} \;\; \fq\mbox{\rm{-extreme}}\;\}.
$$
We note that the assignment $P \mapsto P_0$ mentioned above defines a surjective map
\begin{equation}
\label{e: relation cP gs}
\cP_\gs(A) \twoheadrightarrow \cP_\gs(\Aq)
\end{equation}
For a given $P_0 \in \cP_\gs(\Aq),$ the fiber of $P_0$ for the map (\ref{e: relation cP gs})
consists of the parabolic subgroups $P \in \cP_{\gs}(A)$ with $\gS(P) = \gS(P_0) \cup (\gS(P) \cap \fahd).$ It is readily seen that the map
$P \mapsto \gS(P)\cap \fahd$ defines a bijection from this fiber onto the set of positive systems for the
root system $\gS \cap \fahd.$
\begin{Rem}
\label{r: PH open}
If $\ga \in \gS \cap \fahd,$ then
the associated root space $\fg_\ga$ is contained in $\fh,$ see
\cite[Lemma 4.1]{vdBanKuit_EisensteinIntegrals}. Hence, if $P \in \cP_\gs(A)$  and $P_0 $ the unique
group in $\cP_\gs(\Aq)$ containing $P,$ then $N_P H = N_{P_0}H,$ and $PH = P_0H.$ In particular,
it follows that $PH$ is an open subset of $G.$
\end{Rem}
For $Q \in \cP(A)$ we define
\begin{equation}
\label{e: defi P gs A Q}
\cP_\gs(A,Q): = \{ P \in \cP_\gs(A) :  P \succeq Q\}.
\end{equation}
It follows from Lemma \ref{l: extremity and ordering} (a) that this set is non-empty.
The following lemma will be used frequently.

\begin{Lemma}
\label{l: Q  P zero and P}
Let $Q \in \cP(A)$ and $P_0 \in \cP_\gs(\Aq).$
Then the following assertions are equivalent.
\begin{enumerate}
\itema
There exists a $P \in \cP_{\gs}(A)$ such that $Q \preceq P \subseteq P_0.$
\itemb
$\gS(Q, \gs\Cartan) \subseteqq \gS(P_0).$
\end{enumerate}
If (b) is valid, then the group $P$ in (a) is uniquely determined.
\end{Lemma}

\begin{proof}
First  assume (a). Let $P_0 $ be the unique parabolic subgroup
from $\cP_\gs(\Aq)$ containing $P.$ Then
$\gS(Q,\gs \Cartan) \subseteqq \gS(P, \gs\Cartan)\subseteqq\gS(P_0).$
Hence, (b).

Now, assume (b). By the discussion above there exists
a unique $\fq$-extreme $P$ with $P\subseteq P_0$ and
$\gS(P) \cap \fahd = \gS(Q)\cap \fahd.$ For this $P,$ we have
$\gS(Q, \gs \Cartan) \subseteqq\gS(P_0) = \gS(P, \gs\Cartan).$
Furthermore, $\gS(P,\gs) = \gS(P)\cap \fahd = \gS(Q) \cap \fahd
\subseteqq\gS(Q,\gs).$ Hence, $Q \preceq P$ and we infer that
(a) is valid.
\end{proof}

We fix an $\Ad(G)$-invariant symmetric bilinear form
\begin{equation}
\label{e: defi B}
B: \fg \times \fg \to \R
\end{equation}
such that $B$ is $\Cartan$- and $\gs$-invariant,  $B$ agrees with the Killing form on $[\fg,\fg]$ and $-B(\dotvar,\theta\dotvar)$ is positive definite on $\fg$.

Haar measures on compact Lie groups and invariant measures on compact
homogeneous spaces will be normalized such that they are probability measures.
If $N$ is a simply connected nilpotent Lie subgroup of $G$ with Lie algebra $\fn$, then we will normalize the Haar measure on $N$ such that its pull-back under the exponential
map coincides with the Lebesgue measure on $\fn$ normalized according to the restriction
of the inner product $-B(\dotvar, \Cartan(\dotvar)).$
\section{Radon transforms}
\subsection{Decompositions of nilpotent groups}
Let $P\in \cP(A).$
For a given element $X\in \faq$ we define the Lie subalgebra
$$
\fn_{P,X}
:=\bigoplus_{\genfrac{}{}{0pt}{}{\alpha\in\Sigma(P)}{\alpha(X)>0}}\fg_{\alpha}
$$
and denote by $N_{P,X}$  the associated connected Lie subgroup of $G.$
The following lemma is proved in \cite[Prop.\  2.16]{BalibanuVdBan_ConvexityTheorem}.

\begin{Lemma}\label{Lemma N_(Q,X) x (N_Q cap H) to N_Q diffeo}
There exists $X\in\faq$  such that
\begin{equation}\label{eq condition on X}
\begin{cases}
    \alpha(X)\neq 0 & \text{if}\alpha\in\Sigma\setminus \fahd\\
         \alpha(X)>0 & \text{if}\alpha\in\Sigma(P,\sigma\theta).
\end{cases}
\end{equation}
For any such $X,$  the multiplication map
$$
N_{P,X}\times H_{N_P}\to N_P
\qquad
(n,n_{H})\mapsto nn_{H}
$$
is a diffeomorphism.
\end{Lemma}

The groups $N_P$ and $H_{N_P}$ are both unimodular. Hence, there exists an $N_P$-invariant measure on $N_P/H_{N_P}$.
We normalize the measure on $N_P/H_{N_P}$ such that for every $\psi\in C_{c}(N_P)$
$$
\int_{N_P}\psi(n)\,dn
=\int_{N_P/H_{N_P}}\int_{H_{N_P}}\psi(xn)\,dn\,dx.
$$
Lemma \ref{Lemma N_(Q,X) x (N_Q cap H) to N_Q diffeo} has the following corollary.

\begin{Cor} \label{Cor int_N/(N cap H)=int_N_Q,X}
Let $X \in \faq$ be as in Lemma \ref{Lemma N_(Q,X) x (N_Q cap H) to N_Q diffeo}.
Let $\phi\in L^1(N_P /H_{N_P})$.
Then
$$
\int_{N_P /H_{N_{P}}}\phi(x)\,dx
=\int_{N_{P,X}}\phi(n\big)\,dn.
$$
\end{Cor}

\begin{Lemma}
\label{l: deco NPX and NQX}
Let $P,Q \in \cP(A)$ and assume that $X \in \faq$ satisfies the conditions of
Lemma \ref{Lemma N_(Q,X) x (N_Q cap H) to N_Q diffeo}.
If $Q \preceq P,$  then  both
$N_P \cap \bar N_Q$ and $N_{Q,X}$ are contained in $N_{P,X}$ and the multiplication map
$$
(N_{P} \cap \bar N_Q)\times N_{Q,X} \to N_{P,X}
$$
is a diffeomorphism.
\end{Lemma}

\begin{proof}
Since $\gS(P, \gs) \subseteq \gS(Q,\gs),$
it follows that $\gS(P) \cap \gS(\bar Q)
\subseteq \gS(P,\gs \Cartan)$ and we infer that the first inclusion follows.

Let $\ga \in \gS(Q)$ be such that $\ga(X) > 0.$ Assume $-\ga \in \gS(P).$
Then $-\ga$ is negative on $X$ hence cannot belong to $\gS(P, \gs \Cartan)$ and
must belong to $\gS(P,\gs).$ The latter set is
contained in $\gS(Q,\gs)$ hence in $\gS(Q),$ contradiction. We conclude that
$\ga \in \gS(P).$ This establishes the second  inclusion.

From the two established inclusions it follows that $N_P \cap \bar N_Q = N_{P,X} \cap \bar N_Q,$
and  $N_{Q,X} = N_{P,X} \cap N_Q$ and we see that the above
map is a diffeomorphism indeed.
\end{proof}

\begin{Cor}
\label{c: dominant radon integral}
Let $P,Q\in\cP(A)$ satisfy $Q \preceq P$ and assume that $\gf \in C(G/H)$ is integrable over $N_P/H_{N_P}.$ Then for almost
all $n \in N_{P} \cap \bar N_Q$ the function $L_{n^{-1}}\gf$ is integrable
over $N_Q/H_{N_Q}$ and
$$
\int_{N_P/H_{N_P}} \gf(x )\; d x  =
\int_{ N_P\cap \bar N_Q } \int_{N_Q/H_{N_Q}}  \gf(n y)\; dn\, d y
$$
with absolutely convergent outer integral.
\end{Cor}

\begin{proof}
Let $X \in \faq$ be as in Lemma \ref{Lemma N_(Q,X) x (N_Q cap H) to N_Q diffeo}.
Then the result follows from Corollary \ref{Cor int_N/(N cap H)=int_N_Q,X} and Lemma
\ref{l: deco NPX and NQX} combined with Fubini's theorem, in view of the normalization
of measures on the nilpotent groups involved, see the end of Section \ref{section Notation}.
\end{proof}

\subsection{Invariance of integrals}
As in the previous section, we assume that $Q \in \cP(A).$
Recall that $L=MA$. We define the character $\Deltach_{Q}$ on $L$ by
\begin{equation}
\label{e: defi Delta Q}
\Deltach_{Q}(l)
=\left|\frac{\det\Ad(l)\big|_{\fn_{Q}}}{\det\Ad(l)\big|_{\fn_{Q}\cap\Cen_{\fg}(\faq)}}\right|^{\frac12}
\qquad(l\in L)
\end{equation}
Since $M$ is compact, it follows from (\ref{e: defi rho 2x}) that
\begin{equation}
\label{e: Delta Q and rho}
\Deltach_{Q}(ma)
=a^{\rho_{Q}-\rho_{Q,\fh}}
\qquad(m\in M, a\in A).
\end{equation}

\begin{Lemma}\label{Lemma sqrt(Delta) int_N/(N cap H) is L cap H invariant}
Let $\phi$ be a measurable function on $G/H$ such that
$$
\int_{N_{Q}/H_{N_{Q}}}|\phi(n)|\,dn<\infty.
$$
Then for every $l\in H_{L}$ the function $n\mapsto \phi(ln)$ is absolutely integrable on
$N_Q/H_{N_Q}$ and
$$
\Deltach_{Q}(l)\int_{N_{Q}/H_{N_{Q}}}\phi(ln)\,dn
=\int_{N_{Q}/H_{N_{Q}}}\phi(n)\,dn.
$$
\end{Lemma}

\begin{proof}
Assume that $X\in\faq$ satisfies (\ref{eq condition on X}) and let $l\in H_{L}.$
By applying  Lemma \ref{Lemma N_(Q,X) x (N_Q cap H) to N_Q diffeo},
performing a substitution of variables and applying the same lemma once more, we obtain
the following identities of absolutely convergent integrals
\begin{eqnarray*}
\int_{N_{Q}/H_{N_{Q}}}\phi(ln)\,dn &=&
 \int_{N_{Q,X}}\phi(ln')\,dn' \\
 &=& \int_{N_{Q,X}}\phi(ln'l^{-1})\,dn'
 =  D( l)^{-1} \, \int_{N_{Q,X}}\phi(n')\,dn'\\
 &=&  D(l)^{-1}\; \int_{N_{Q}/H_{N_{Q}}}\phi(n)\,dn,
\end{eqnarray*}
\medbreak
{\ }{\ }{\ }{\ }{\ }{\ }where
$
D(l) = \left|\det\Ad(l)\big|_{\fn_{Q,X}}\right|.
$
\medbreak\noindent
Thus, it suffices to show that $D(l)$ equals $\gd_Q(l)$ as defined in (\ref{e: defi Delta Q}). Since
$H_L = (M\cap H) (A \cap H)$ and $M$ is compact, we see that $D = \gd_Q =1$  on $M\cap H$
and it suffices to prove the identity for $l =a \in A \cap H.$ Equivalently, in view of
(\ref{e: Delta Q and rho}) it suffices to
prove the identity of Lemma \ref{l: rho identity} below.
\end{proof}

\begin{Lemma}
\label{l: rho identity}
Let $X \in \faq$ be as in {\rm (\ref{eq condition on X})}. Then
\begin{equation}
\label{e: rho identity}
(\rho_{Q}-\rho_{Q,\fh})\big|_{\fah}
= \sum_{\substack{\alpha\in\Sigma(Q)\\ \ga(X) > 0}} m_{\alpha}\alpha\big|_{\fah}.
\end{equation}
\end{Lemma}

\begin{proof}
We write $\gS(Q,X)$ for the set of roots $\ga \in \gS(Q)$ with $\ga(X) > 0.$ For the purpose of the proof, it will be convenient to use the notation
$$
S(\Phi):= \sum_{\ga \in A} m_\ga \ga\big|_{\fah},
$$
for $\Phi \subseteq   \gS.$ Then the expression on the left-hand side of
(\ref{e: rho identity})  equals $\frac 12 \, S(\gS(Q)\!\setminus\! \fahd),$
whereas the expression on  the right-hand side equals $S(\gS(Q,X)).$
We observe that $\gS(Q)$ is the disjoint union of the sets
$\gS(Q,\gs)$ and $\gS(Q,\gs\Cartan).$ Furthermore,
 $S(\gS(Q,\gs \Cartan)) = 0.$ Hence,
 $$
( \rho_Q - \rho_{Q, \fh})\big |_{\fah} =  \frac12 \, S(\gS(Q,\gs)  \setminus\fahd).
 $$
 Next, we observe that $\gS(Q, \gs) \!\setminus \!\fahd$ is the disjoint union of $\gS(Q, X) \cap \gS(Q, \gs)$ and $\gs(\gS(Q, X) \cap \gS(Q, \gs))$ so that
$$
 \frac12 \, S(\gS(Q, \gs)\setminus \fahd ) = S(\gS(Q, \gs) \cap \gS(Q, X)).
$$
Finally, using that $\gS(Q,X) \supseteq \gS(Q, \gs\Cartan)$ we find
\begin{eqnarray*}
 S(\gS(Q, \gs) \cap \gS(Q, X)) &=&
 S(\gS(Q, \gs) \cap \gS(Q, X)) + S(\gS(Q, \gs\Cartan)) \\
 & = &
 S(\gS(Q, X))
 \end{eqnarray*}
 and the lemma follows.
 \end{proof}

\subsection{Convergence of integrals}
As before, we assume that $Q \in \cP(A).$
If $P \in \cP(A)$ is a $\fq$-extreme parabolic subgroup, then $PH$ is an open subset of
$G,$ see Remark \ref{r: PH open}. Its natural image in $G/H$ will be denoted by $P\cdot H.$

\begin{Lemma}
\label{l: int over P dot H}
Let $P\in\cP_{\sigma}(A,Q)$.
Then the $G$-invariant measure on $G/H$ and the $L$-invariant measure on $L/H_{L}$
can be normalized so that for every $\phi\in L^1(G/H)$
\begin{equation}\label{e: int over P dot H}
\int_{P\cdot H}\phi(x)\,dx
=\int_{N_{P}\cap \bar N_Q }\int_{L/H_{L}}\frac{\Deltach_{Q}(l)}{\Deltach_{P}(l)}\int_{N_{Q}/H_{N_{Q}}}
    \phi(\overline{n}ln)\,dn\,dl\,d\overline{n}
\end{equation}
with absolutely convergent integrals.
\end{Lemma}

Note that by Lemma
\ref{Lemma sqrt(Delta) int_N/(N cap H) is L cap H invariant}
the function
$$L\ni l\mapsto \Deltach_{Q}(l)\int_{N_{Q}/H_{N_{Q}}} \phi(ln)\,dn
$$
is right $H_{L}$-invariant if the integral is absolutely convergent for every $l \in L.$ 
Since $\Deltach_{P}$ is a right $H_{L}$-invariant function as well,
the right-hand side of (\ref{e: int over P dot H})
is well-defined.

\begin{proof}[Proof of Lemma \ref{l: int over P dot H}]
It suffices to prove the lemma for non-negative integrable functions only.
Let $\phi\in L^1(G/H)$ be non-negative. Since $P\cdot H$ is an open subset of $G/H$, the integral over $P\cdot H$ is absolutely convergent. The repeated integral on the right-hand side
of (\ref{e: int over P dot H})
is well defined (although possibly infinitely large).
To prove the lemma, we start by rewriting the right-hand side and then show that it equals the left-hand side.

Note that $L/H_{L}$ is diffeomorphic to $M/H_{M}\times A_{\fq}$ and the $L$-invariant measure on $L/H_{L}$ equals the product of the $M$-invariant measure on $M/H_{M}$ and the Haar measure on $A_{\fq}$.
Furthermore,  from (\ref{e: Delta Q and rho}) we infer that
$$
\frac{\Deltach_{Q}(ma)}{\Deltach_{P}(ma)}
=a^{\rho_{Q}-\rho_{P}}\qquad(m\in M,a\in A_{\fq}).
$$
Hence
\begin{align}
\nonumber&\int_{ N_{P}\cap \bar N_Q }\int_{L/H_{L}}\frac{\Deltach_{Q}(l)}{\Deltach_{P}(l)}\int_{N_{Q}/H_{N_{Q}}}
    \phi(\overline{n}ln)\,dn\,dl\,d\overline{n}\\
\label{eq int_(N_P0 cap theta N_Q) int_M int_A_q int_(N_Q/N_Q cap H)}
    &\qquad=\int_{N_{P}\cap \bar N_Q }\int_{M}\int_{A_{\fq}}
    a^{\rho_{Q}-\rho_{P}}\int_{N_{Q}/H_{N_{Q}}}\phi(\overline{n}man)
    \,dn\,da\,dm\,d\overline{n}.
\end{align}
Here we have used that $H_{M}$ is compact and has volume equal to
$1$ by our chosen  normalization of the Haar measure.

Let $P_{0}$ be the unique minimal $\sigma\theta$-stable parabolic subgroup such that $P\subseteq    P_{0}.$ Then the set of roots of $\fa$ in $\fn_{P_0}$ is given by $\gS(P_0) = \gS(P) \setminus \fahd $
and
$$
P_0 = Z(\faq) N_{P_0}.
$$
It follows that $\rho_P - \rho_{P_0}$ vanishes on $\faq.$

Let  $X\in\faq$ be such that $\alpha(X)>0$ for every $\alpha\in\Sigma(P_{0})$. Then $X$ satisfies (\ref{eq condition on X}) and it is readily seen that $N_{Q,X}=N_{Q}\cap N_{P_{0}}$. By Corollary \ref{Cor int_N/(N cap H)=int_N_Q,X} the integral over $N_{Q}/H_{N_{Q}}$ can be replaced by an integral over $N_{Q}\cap N_{P_{0}}$.
Therefore, (\ref{eq int_(N_P0 cap theta N_Q) int_M int_A_q int_(N_Q/N_Q cap H)}) can be rewritten as
\begin{equation}
\label{e: repeated integral for Q P}
\int_{\bar N_{Q}\cap N_{P}}\int_{M}\int_{A_{\fq}}
    a^{\rho_{Q}-\rho_{P_{0}}}\int_{N_{Q}\cap N_{P_{0}}}
    \phi(\overline{n}man)\,dn\,da\,dm\,d\overline{n}.
\end{equation}
Note that $(\bar N_{Q}\cap N_{P})=(\bar N_{Q}\cap N_{P_{0}})$.
The multiplication map
$$
(\bar N_{Q}\cap N_{P_{0}})\times(N_{Q}\cap N_{P_{0}})\to N_{P_{0}}
$$
is a diffeomorphism with Jacobian equal to $1$. We now change the order of integration
in (\ref{e: repeated integral for Q P}) and subsequently apply the change of variables
$n\mapsto (ma)n(ma)^{-1}$ to the integral over $\bar N_{Q}\cap N_{P_{0}}.$
This change of variables has Jacobian equal to $a^{\rho_{P_0} - \rho_Q}$
by the lemma below. Finally, we  rewrite the double integral over $\bar N_{Q}\cap N_{P_{0}}$ and $N_{Q}\cap N_{P_{0}}$ as an integral over $N_{P_{0}}$. We thus  infer that the integral in (\ref{e: repeated integral for Q P}) equals
\begin{equation}
\label{e: second rewrite integral P Q}
\int_{M}\int_{A_{\fq}}\int_{N_{P_{0}}}\phi(man)\,dn\,da\,dm.
\end{equation}
Note that $M_{P_{0}}\cap K\cap H$ centralizes $A_{\fq}$ and normalizes $N_{P_{0}}$. Moreover, $|\det\Ad(m)\big|_{\fn_{P_{0}}}|=1$ for all $m\in M_{P_{0}}\cap K\cap H$. Since the volume of $M_{P_{0}}\cap K\cap H$ equals $1$ (by our chosen normalization of Haar measure), it follows  that
the integral (\ref{e: second rewrite integral P Q}) equals
$$
\int_{M}\int_{M_{P_{0}}\cap K\cap H}\int_{A_{\fq}}\int_{N_{P_{0}}}\phi(mm'an)\,dn\,da\,dm'\,dm.
$$
It follows from \cite[Lemma 4.3]{vdBanKuit_EisensteinIntegrals} that $M(M_{P_{0}}\cap H)=M_{P_{0}}$. Therefore the integrals over $M$ and $M_{P_{0}}\cap K\cap H$ can be replaced by one integral over $M_{P_{0}}\cap K$. To conclude the proof, we note that
$$
\int_{M_{P_{0}}\cap K}\int_{A_{\fq}}\int_{N_{P_{0}}}\phi(mm'an)\,dn\,da\,dm
=c\int_{P_{0}\cdot H}\phi(x)\,dx
$$
for some constant $c>0$ by \cite[Thm.~1.2]{Olafsson_FourierAndPoissonTransformationAssociatedToASemisimpleSymmetricSpace},
and observe that $P_{0}\cdot H=P\cdot H$, see Remark \ref{r: PH open}.
\end{proof}

\begin{Lemma}
\label{l: Jacobian of int over Q P}
Let $P_{0}\in\cP_{\sigma}(A_{\fq})$ satisfy $\Sigma(Q,\sigma\theta)\subseteq \Sigma(P_{0})$. Then
\begin{equation}
\label{e: det Ad and rho Q P}
\left| \det \Ad(ma)\big|_{\bar \fn_Q \cap \fn_{P_0}}\right|  = a^{\rho_{P_0} - \rho_Q}, \qquad (m \in M, a \in \Aq).
\end{equation}
\end{Lemma}

\begin{proof}
Given a subset $\Phi \subseteq   \gS$ we agree to write
$$
T(\Phi) = \sum_{\ga \in \Phi} m_\ga \ga\big|_{\faq}.
$$
Then, the expression on the left-hand side
of (\ref{e: det Ad and rho Q P})  equals
$$
a^{T(\Cartan\gS(Q) \cap \gS(P_0))}
$$
Since $\gS(Q) $ is the disjoint union of $\gS(Q,\gs)$ and $\gS(Q, \gs \Cartan),$ whereas
the latter set is contained in $\gS(P_0),$ it follows that  the expression on the left-hand side
of (\ref{e: det Ad and rho Q P})  equals
\begin{equation}
\label{e: expression with T Q P}
a^{T(\Cartan\gS(Q,\gs) \cap \gS(P_0))}.
\end{equation}
On the other hand, the expression on the right-hand side of (\ref{e: det Ad and rho Q P}) equals
$$
a^{\frac12 [T(\gS(P_0))  - T(\gS(Q)]}
$$
Now $T(\gS(Q,\gs)) = 0$ and since $\gS(Q)$ is the disjoint union of
$\gS(Q,\gs)$ and $\gS(Q, \gs \Cartan),$  whereas the latter set is contained
in $\gS(P_0),$ it follows that
\begin{eqnarray*}
T(\gS(P_0))  - T(\gS(Q)) & = &  T(\gS(P_0) \setminus \gS(Q, \gs \Cartan))\\
&=& T(\gS(P_0) \cap \gS(Q, \gs)) + T(\gS(P_0) \cap \Cartan \gS(Q,\gs))\\
&=&
T(\Cartan\gs \gS(P_0) \cap \gs \gS(Q, \gs)) + T(\gS(P_0) \cap \Cartan \gS(Q,\gs))\\
&=& 2 T(\gS(P_0) \cap \Cartan \gS(Q,\gs)).
\end{eqnarray*}
Combining these we find that the expression on the right-hand side of (\ref{e: det Ad and rho Q P}) equals
(\ref{e: expression with T Q P}) as well.
\end{proof}

\begin{Lemma}\label{l: int Radon over K L HL}
Let $P\in\cP_{\sigma}(A,Q)$.
There exists a constant $c>0$ such that for every $\phi\in L^1(G/H)$
$$
\int_{K}\int_{L/H_{L}}\frac{\Deltach_{Q}(l)}{\Deltach_{P}(l)}\int_{N_{Q}/H_{N_{Q}}}|\phi(kln)|\,dn\,dl\,dk
\leq c\|\phi\|_{L^1}.
$$
\end{Lemma}

\begin{proof}
Applying Lemma \ref{l: int over P dot H}
to left $K$-translates of $\phi$ we find
\begin{align*}
\|\phi\|_{L^1}
&=\int_{K}\int_{G/H}|\phi(k\cdot x)|\,dx\,dk
\geq\int_{K}\int_{P\cdot H}|\phi(k\cdot x)|\,dx\,dk\\
&=\int_{K}\int_{ \bar N_{Q}\cap N_{P}}\int_{L/H_{L}}\int_{N_{Q}/H_{N_{Q}}}
    \frac{\Deltach_{Q}(l)}{\Deltach_{P}(l)}|\;\phi(k\overline{n}ln)|\,\,dn\,dl\,d\overline{n}\,dk.
\end{align*}
For $g\in G$ we write $k_{Q}(g)$, $a_{Q}(g)$ and $n_{Q}(g)$
for the elements of $K$, $A$ and $N_{Q}$ respectively such that the Iwasawa decomposition of $g$ is given by $g=k_{Q}(g)a_{Q}(g)n_{Q}(g)$. Let $C$ be a compact subset of $\theta N_{Q}\cap N_{P}$ with open interior. By a change of the integration variables from $N_{Q}/H_{N_{Q}}$, $L/H_{L}$ and $K$ we obtain
\begin{align*}
\|\phi\|_{L^1}
&\geq \int_{K}\int_{C}\int_{L/H_{L}}\int_{N_{Q}/H_{N_{Q}}}
    \frac{\Deltach_{Q}(l)}{\Deltach_{P}(l)}
        |\phi(kk_{Q}(\overline{n})a_{Q}(\overline{n})n_{Q}(\overline{n})ln)|\,dn\,dl\,d\overline{n}\,dk\\
&=\int_{C}
    \frac{\Deltach_{P}(a_{Q}(\overline{n}))}{\Deltach_{Q}(a_{Q}(\overline{n}))} \,d\overline{n}
    \int_{K}\int_{L/H_{L}}\int_{N_{Q}/H_{N_{Q}}}\frac{\Deltach_{Q}(l)}{\Deltach_{P}(l)}
        |\phi(kln)|\,dn\,dl\,dk.
\end{align*}
Note that the first integral on the right-hand side is finite since $C$ is compact. Moreover, it is strictly positive since $C$ has an open interior and the integrand is strictly positive.
This proves the lemma.
\end{proof}

We denote by $L^1_\loc(G/N_Q)$ the space of locally integrable functions on $G/N_Q.$
Let $dx$ be a choice of invariant measure on $G/N_Q.$ Then for each compact subset $C \subseteq G/N_Q$ the function
$$
\nu_C: \phi \mapsto \int_{C} |\phi(x)|\; dx
$$
defines a continuous seminorm
on $L^1_\loc(G/H).$ The seminorms thus defined determine a Fr\'echet topology on $L^1_\loc(G/H).$
It is readily seen that the left regular representation of $G$ in $L^1_\loc(G/H)$ is continuous
for this topology.

\begin{Cor}
\label{c: intro bp R Q}
Let $Q \in \cP(A).$ Then for every $\phi \in L^1(G/H)$ the integral
\begin{equation}
\label{e: bp integral phi over NQ mod H}
\bp \Rt_Q\phi(g):= \int_{N_{Q}/H_{N_{Q}}}\phi(gn)\,dn
\end{equation}
converges for $g$ in a right $N_Q$-invariant measurable subset of $G$ whose complement
is of measure zero. The defined function $\bp \Rt_Q\phi(g)$ is locally integrable on $G /N_Q.$
Finally, the resulting map
$$
\bp \Rt_Q: L^1(G/H) \to L^1_\loc(G/N_Q)
$$
is continuous linear and $G$-equivariant.
\end{Cor}

\begin{proof}
Let $\phi \in L^1(G/H).$
By the Iwasawa decomposition, $KA$ is a closed submanifold of $G.$
Since the multiplication map $K \times L \to KA$ defines a fiber bundle with fiber $M,$
it follows from Lemma
\ref{l: int Radon over K L HL}
 combined with Fubini's theorem that there exists a subset
$\Omega \subseteq   KA,$ whose complement has Lebesgue measure zero, such that the integral
(\ref{e: bp integral phi over NQ mod H}) converges for all $x \in \Omega.$ Furthermore,
the resulting function $\bp \Rt_Q\phi$
is locally integrable on $KA.$ By invariance of the measure on $N_Q/H_{N_Q}$ it follows that
(\ref{e: bp integral phi over NQ mod H})  converges for $g \in \Omega N_Q.$ By the Iwasawa
decomposition the set $\Omega N_Q$ has a complement of measure $0.$
We infer that the resulting function $\bp \Rt_Q\phi$ is defined almost everywhere
and locally integrable on $G/N_Q.$ By application of Fubini's theorem, it follows from
the estimate in Lemma \ref{l: int Radon over K L HL}
that the map $\bp \Rt_Q: L^1(G/H) \to L^1_\loc(G/N_Q)$ thus defined
is continuous linear.  Its $G$-equivariance is obvious from the definition.
\end{proof}

We write $L^1(G/H)^{\infty}$ for the space of smooth vectors for the left-regular representation $L$ of $G$  in  $L^1(G/H)$. If $\phi \in C^\infty(G/H)$ and
$L_u \phi \in L^1(G)$ for all $u \in \Ua(\fg),$ then $\phi \in L^1(G)^\infty;$ this follows
by a straightforward application of Taylor's theorem with remainder term, see also  \cite[Thm.~5.1]{Poulsen_OnSmoothVectorsAndIntertwiningBilinearForms}.

Conversely, any function in $L^1(G/H)^\infty$ can be represented by a smooth function
$\phi \in C^\infty(G).$ This follows from the analogous
local statement in $\R^n$ by using a partition of unity.
We may thus identify $L^1(G/H)^\infty$ with the space of $\phi \in C^\infty(G/H)$ such that
$L_u \phi \in L^1(G/H),$ for all $u \in \Ua(\fg).$

Likewise, we write $L^1_\loc(G/N_Q)^\infty$ for the Fr\'echet space of smooth vectors
in the $G$-space $L^1_\loc(G/N_Q).$ By similar remarks as those made above it follows that the inclusion map $C^\infty(G/N_Q) \to L^1_\loc(G/N_Q)$ induces
a topological linear isomorphism
\begin{equation}
\label{e: iso smooth Lone loc}
C^\infty(G/N_Q)\;\;\; { \buildrel\simeq\over \longrightarrow} \;\;\; L^1_\loc(G/N_Q)^\infty.
\end{equation}

By equivariance, it follows from Corollary \ref{c: intro bp R Q} that the map $\bp \Rt_Q$ restricts
to a continuous linear map $L^1(G/H)^\infty \to L^1_\loc(G/N_Q)^\infty.$ The following proposition
asserts that  the integral transform $\bp \Rt_Q$ actually sends the smooth representatives for functions in the first of these spaces to smooth representatives of functions in the second.

\begin{Prop}
\label{p: smoothness N integral}
\label{Prop int_N_Q/(N_Q cap H) phi integrable for phi smooth L^1 vector}
Let $Q \in \cP(A)$ and $\phi\in L^1(G/H)^{\infty}$. Then for every $g \in G$ the integral
\begin{equation}
\label{e: integral phi over NQ mod H}
\Rt_Q\phi(g):= \int_{N_{Q}/H_{N_{Q}}}\phi(gn)\,dn
\end{equation}
is absolutely convergent and the displayed integral defines a smooth
function of $g \in G.$  The indicated transform defines a continuous linear $G$-equivariant map
$$
\Rt_Q: L^1(G/H)^\infty \to C^\infty(G/N_Q).
$$
\end{Prop}

\begin{proof}
By \cite[Thm.~3.3]{DixmierMalliavin_FactorisationsDeFunctionsEtDeVecteursIndefinimentDifferentiables} the space $L^1(G/H)^{\infty}$ is spanned by functions of the form
$\phi = \chi*\psi$ with $\chi\in C_{c}^{\infty}(G)$ and $\psi\in L^1(G/H)$. Therefore, it suffices to prove the proposition for such functions.
Let $\chi\in C_{c}^{\infty}(G)$ and
$\psi \in L^1(G/H),$ and put $\phi = \chi * \psi.$

It follows from Corollary \ref{c: intro bp R Q}
that the integral for $\chi *  \Rt_Q \psi(g),$ given by
$$
J(g) := \int_G \chi(g \gamma)\int_{N_{Q}/H_{N_{Q}}}\psi(\gamma^{-1} n)\,dn\,d\gamma
$$
is absolutely convergent for every $g \in G$ and the defined function $J: G \to \C$ is smooth.
By a change of variables, followed by application of Fubini's theorem the integral may be rewritten as
\begin{eqnarray*}
J(g) & = & \int_{G}\chi(\gamma)\int_{N_{Q}/H_{N_{Q}}}\psi(\gamma^{-1}gn)\,dn\,d\gamma\\
&=&
\int_{N_{Q}/H_{N_{Q}}}(\chi*\psi)(gn)\,dn \\
&= & \int_{N_Q/H_{N_Q}} \gf(gn)\; dn.
\end{eqnarray*}
All assertions but the last now follow. By equivariance, it follows that the map $\bp \Rt_{Q}$
defined in Corollary \ref{c: intro bp R Q} restricts to a continuous linear map
$L^1(G/H)^\infty \to L^1_\loc(G/N_Q)^\infty.$ For $\phi \in L^1(G/H)^\infty,$
the function $\bp \Rt_Q(\phi)$ is represented by the smooth function $\Rt_Q(\phi).$
The last assertion now follows from the fact that (\ref{e: iso smooth Lone loc}) is a topological linear isomorphism.
\end{proof}

\begin{Defi}\label{Defi Radon and HC-transform}
The Radon transform $\Rt_{Q}$ is defined to be the $G$-equivariant continuous linear map
$$
\Rt_Q: \;\; L^1(G/H)^{\infty}\to C^{\infty}(G/N_{Q})
$$
given by (\ref{e: integral phi over NQ mod H}).
\end{Defi}

\section{Harish-Chandra transforms}
\label{s: Harish-Chandra transforms}
We retain the assumption that $Q \in \cP(A).$
In terms of the Radon transform $\Rt_Q,$ defined in the previous section, we define
a new transform as follows.

\begin{Defi}
The Harish-Chandra transform $\Ht_{Q}$ is defined to be the
continuous linear map
$$
L^1(G/H)^{\infty}\to C^{\infty}(L/H_{L})
$$
given by
$$
\Ht_{Q}\phi(l)
=\Deltach_{Q}(l) \Rt_{Q}\phi(l)
\qquad(l\in L).
$$
\end{Defi}

\begin{Ex}[\bf Group case]\label{Ex Group case - HC-transform def compared to HC's def}
Let $\bp G$ be a reductive Lie group of the Harish-Chandra class. Then $\bp G$ is diffeomorphic to $G/H$, where $G=\bp G\times\bp G$ and $H=\diag(\bp G)$, via the map
$$
G/H\to\bp G;\qquad (g_1,g_{2})\mapsto g_1g_{2}^{-1}.
$$
Under this map, the action of $G$ on $G/H$
corresponds
 to the left  times right action of
 $\bp G \times \bp G$ on $\bp G.$ As $H$ is the fixed-point group of the involution
 $\gs: G \to G,$ $ (\bp x, \bp y) \mapsto (\bp y, \bp x),$ the pair $(G,H)$  is symmetric.
 Let $\fg=\fk\oplus\fp$
 be a $\sigma$-stable Cartan decomposition of $\fg$. Then $\fk=\bp\fk\times\bp\fk$ and $\fp=\bp\fp\times\bp\fp$, where $\bp\fg=\bp\fk\oplus\bp\fp$ is a Cartan decomposition of $\bp\fg$. Let $\bp\fa$ be a maximal abelian subspace of $\bp\fp$ and let $\fa=\bp\fa\times\bp\fa$. Then
 $$
 \faq :=\fa\cap\fq = \{(H,-H) :  H \in \bp\fa \}
 $$
 is a maximal abelian subspace of $\fp\cap\fq$. Every minimal parabolic subgroup of $G$ is of the form $\bp P\times\bp Q$, where $\bp P$ and $\bp Q$ are minimal parabolic subgroups of $\bp G$. Let $\bp A=\exp(\bp\fa)$ and let $A=\bp A\times\bp A$. Let
 $\bp L = \Cen_G(\bp A)$  and let $L=\bp L\times\bp L$.

Every $\fh$-extreme parabolic subgroup
is of the form $\bp P\times\bp P$ where $\bp P$ is a minimal parabolic subgroup of $\bp G$. Let $\bp P$ be a minimal parabolic subgroup containing $\bp A$. Under the identifications $G/H\simeq \bp G$ and $L/H_{L}\simeq \bp L$ the transform $\Ht_{\bp \!P}:=\Ht_{\bp \!P\times\bp \!P}$ is given by
$$
\Ht_{{\bp \!P}}\phi(ma)
=a^{\rho_{\,\bp \!P}}\int_{N_{\bp \!P}}\phi(man)\,dn\qquad(\phi\in L^1(\bp G)^{\infty}, m\in \bp M, a\in \bp A).
$$
This shows that $\Ht_{\bp \!P}$ equals the map
$\phi\mapsto \phi^{(\bp \!P)}$, defined by Harish-Chandra
in \cite[p.\ 145]{Harish-Chandra_HarmonicAnalyisOnRealReductiveGroupsI}.

Similarly, under the described  identifications the Radon transform
$\Rt_{\bp P}:=\Rt_{\bp P\times\bp P}$ is given by
\begin{align*}
\Rt_{\bp P}\phi(g_1,g_{2})
=\int_{N_{\bp P}}\phi(g_1ng_{2}^{-1})\,dn\qquad(\phi\in L^1(\bp G)^{\infty}, g_1,g_{2}\in G).
\end{align*}
The function $\Rt_{\bp P}\phi(\dotvar, e)$ is equal to $\phi^{\bp P},$ defined by Harish-Chandra in \cite[p.\ 145]{Harish-Chandra_HarmonicAnalyisOnRealReductiveGroupsI}.
\end{Ex}

In the remainder of this section we investigate some of the properties of the Harish-Chandra transform. We start with a lemma.

\begin{Prop}\label{Prop sqrt(1/Delt_P)H_Q phi vanishing at infinity for phi smooth L^1 vector}
Let $P\in\cP_{\sigma}(A,Q)$ and $\phi\in L^1(G/H)^{\infty}$. Then
\begin{equation}
\label{e: Delta P inv times H Q}
l\mapsto {\Deltach_{P}(l)}^{-1} {\Ht_{Q}\phi(l)},\qquad L/H_{L}\to\C
\end{equation}
defines a function in $L^1(L/H_{L})^{\infty}.$
\end{Prop}

\begin{Rem}
In particular, the function (\ref{e: Delta P inv times H Q}) vanishes at infinity by Lemma
\ref{l: vanishing at infinity} below.
\end{Rem}

\begin{proof}
We will first prove that $\Deltach_{P}^{-1}\Ht_{Q}\phi$  is integrable.
From \cite[Thm.~3.3]{DixmierMalliavin_FactorisationsDeFunctionsEtDeVecteursIndefinimentDifferentiables} it follows that $L^1(G/H)^{\infty}$ is spanned by functions of the form
$\phi = \chi * \psi$ with $\chi\in C_{c}^{\infty}(G)$ and $\psi\in L^1(G/H)$.  Hence,
we may assume that $\phi$ is of this form.
It follows from
Lemma \ref{l: int Radon over K L HL} and Fubini's theorem that the integral
$$
\int_{L/H_{L}}\frac{\Deltach_{Q}(l)}{\Deltach_{P}(l)}\int_{N_{Q}/H_{N_{Q}}}
        \psi(\gamma^{-1}ln)\,dn\,dl
$$
is absolutely convergent for almost every $\gamma\in G$, and that the almost everywhere defined function on $G$ thus obtained is locally integrable. Therefore, the integral
$$
\int_{G}\chi(\gamma)\int_{L/H_{L}}\frac{\Deltach_{Q}(l)}{\Deltach_{P}(l)}\int_{N_{Q}/H_{N_{Q}}}
        \psi(\gamma^{-1}ln)\,dn\,dl\,d\gamma
$$
is absolutely convergent, and by Fubini's theorem it is equal to
$$
\int_{L/H_{L}}\frac{\Deltach_{Q}(l)}{\Deltach_{P}(l)}\int_{N_{Q}/H_{N_{Q}}}
        (\chi*\psi)(ln)\,dn\,dl
=\int_{L/H_{L}}\frac{\Ht_{Q}(\phi)(l)}{\Deltach_{P}(l)}\,dl.
$$
This proves the integrability of $\Deltach_{P}^{-1}\Ht_{Q}\phi$.

We move on to show that $\Deltach_{P}^{-1}\Ht_{Q}\phi\in L^1(L/H_{L})^{\infty}.$
In view of the remarks above (\ref{e: iso smooth Lone loc})
 it suffices to prove that $u(\Deltach_{P}^{-1}\Ht_{Q}\phi)$ is integrable for each $u\in \Ua(\fl)$.

For conciseness we write
$h: =  \Deltach_{Q}^{-1}\Ht_{Q}\phi = \Rt_Q(\phi)|_L$.
Let $u\in \Ua(\fl)$. By the Leibniz rule there exists an $n\in\N$ and $u_{j},v_{j}\in \Ua(\fl)$ for $1 \leq j \leq n$, such that
$$
 u \Big(\Deltach_P^{-1} \Ht_Q\phi   \Big) = u \Big(\frac{\Deltach_{Q}}{\Deltach_{P}}h\Big)
=\sum_{j=1}^{n} \Big(u_{j}\frac{\Deltach_{Q}}{\Deltach_{P}}\Big)v_{j}(h).
$$
Since
$\Deltach_Q/\Deltach_P$  is a character on $L$ ,
there exist constants $c_j$ such that $u_{j}(\Deltach_{Q}/\Deltach_{P}) =c_{j} \Deltach_{Q}/\Deltach_{P}$.
Therefore,
\begin{eqnarray*}
 u \Big(\Deltach_P^{-1} \Ht_Q\phi   \Big)
& = & \sum_{j=1}^{n}c_{j}\frac{\Deltach_{Q}(l)}{\Deltach_{P}(l)}
v_j(\Rt_Q\gf)(l) \\
& = & \sum_{j=1}^{n}c_{j}\frac{\Deltach_{Q}(l)}{\Deltach_{P}(l)}
\Rt_Q (v_j\phi) (l) \\
& = & \sum_{j=1}^{n}c_{j}\frac{1}{\Deltach_{P}(l)}\Ht_{Q}(v_{j}\phi)(l)
    \qquad(l\in L).
\end{eqnarray*}
Here we note that the above interchange of $v_j$ and $\Rt_{Q}$
is justified by the final assertion of Proposition \ref{p: smoothness N integral}.
By the first part of the proof, the functions $\Deltach_{P}^{-1}\Ht_{Q}(v_{j}\phi )$ are integrable on
$L/H_{L}$.
It follows that $u(\Deltach_P^{-1}\Ht_Q \phi)$ is integrable as well.
\end{proof}

We denote by $C_0(L/H_L)$ the space of continuous functions $L/H_L \to \C$
which vanish at infinity. Equipped with the sup-norm, this is a Banach space.

\begin{Lemma}
\label{l: vanishing at infinity}
The space $ L^1(L/H_{L})^{\infty}$ is contained in $C_0(L/H_L),$
with continuous inclusion map.
\end{Lemma}

For a general symmetric space, this result is proved in \cite{KrotzSchlichtkrull_OnFunctionSpacesOnSymmetricSpaces}.
We only need it in the present more restricted setting, which is essentially Euclidean.

\begin{proof}
The multiplication map induces a diffeomorphism
$L/H_L \simeq M/H_M \times \Aq.$
Since $M/H_M$ is compact, it readily follows that
$$
L^1(L/H_L)^\infty \embeds C(M/H_M, L^1(\Aq)^\infty),
$$
with continuous inclusion map. By the Fourier inversion formula on $\Aq,$ combined with
application of the Riemann--Lebesgue lemma,
it follows that $L^1(\Aq)^\infty \subseteq   C_0(\Aq)$ continuously.
Hence, $C(M/H_M, L^1(\Aq)^\infty)$ is contained in $C(M/H_M, C_0(\Aq)),$ continuously.
Let $C_b(\Aq)$ be the Banach space of bounded continuous functions on $\Aq,$ equipped with the sup-norm. Then $C_0(\Aq)$ is a closed subspace of $C_b(\Aq).$
Likewise, $C_0(L/H_L)$ is a closed subspace of  $C_b(L/H_L),$ the Banach space of bounded continuous functions on $L/H_L.$

By compactness of $M/H_M,$ the diffeomorphism mentioned in the beginning of the proof induces a continuous linear
isomorphism
$$
\psi: C(M/H_M, C_b(\Aq))  \to C_b(L/H_L).
$$
It suffices to show that $\psi$ maps the subspace $C(M/H_M, C_0(\Aq))$ into the closed subspace
$C_0(L/H_L)$ of $C_b(L/H_L).$ This can be achieved by application of a straightforward
argument involving the compactness of $M/H_M.$
\end{proof}

We recall the definition of the continuous linear map $\bp \Rt_Q: L^1(G/H) \to L^1_\loc(G/H)$
from Corollary \ref{c: intro bp R Q} and note that by Proposition \ref{p: smoothness N integral} this map can be viewed as the continuous linear extension of the restriction of the Radon transform $\Rt_Q$ to
$C^\infty_c(G/H).$

\begin{Thm}
\label{t: vanishing Rt Q on Lone}
Let $\phi \in L^1(G/H) \cap L^2_\ds(G/H).$ Then $\bp\Rt_Q(\phi) = 0.$
\end{Thm}

\begin{proof}
By equivariance and continuity of $\bp\Rt_Q,$ see Corollary \ref{c: intro bp R Q},
 it suffices to prove this result for a $K$-finite
function $\phi.$ Thus, we may assume that  $\phi \in L^1(G/H) \cap L^2_\ds(g/H)_\vartheta,$
with $\vartheta$ a finite subset of $K^{\wedge}.$
It follows from the theory of the discrete series developed in \cite{Oshima_Matsuki_ds}
and \cite[Lemma 12.6 \& Rem.\ 12.7]{vdBan_Schlichtkrull_Plancherel1}
that the space $L^2_\ds(G/H)_\vartheta$ is finite dimensional
and consists of smooth functions. It follows that the center of $U(\fg)$ acts finitely on
$\phi.$ In view of \cite[Thm.\ 7.3]{vdBan_AsymptoticBehaviourOfMatrixCoefficientsRelatedToReductiveSymmetricSpaces}, we infer that $\phi$ is
contained in the $L^1$-Schwartz space $\cC^1(G/H)$
and therefore, so is the $(\fg, K)$-module $V$ generated by $\phi.$ In particular, $V$
is contained in $L^1(G/H)^\infty$ and we see that $\bp\Rt_Q = \Rt_Q$ on $V.$

We now observe that the assignment $T: \psi  \mapsto \Rt_Q(\psi)|_{A}$ defines a linear map $V \to C^\infty(A/(A\cap H)),$ see Proposition \ref{p: smoothness N integral}.
Since $L$ normalizes $\fn_Q,$ we infer that  $T$ factors through
a map $\bar T: V/\fn_Q V \to C^\infty(A/A\cap H).$ It is well known that $\dim V/\fn_Q V < \infty,$ see
\cite[Lemma 4.3.1]{Wallach_RealReductiveGroupsI}.
From the equivariance of $\Rt_Q$ it follows that $\bar T$ is $U(\fa)$-equivariant.
Hence, for $\psi \in V,$ the function $T(\psi)$ is of exponential polynomial type on
$\Aq \simeq A/A\cap H.$ By application of Proposition
\ref{Prop sqrt(1/Delt_P)H_Q phi vanishing at infinity for phi smooth L^1 vector}
and Lemma \ref{l: vanishing at infinity} we infer that
$\gd_P^{-1}\gd_Q T(\psi)$ is an exponential polynomial function on $\Aq$ which vanishes
at infinity.  This implies $T\psi = 0,$ hence
$T(V) = 0.$
It follows that the map $\psi \mapsto \Rt_Q \psi(e) $ is zero on the closure of $V$ in $L^1(G)^\infty,$ hence on $L_g \phi$ for every $g \in G.$ We conclude that $\Rt_Q \phi = 0.$
\end{proof}

\begin{Rem}
For Radon transforms associated with minimal $\gs\Cartan$-stable parabolic subgroups the analogous result
for analytic vectors in $L^1(G/H)$ was obtained by a similar $\fa$-weight analysis in  \cite[Thm~4.1]{Krotz_horospherical_transform}. Let $P_0 \in \cP_\gs(\Aq)$ be such a parabolic subgroup. Then there
exists a parabolic subgroup $P \in \cP(A)$ such that $P \subseteq P_0.$ Since $N_{P_0} \simeq
N_{P}/(N_P \cap H),$ the Radon transform for $P_0$ coincides with $\Rt_P,$ and our
result implies that the restriction to analytic vectors is unnecessary.
\end{Rem}

The results in
Proposition \ref{Prop sqrt(1/Delt_P)H_Q phi vanishing at infinity for phi smooth L^1 vector} can be improved if only compactly supported smooth functions on $G/H$ are considered. We start by describing the support of the Harish-Chandra transform of a function in terms of the latter's support.
To  prepare for this, we introduce some notation.

 For each $\ga \in \gS \cap \faqd$
the root space $\fg_\ga$ is invariant under the involution $\gs\Cartan$  so that the root space
admits the decomposition
$
\fg_\ga= \fg_{\ga , +} \oplus \fg_{\ga, -}
$
into the $\pm 1$ eigenspaces for this involution. Accordingly, for any $Q \in \cP(A)$ we define the set
\begin{equation}
\label{e: defi gS Q minus}
\gS(Q)_- :=   \{\ga \in \gS(Q,\gs \Cartan) : \ga \in
\faqd \implies \fg_{\ga, -} \neq 0\}.
\end{equation}
We define the cone
\begin{equation}\label{eq def Gamma_Q}
\Gamma(Q)
=\sum_{\ga \in\Sigma(Q)_-} \R_{\geq0} \;\pr_\fq \,H_{\ga}.
\end{equation}
Here $H_{\ga}$ denotes the unique element of $\fa$
for which $\ga=B(H_{\ga},\cdot)$, see (\ref{e: defi B}).
Furthermore, $\pr_\fq: \fa \to \faq$ denotes the $B$-orthogonal projection.

\begin{Prop}\label{Prop support of H phi}
Let $C\subseteq   \faq$ be compact, convex and invariant under the action of $\Nor_{K\cap H}(\faq)$. If $\phi\in C_{c}(G/H)$ satisfies
$$
\supp(\phi)\subseteq    K\exp(C)\cdot H,
$$
then
$$
\supp(\Ht_{Q}\phi)\subseteq    M\exp(C+\GammaQ)\cdot H_{L}.
$$
\end{Prop}

This proposition generalizes \cite[Thm.~5.1]{AndersenFlenstedJensenSchlichtkrull_CuspidalDiscreteSerieseForSemisimpleSymmetricSpaces}, which deals with the special case in which $G/H$ is a real hyperbolic space and $Q$ is $\fh$-extreme. See also \cite[Sect.~4]{Kuit_SupportTheorem}, where similar results are proved for $\sigma\theta$-stable parabolic subgroups.

\begin{proof}[Proof of Proposition \ref{Prop support of H phi}]
Let $\phi\in C_{c}(G/H)$. Assume that $m\in M$ and $a\in A_{\fq}$ are such that $\Ht_{Q}\phi(ma)\neq 0$. Then
$$
maN_{Q}\cap K\exp(C)H\neq\emptyset.
$$
Let $\Acomp_{Q}$ be the map $G\to \faq$ determined by $g\in K\exp\big(\Acomp_{Q}(g)\big)(A\cap H)N_{Q}$. Then $\log(a)\in\Acomp_{Q}\big(\exp(C)H\big)$.
By \cite[Thm.~10.1]{BalibanuVdBan_ConvexityTheorem}
$$
\Acomp_{Q}\big(\exp(C)H\big)
=\bigcup_{X\in C}\ch (\Nor_{K\cap H}(\faq)\cdot X)+\GammaQ.
$$
Since $C$ is convex and $\Nor_{K\cap H}(\faq)$-invariant, it follows that the right-hand side equals $C+\GammaQ$. Therefore, $\log(a)\in C+\GammaQ$. The compactness of $C$ and the fact that $\GammaQ$ is closed imply that $C+\GammaQ$ is closed, hence $M\exp(C+\GammaQ)\cdot H_{L}$ is closed. The support of $\Ht_{Q}\phi$ equals the closure of the subset of $L/H_{L}$ on which $\Ht_{Q}\phi$ is nonzero, hence
$$
\supp(\Ht_{Q}\phi)
\subseteq    M\exp(C+\GammaQ)\cdot H_{L}.
$$
\end{proof}

For a compact subset $U$ of $G/H$, let $C_{U}^\infty(G/H)$ be the space of smooth functions on $G/H$ with support contained in $U$, equipped with the usual Fr\'echet topology.
As usual, we equip the space
$C_{c}^\infty(G/H)$ with the inductive limit topology of the family of spaces
$C_{U}^\infty(G/H)$ where $U$ runs over all compact subsets of $G/H$.

\begin{Prop}\label{Prop Delta_P_0^(-1/2)H_Q maps C_c^infty to C(L/L cap H)}
Let $P \in \cP_\gs(A,Q).$ Then  $\Deltach_P^{-1} \Ht_{Q}$ is a continuous linear map
$C_c^\infty(G/H) \to L^1(L/H_L)^\infty.$
\end{Prop}

\begin{proof}
Let $\phi\in C_{c}^{\infty}(G/H)$ and let $u\in \Ua(\fg)$. Let  $c_{j}$ and $v_{j}$ be as in the proof for Proposition \ref{Prop sqrt(1/Delt_P)H_Q phi vanishing at infinity for phi smooth L^1 vector}. Then
$$
u\left(\frac{\Ht_{Q}\phi}{\Deltach_{P}}\right)
=\sum_{j=1}^{n}c_{j}\frac{\Ht_{Q}(v_{j}\phi)}{\Deltach_{P}}.
$$
Let $U$ be a compact subset of $G/H$ such that $\supp\phi\subseteq    U$ and let $\vartheta\in C_{c}^{\infty}(G/H)$ be non-negative and equal to $1$ on an open neighborhood of
$U$. Then
\begin{equation}\label{eq estimate u(HQ phi)}
\left|u\left(\frac{\Ht_{Q}\phi}{\Deltach_{P}}\right)\right|
\leq\Big(\sum_{j=1}^{n}|c_{j}|\sup|v_{j}\phi|\Big)\frac{\Ht_{Q}\vartheta}{\Deltach_{P}}.
\end{equation}
It follows from Proposition \ref{Prop sqrt(1/Delt_P)H_Q phi vanishing at infinity for phi smooth L^1 vector} that $\Deltach_{P}^{-1}\Ht_{Q}\vartheta\in L^1(L/H_{L})^{\infty}$. From (\ref{eq estimate u(HQ phi)}) we now conclude that $\Deltach_{P}^{-1}\Ht_{Q}$ is a continuous linear
map $C_{c}^{\infty}(G/H)\to L^1(L/H_{L})^{\infty}$.
\end{proof}

We write $\Gamma(Q)^{\circ}$ for the dual cone of $\Gamma(Q)$, i.e.,
$$
\Gamma(Q)^{\circ}
=\{\lambda\in\faq^{*}:\lambda\geq 0 \textnormal{ on } \Gamma(Q)\}.
$$
Furthermore, we define
\begin{equation}\label{eq def Omega_Q}
\Omega_{Q}
:=\bigcup_{P\in\cP_{\sigma}(A,Q)}-(\rho_P-\rho_{P,\fh})-\Gamma(Q)^\circ  + i\faq^{*}.
\end{equation}
\begin{Cor}
Let $\lambda\in\Omega_{Q}$ and let
the character $\chi_\gl: L \to \R_{> 0}$   be given by
$$
\chi_{\lambda}(ma)
=a^{\lambda}\qquad(m\in M, a\in A).
$$
Then $\chi_{\lambda}\Ht_{Q}\phi\in L^1(L/H_{L})^{\infty}$ for every $\phi\in C_{c}^{\infty}(G/H)$.
Moreover, the map
$$
C_{c}^{\infty}(G/H)\to L^1(L/H_{L})^{\infty};
    \qquad
    \phi\mapsto\chi_{\lambda}\Ht_{Q}\phi
$$
is continuous.
\end{Cor}

\begin{proof}
For every $u\in\Ua(\fl)$ the function $u(\chi_{\lambda}\Deltach_{P})$ is bounded on $\GammaQ$. The
result now follows by application of the Leibniz
rule and Propositions \ref{Prop support of H phi} and \ref{Prop Delta_P_0^(-1/2)H_Q maps C_c^infty to C(L/L cap H)}.
\end{proof}

\section{Harish-Chandra -- Schwartz functions}
\label{s: Schwartz functions}
\subsection{Definitions}
In this subsection we recall some basic facts on the Harish-Chandra space of
$L^p$-Schwartz functions on $G/H$ from \cite[Sect.\ 17]{vdBan_PrincipalSeriesII}, and give a
characterization that will be useful in the next subsection.

Let $\tau:G/H\to[0,\infty[\,$ and  $\Theta:G/H\to\,]0,1]$ be defined by
$$
\tau(kaH) =  \|\log a\|,\qquad \Theta(g\cdot H)
=\sqrt{\Xi(g\sigma(g)^{-1})}.
$$
Here $\Xi$ is Harish-Chandra's bi-$K$-invariant elementary spherical function $\phi_{0}$ on $G,$
see, e.g.,  \cite[p.~329]{Varadarajan_HarmonicAnalysisOnRealReductiveGroups}.
Let $V$ be a complete locally convex Hausdorff space and let $\mathcal{N}(V)$
denote the set of continuous seminorms on $V$.
Let $1 \leq p < \infty.$
A
smooth function $\phi:G/H\to V$ is said to be $L^p$-Schwartz if
 all seminorms
$$
\mu_{u,N,\eta}(\phi):=\sup\ \Theta^{-\frac{2}{p}}(1+\tau)^{N}\eta(u\phi)\qquad
	\big(u\in\Ua(\fg),\ N\in\N, \eta\in\mathcal{N}(V)\big)$$ are finite.
The space of
such functions is denoted by $\cC^{p}(G/H,V)$.
Equipped with the topology induced by the mentioned semi-norms, $\cC^{p}(G/H,V)$ is a complete locally convex space. Furthermore, it is Fr\'echet if $V$ is Fr\'echet.

Let $\fv$ be a $\sigma$ and $\theta$-stable central subalgebra of $\fg$ such that $G={}^{\circ}G\times\exp(\fv)$, where ${}^{\circ}G=K\exp\big(\fp\cap[\fg,\fg])$.
Define the functions $\Phi_1, \Phi_2: G \to [1,\infty [\,$ by
$$
\Phi_1
:=  1+|\log\after\Theta|=1-\log\after \Theta,
$$
$$
\Phi_{2}\big(g\exp(v_{\fh}+v_{\fq})\big):=\sqrt{1+\|v_{\fq}\|^{2}}
    \qquad(g\in{}^{\circ}G, v_{\fh}\in\fv\cap\fh, v_{\fq}\in\fv\cap\fq).
$$
By \cite[Lemma 17.10]{vdBan_PrincipalSeriesII} there exists a positive constant $C$ such that
$$C^{-1}(1+\tau)\leq\Phi_1+\Phi_{2}\leq C(1+\tau).$$
Moreover, $\Phi_1$ and $\Phi_{2}$ are real analytic and for every $u\in \Ua(\fg)$ there exists a constant $c>0$ such that
\begin{equation}
\label{e: derivative Phi}
|u\Phi_{j}|\leq c\Phi_{j} \qquad (j=1,2).
\end{equation}
The following result is now straightforward.

\begin{Lemma}\label{L: Cor characterization of Schwartz functions}
Let $\phi:G/H\to V$ be smooth. Then $\phi\in\cC^{p}(G/H,V)$ if and only if all seminorms
$$\nu_{u,N,\eta}(\phi):=
	\sup\ e^{\frac{2}{p}\Phi_1}(\Phi_1+\Phi_{2})^{N}\eta(u\phi)\qquad
	\big(u\in\Ua(\fg),\ N\in\N, \eta\in\mathcal{N}(V)\big)
$$
are finite.
\end{Lemma}

We write $\cC^{p}(G/H)$ for $\cC^{p}(G/H,\C)$ and $\nu_{u,N}$ for $\nu_{u,N,|\cdot|}$. For convenience, we suppress  the super-script $p$ if $p=2$.

\subsection{Domination by $K$-fixed Schwartz functions}
We start this subsection with an important result which further on
will be applied to reduce the convergence of certain integrals to the case of $K$-finite functions.

\begin{Prop}\label{Prop domination by K-inv Schwartz functions}
There exists a map  $\cC^p(G/H)\to\cC^p(G/H)^{K};$  $\phi\mapsto \widehat{\phi}$ with the following properties.
\begin{enumerate}
\itema
$ |\phi|\leq \widehat{\phi},\;\;$ for all $\phi \in \cC^p(G/H).$
\itemb
Let $\nu$ be a continuous seminorm on $\cC^p(G/H).$ Then there exist constants $k \in \N$
and $C > 0$ such that, for all $\phi \in \cC^p(G/H),$
\begin{equation}
\label{e: estimate nu of hat phi}
\nu ( \widehat \phi )  \leq C \,\nu_{0, k} (\phi).
\end{equation}
\end{enumerate}
\end{Prop}

We first prove two lemmas.

Let $\cL(\R^{2})$ be the space of locally integrable functions $\R^{2}\to\C$ which are constant on $\R^{2}\setminus[2,\infty[{\,}^{2}$.

\begin{Lemma}\label{Lemma condions S => condition chi*S}
Let $\chi\in C_{c}^{\infty}(\R^{2})$ satisfy
$\supp(\chi)\subseteq \,\,] -1,1[{\,}^{2}$. For every $N\in\N$ there exists a constant $c_{N}>0$ such that for  all  $S\in\cL(\R^{2})$
$$
\sup_{(x,y)\in[1,\infty[{\,}^{2}}e^{\frac{2}{p}x}(x+y)^{N}\Big|\big(\chi*S\big)(x,y)\Big|
\leq c_{N}\sup_{(x,y)\in[1,\infty[{\,}^{2}}e^{\frac{2}{p}x}(x+y)^{N}|S(x,y)|.
$$
\end{Lemma}

\begin{proof}
Let $c=\int_{\R^{2}}|\chi(\xi)|\,d\xi$. Then
$$
|\chi*S(x,y)|
\leq c\sup_{(x,y)+ \,]-1,1\,[ ^{2}} |S|,
$$
hence
\begin{align*}
\sup_{(x,y)\in[1,\infty[{\,}^{2}}\Big|e^{\frac{2}{p}x}(x+y)^{N}\big(\chi*S\big)(x,y)\Big|
&\leq c \sup_{(x,y)\in[1,\infty[{\,}^{2}}
    \Big(e^{\frac{2}{p}x}(x+y)^{N}\sup_{(x,y)+\,]-1,1\,[{\,}^{2}} |S|\Big)\\
&\leq c \sup_{(u,v)\in \,]\,0,\infty\,[{\,}^{2}}
    \Big(e^{\frac{2}{p}(u+1)}(u+v+2)^{N} |S(u,v)|\Big).
\end{align*}
Since $S$ is constant on $]\,0 ,\infty\,[{\,}^{2}\setminus [2,\infty\,[{\,}^{2}$,
the supremum over $]\,0,\infty\,[{\,}^{2}$ can be replaced by a supremum over $[1,\infty\,[{\,}^2$. Using that
$$
\frac{u+v+2}{u+v}
\leq2
\qquad\big((u,v)\in[1,\infty\,[{\,}^{2}\big),
$$
we find
$$
\sup_{(x,y)\in[1,\infty\,[{\,}^{2}}\Big|e^{\frac{2}{p}x}(x+y)^{N}\big(\chi*S\big)(x,y)\Big|
\leq c\,e^{\frac{2}{p}}2^{N}\sup_{(x,y)\in[1,\infty,[{\,}^{2}}\Big(e^{\frac{2}{p}x}(x+y)^{N}|S(x,y)|\Big).
$$
This establishes the estimate.
\end{proof}

For $\phi\in\cC^{p}(G/H)$, let $S_{\phi}:\R^{2}\to\R$ be the function which for $(x,y)\in[i,i+1 [\,\times
[j,j+1[\,$ with $i,j\in\Z$ is given by
$$
S_{\phi}(x,y)
=\sup_{\Phi_1^{-1}([i,\infty\,[\,)\cap\Phi_{2}^{-1}([j,\infty\,[\,)}|\phi|.
$$
Note that $\Phi_j^{-1}([1, \infty\,[\,) = \Phi^{-1}_j(\R),$ so that $ S_{\phi}\in\cL(\R^{2})$.
For $\chi\in C_{c}^{\infty}(\R^{2})$ we define the smooth function
\begin{equation}
\label{e: defi hat phi chi}
\widehat{\phi}_{\chi}:G/H\to\C,\qquad x\mapsto \big(\chi*S_{\phi}\big)\big(\Phi_1(x),\Phi_{2}(x)\big).
\end{equation}
Since $\Phi_1$ and $\Phi_2$ are
left $K$-invariant, so is the function (\ref{e: defi hat phi chi}).

\begin{Lemma}\label{Lemma construction $K$-invariant Schwartz function}
Let  $\chi\in C_{c}^{\infty}(\R^{2})$ have support contained in $\,]-1,1\,[{\,}^{2}$.
\begin{enumerate}
\itema
If  $\phi\in\cC^p(G/H)$ then $\widehat{\phi}_{\chi}\in\cC^{p}(G/H)^{K}$.
\itemb
Let $u \in \Ua(\fg)$ be of order $n.$ Then for every $N\in\N$ there exists a  constant $c_{u,N}>0$
such that
\begin{equation}\label{e: estimate nu hat phi}
\nu_{u, N}(\widehat{\phi}_{\chi})
\leq c_{u, N}\,\nu_{0,N + n}(\phi) \qquad (\phi \in \cC^p(G/H)).
\end{equation}

\end{enumerate}
\end{Lemma}

\begin{proof}
Since the function (\ref{e: defi hat phi chi}) is smooth and left $K$-invariant, it suffices to prove (b).
Let $u\in\Ua(\fg)$ and let $n$ be the order of $u$. Then by repeated application of the Leibniz and the chain rule it follows that there exists a finite set $F \subseteq   \Ua(\fg)$ and
for every multi-index $\mu$ in two variables, with $|\mu|\leq n$, a polynomial expression
$P_\mu$ in $(v\Phi_j :  v\in F, j =1,2),$
of total degree at most $n,$ such that
$$
u\widehat{\phi}_{\chi} =
\sum_{|\mu|\leq n} P_\mu \cdot \widehat{\phi}_{\partial^{\mu}\chi} \qquad (\phi \in \cC^p(G/H)).
$$
In view of (\ref{e: derivative Phi}) this leads to the existence of a constant $C > 0$ such that
$$
|u\widehat{\phi}_{\chi}|
\leq C \, (\Phi_1+\Phi_{2})^n \, \sum_{|\mu|\leq n} |\widehat{\phi}_{\partial^{\mu}\chi}|.
$$
Therefore,
$$
\nu_{u,N}(\widehat{\phi}_{\chi})
\leq C \, \sum_{|\mu|\leq n} \nu_{0,N+ n}(\widehat{\phi}_{\partial^{\mu}\chi}).
$$
Thus, in order to prove the lemma, it suffices to prove that  for every $N\in\N$ and
$\chi\in C_{c}^{\infty}(\R^{2})$ the estimate (\ref{e: estimate nu hat phi}) holds for
$u=1$.

Let $N\in\N$. Then
\begin{align*}
\nu_{0,N}(\widehat{\phi}_{\chi})
&=\sup_{G/H}\ e^{\frac{2}{p}\Phi_1}(\Phi_1+\Phi_{2})^{N}
    \big|\big(\chi*S_{\phi}\big)\after(\Phi_1\times\Phi_{2})\big| \nonumber \\
&=\sup_{(x,y)\in[1,\infty\,[{\,}^{2}}\ e^{\frac{2}{p}x}(x+y)^{N}\big|\chi*S_{\phi}(x,y)\big|.
\end{align*}
By Lemma \ref{Lemma condions S => condition chi*S} we now infer the
existence of a constant $c_{N}>0$ such that
\begin{equation}
\label{e: estimate nu zero N}
\nu_{0,N}(\widehat{\phi}_{\chi}) \leq
c_{N}\sup_{(x,y)\in[1,\infty\,[{\,}^{2}}e^{\frac{2}{p}x}(x+y)^{N}|S_{\phi}(x,y)|.
\end{equation}
Let $(x,y)\in[1,\infty\,[{\,}^{2}$.   There exist unique $i,j\in\Z_{\geq0}$ such that $i\leq x<i+1$ and $j\leq y<j+1$. Then
\begin{align*}
&e^{\frac{2}{p}x}(x+y)^{N}|S_{\phi}(x,y)|
=e^{\frac{2}{p}x}(x+y)^{N}
    \Big(\sup_{\Phi_1^{-1}([i,\infty\,[\,)\cap\Phi_{2}^{-1}([j,\infty\,[\,)}|\phi|\Big)\\
&\qquad\leq\sup_{\Phi_1^{-1}([i,\infty\,[\,)\cap\Phi_{2}^{-1}([j,\infty\,[\,)}
    e^{\frac{2}{p}(\Phi_1+1)}(\Phi_1+\Phi_{2}+2)^{N}|\phi|\\
&\qquad\leq\sup\ e^{\frac{2}{p}(\Phi_1+1)}(\Phi_1+\Phi_{2}+2)^{N}|\phi|\\
& \qquad\leq
\sup\ e^{\frac{2}{p}(\Phi_1+1)}2^N (\Phi_1+\Phi_{2})^{N}|\phi| = 2^N e^{\frac2{p}} \nu_{0,N} (\phi).
\end{align*}
Combining this estimate with (\ref{e: estimate nu zero N}) we obtain
the estimate of (b) with $u = 1.$
\end{proof}

\begin{proof}[Proof of Proposition \ref{Prop domination by K-inv Schwartz functions}]
Let $\chi\in C_{c}^{\infty}\big(\,]\,0,1\,[{\,}^{2}\big)$ be a non-negative function such that $\int_{\R}\chi(x)\,dx=1.$ If $\phi \in \cC^p(G/H)$
then $\widehat{\phi}_{\chi}\in C^\infty(G/H)^{K}.$
 Moreover, since $S_{\phi}$ is decreasing in both variables, it follows from the condition on $\supp \chi$ that
$$
\widehat{\phi}=
\big(\chi*S_{\phi}\big)\circ(\Phi_1\times\Phi_1)
\geq S_{\phi}\circ(\Phi_1\times\Phi_1)
\geq |\phi|.
$$
This establishes (a).

In order to complete the proof, it suffices to prove (b) for $\nu = \nu_{u, N},$ with
$u \in U(\fg)$ of order at most $n$ and for $N \in \N.$ Let $k = N + n.$ Then the estimate
(\ref{e: estimate nu of hat phi}) follows by application of Lemma \ref{Lemma construction $K$-invariant Schwartz function}.
\end{proof}

For the application of Proposition \ref{Prop domination by K-inv Schwartz functions} we will need the following useful lemma.

\begin{Lemma}
\label{l: monotone approximation Schwartz}
Let $\psi \in \cC(G/H)^K$ be non-negative.
Then there exists a monotonically increasing sequence $(\psi_j)_{j\in \N}$ in $C_c^\infty(G/H)^K$
such that $\psi_j \to \psi$ in $\cC(G/H)^K,$ for  $j \to \infty.$
\end{Lemma}

\begin{proof}
For $r > 0$ we define $B(r): = \{x \in G/H : \tau(x) \leq r\}.$
By \cite[Lemma 2.2]{vdBan_finite_multiplicities_Plancherel} and its proof,
 there exists a sequence of functions $g_j \in C_c^\infty(G/H)$ such that the following conditions are
 fulfilled,
\begin{enumerate}
\item[{\rm (1)}]
$0 \leq g_j  \leq g_{j+1} \leq 1,$  for $j\geq 0;$
\item[{\rm (2)}]
$g_j = 1$ on $B(j)$ and $\supp g_j \subseteq B(j+1),$ for $j \geq 0;$
\item[{\rm (3)}]
for every $u \in U(\fg)$ there exists $C_u > 0$ such that $\sup_{G/H} |L_u g_j| \leq C_u$ for all $j \geq 1;$
\end{enumerate}
By using the argument of \cite[Thm.\ 2, p.\ 343]{Varadarajan_HarmonicAnalysisOnRealReductiveGroups} one now readily checks
that the sequence $\psi_j = g_j \psi$ satisfies our requirements.
\end{proof}

Proposition \ref{Prop domination by K-inv Schwartz functions} now
leads to the following results concerning the Radon and Harish-Chandra transforms.

\begin{Prop}\label{Prop If H_Q extends to C(G/H)^K, then R_Q extends to C(G/H) with conv integrals}
Assume that the restriction of $\Ht_{Q}$ to $C_{c}^{\infty}(G/H)^{K}$ extends to a continuous linear map
$\cC(G/H)^{K}\to C(L/H_{L}).$
Then $\Rt_{Q}$ extends to a continuous linear map
$$
\Rt_{Q}:\cC(G/H)\to C^{\infty}(G/N_{Q})
$$
and for every $\phi\in\cC(G/H)$,
\begin{equation}\label{eq Rt_Q phi=int_N/(N cap H)phi}
\Rt_{Q}\phi(g)
=\int_{N_{Q}/H_{N_{Q}}}\phi(gn)\,dn
    \qquad(g\in G)
\end{equation}
with absolutely convergent integrals.
Furthermore, the restriction of $\Ht_{Q}$ to $C_{c}^{\infty}(G/H)$ extends to a continuous linear map
$$
\Ht_{Q}:\cC(G/H)\to C^{\infty}(L/H_{L})
$$
and for  every $\phi\in\cC(G/H),$
\begin{equation}\label{eq Ht_Q phi=delta int_N/(N cap H)phi}
\Ht_{Q}\phi(l)
=\Deltach_{Q}(l)\int_{N_{Q}/H_{N_{Q}}}\phi(ln)\,dn
\qquad(l\in L)
\end{equation}
with absolutely convergent integrals.
\end{Prop}

\begin{proof}
Since
$$
\Rt_{Q}\phi(kan)
=\Deltach_{Q}(a)^{-1}\Ht_{Q}\phi(a)
\qquad\big(\phi\in C_{c}^{\infty}(G/H)^{K}, k\in K, a\in A, n\in N_{Q}\big),
$$
it follows from the assumption in the proposition that the restriction of $\Rt_{Q}$ to $C_{c}^{\infty}(G/H)^{K}$ extends to a continuous linear map $\cC(G/H)^{K}\to C(G/N_{Q})^K$.

Let $\psi \in\cC(G/H)^{K}$ be non-negative. We claim that $\Rt_{Q}\psi$ is given by (\ref{eq Rt_Q phi=int_N/(N cap H)phi}).
To see this, let $(\psi_{j})_{j\in\N}$ be a monotonically increasing sequence as in Lemma
\ref{l: monotone approximation Schwartz}.
Then for every $g\in G$, we have
$$
\Rt_{Q}{\psi}(g)
=\lim_{j\to\infty}\Rt_{Q}{\psi}_{j}(g)
=\lim_{j\to\infty}\int_{N_{Q}/H_{N_{Q}}}{\psi}_{j}(gn)\,dn.
$$
Since the sequence ${\psi}_{j}$ is monotonically increasing, the monotone convergence theorem implies that (\ref{eq Rt_Q phi=int_N/(N cap H)phi}) holds and that the integral is absolutely convergent, for every $g \in G.$

By Proposition \ref{Prop domination by K-inv Schwartz functions} every element of $\cC(G/H)$ can be dominated by an element of $\cC(G/H)^{K}$.  Hence, for every $\phi\in\cC(G/H)$ and $g\in G$ the integral in (\ref{eq Rt_Q phi=int_N/(N cap H)phi}) is absolutely convergent. For $\phi\in\cC(G/H)$ we now define $\Rt_{Q}\phi$ and $\Ht_{Q}\phi$ by (\ref{eq Rt_Q phi=int_N/(N cap H)phi}) and (\ref{eq Ht_Q phi=delta int_N/(N cap H)phi}), respectively. To finish the proof of the proposition, it suffices to show that $\Rt_{Q}\phi$ is smooth and that the map $\Rt_{Q}:\cC(G/H)\to C^{\infty}(G/N_{Q})$ is continuous.

By assumption, there exists a continuous seminorm $\nu$ on $\cC(G/H)$ such that for
all $\psi \in \cC(G/H)^K,$
$$
\sup_{G}|\Rt_Q(\psi)| \leq \nu(\psi).
$$
Let $\phi\mapsto\widehat{\phi}$ be a map $\cC(G/H) \to \cC(G/H)^K$
as in Proposition \ref{Prop domination by K-inv Schwartz functions}, with $p =2.$
Let $C > 0$ and $n \in \N$ be associated with $\nu$ as in the mentioned proposition.
Then it follows that for all $\phi \in \cC(G/H),$
$$
|\Rt_Q(\phi) | \leq \Rt_Q(|\phi|) \leq \Rt_Q(\widehat \phi) \leq \nu(\widehat \phi) \leq C \nu_{0, n}(\phi).
$$
We thus see that $\Rt_Q$ defines a continuous linear map
$$
\Rt_Q: \cC(G/H) \to C(G/N_Q).
$$
Since this map intertwines the left $G$-actions, whereas the left regular representation
of $G$ in $\cC(G/H)$ is smooth, it follows that $\Rt_Q$ maps continuously into
the space of smooth vectors of $C(G/N_Q),$ which equals $C^\infty(G/N_Q)$
as a topological linear space.
\end{proof}

\section{Fourier transforms}
\label{s: Fourier transforms}
\subsection{Densities and a Fubini theorem}
In this section we introduce some notation related to densities on homogeneous spaces. Further details can be found in \cite[App.~A]{vdBanKuit_EisensteinIntegrals}. For our purposes it is more convenient to consider right-quotients $S/T$ of Lie groups instead of left-quotients $T\bs S$, which we used in the aforementioned article.

If $V$ is a real finite dimensional vector space, then we write $\cD_{V}$ for
the space of complex-valued densities on $V$, i.e., the $1$-dimensional complex vector space of functions
$\omega:\wedge^{\mathrm{top}}(V)\to\C$ transforming according to the rule
$$
\omega(t\upsilon)
= |t|\omega(\upsilon)\qquad(t\in\R,\upsilon\in\wedge^{\mathrm{top}}V).
$$

For a Lie group $S$ and a closed subgroup $T$, let $\Delta_{\scriptscriptstyle S/T}: T \to \R_+$ be the positive character given by
$$
\Delta_{\scriptscriptstyle S/T}(t) = |\det \Ad_S(t)_{\fs/\ft}|^{-1}
\qquad(t\in T),
$$
where $\Ad_{S}(t)_{\fs/\ft} \in \GL(\fs/\ft)$ denotes the map induced by the adjoint map $\Ad_{G}(t) \in \GL(\fs).$
We denote by $C(S:T:\Delta_{\scriptscriptstyle S/T})$ the space of continuous functions $f: S \to \C$ transforming according to the rule
$$
f(st)
=\Delta_{\scriptscriptstyle S/T}(t)^{-1} f(s)\qquad(s\in S, t\in T).
$$
We denote by $\cM(G:L:\xi)$ the space of measurable functions $f:G \to \C$ transforming according to the same rule.

Given $f \in C(S:T:\Delta_{\scriptscriptstyle S/T})$ and
$\omega \in \cD_{\fs/\ft},$ we denote by $f_{\omega}$ the continuous density on $S/T$ determined by
by
$$
f_\omega(s)
=f(s)\, dl_{s}(e)^{-1*}\omega
\qquad(s\in S).
$$

We fix non-zero elements $\omega_{\scriptscriptstyle S/U} \in \cD_{\fs/\fu},$ $\omega_{\scriptscriptstyle T/U} \in \cD_{\ft/\fu}$
and $\omega_{\scriptscriptstyle S/T} \in \cD_{\fs/\ft}$ such that
$$
\omega_{\scriptscriptstyle T/U} \otimes \omega_{\scriptscriptstyle S/T}= \omega_{\scriptscriptstyle S/U}
$$
with respect to the identification determined by the short
exact sequence
$$
0 \to \ft/\fu \to \fs/\fu \to \fs/\ft \to 0.
$$
See  Equation  (A.10) and the subsequent text in  \cite[App.\ A]{vdBanKuit_EisensteinIntegrals}.
We then have the following Fubini theorem \cite[Thm.\  A.8]{vdBanKuit_EisensteinIntegrals}.

\begin{Thm}\label{Thm Fubini theorem for densities}
Let $\phi \in \cM(S:U:\Delta_{\scriptscriptstyle S/U})$ and let $\phi_{\omega_{\scriptscriptstyle S/U}}$ be the associated measurable density on $S/U$.
Then the following assertions (a) and (b) are equivalent.
\begin{enumerate}
\itema
The density $\phi_{\omega_{\scriptscriptstyle S/U}}$ is absolutely integrable.
\itemb
There exists a right $T$-invariant subset $\cZ$ of measure zero in $S$ such that
\begin{enumerate}
\item[{\rm (1)}]
for every $x \in S\setminus \cZ,$ the integral
$$
I_x(\phi) = \int_{T/U\ni [t]} \Delta_{\scriptscriptstyle S/T}(t)\, \phi(xt)\, dl_{t}([e])^{-1*} \omega_{\scriptscriptstyle T/U},
$$
is  absolutely convergent;
\item[{\rm (2)}]
the function $I(\phi): x \mapsto I_x(\phi)$ belongs to
$\cM(S:T:\Delta_{\scriptscriptstyle S/T});$
\item[{\rm (3)}]
the associated density
$I(\phi)_{\omega_{\scriptscriptstyle S/T}}$ is absolutely integrable.
\end{enumerate}
\end{enumerate}
Furthermore, if any of the conditions (a) and (b) are fulfilled, then
$$
\int_{S/U} \phi_{\omega_{\scriptscriptstyle S/U}} = \int_{S/T} I(\phi)_{\omega_{\scriptscriptstyle S/T}}.
$$
\end{Thm}

\subsection{Eisenstein integrals}
We start by recalling notation, definitions and results from \cite{vdBanKuit_EisensteinIntegrals}.

Let $(\tau,V_{\tau})$ be a finite dimensional representation of $K.$
We write $C^{\infty}(G/H\colon \tau)$ for the space of smooth $V_{\tau}$-valued functions $\phi$ on $G/H$ that satisfy the transformation property
$$
\phi(kx)
=\tau(k)\phi(x)\qquad(k\in K,x\in G/H).
$$
We further write $C_c^{\infty}(G/H\colon \tau)$ and $\cC(G/H\colon \tau)$ for the subspaces of $C^{\infty}(G/H\colon\tau)$ consisting of compactly supported functions and  $L^2$-Schwartz functions respectively.

Let $W(\faq)$ be the Weyl group of the root system of
$\faq$ in $\fg.$ Then
$$
W(\faq)=\Nor_{K}(\faq)/\Cen_{K}(\faq).
$$
 Let $W_{K\cap H}(\faq)$ be the subgroup of $W(\faq)$ consisting
of elements that can be realized in $\Nor_{K\cap H}(\faq)$.
We choose a set $\cW$ of representatives  for $W(\faq)/W_{K\cap H}(\faq)$ in $\Nor_{K}(\faq)\cap\Nor_{K}(\fa_{\fh})$ such that $e \in \cW$.
This is possible because of the following lemma.

\begin{Lemma}
\label{l: about NorKfaq}
$\Nor_{K}(\faq)
=\big(\Nor_{K}(\faq)\cap\Nor_{K}(\fa_{\fh})\big)\Cen_{K}(\faq)$.
\end{Lemma}

This result can be found in \cite[p.\ 165]{Rossmann_structure_ss_spaces}.
For the reader's convenience we give the concise proof.
\begin{proof}
It is clear that $\big(\Nor_{K}(\faq)\cap\Nor_{K}(\fa_{\fh})\big)\Cen_{K}(\faq)\subseteq    \Nor_{K}(\faq)$. To prove the other inclusion, assume that $k\in\Nor_{K}(\faq)$. Then $\Ad(k^{-1})\fa_{\fh}$ is a maximal abelian subspace of $\Cen_{\fg}(\faq)\cap \fp$. Each such maximal abelian subspace is conjugate to $\fa_{\fh}$ by an element from $\Cen_{G}(\faq)\cap K$, i.e., there exists a $k'\in\Cen_{K}(\faq)$ such that $\Ad(k')\Ad(k^{-1})\fa_{\fh}=\fa_{\fh}$. Note that $k'k^{-1}\in\Nor_{K}(\faq)\cap\Nor_{K}(\fa_{\fh})$.
Hence,
$$
k=(kk'^{-1})k'\in\big(\Nor_{K}(\faq)\cap\Nor_{K}(\fa_{\fh})\big)\Cen_{K}(\faq).
$$
\vspace{-40pt}

\end{proof}
\medbreak
 We define
$$
M_{0}
:=\Cen_{K}(\faq)\exp\big(\fp\cap[\Cen_{\fg}(\faq),\Cen_{\fg}(\faq)]\big).
$$
If  $P_{0} \in \cP(\Aq),$
 then $M_{0}A$ is a Levi-subgroup of $P_{0}$.
We write  $\fm_{0n}$ for the direct sum of the non-compact ideals of $\fm_{0}$.
The associated connected subgroup of $M_{0}$ is denoted by $M_{0n}.$

We denote by $\tau_{M}$ the restriction of $\tau$ to $M.$ Since $M$ is a subgroup
of $M_0 \cap K,$ it normalizes $M_{0n} \cap K,$ so that $(V_{\tau})^{M_{0n}\cap K}$
is an $M$-invariant subspace of $V_\tau.$ The restriction of $\tau_M$ to this subspace
is denoted by $\tau_M^0.$
We define
$$
\cA_{M,2}(\tau)
:=\bigoplus_{v \in \cW}\;C^\infty(M/M\cap vHv^{-1}\colon \tau_{M}^{0}).
$$
Each component in the sum is finite dimensional and thus a Hilbert space equipped with the
restriction of the inner product of $L^2(M/\cap vHv^{-1}, \Vtau);$
the direct sum is equipped with the direct sum Hilbert structure, and thus becomes a finite dimensional Hilbert space.

If $\psi\in\cA_{M,2}(\tau)$, we accordingly write $\psi_{v}$ for the
component  of $\psi$ in the space
 $C^\infty(M/M\cap vHv^{-1}:\tau_{M}^{0})$.

Let  $Q \in \cP(A).$ For $v\in\cW$ we define the parabolic subgroup $Q^v \in \cP(A)$ by
\begin{equation}
\label{e: defi Q v}
Q^{v}: = v^{-1}Qv.
\end{equation}
For each $v\in\cW$ we choose a positive density
$$
\omega_{\scriptscriptstyle H/H_{Q^{v}}}\in\cD_{\fh/\fh_{Q^{v}}}
$$ as follows. Fix positive densities $\omega_{\scriptscriptstyle G/H}\in\cD_{\fg/\fh}$ and $\omega_{\scriptscriptstyle G/H_{L}}\in\cD_{\fg/\fh_{L}}$. Furthermore, let $\omega_{v}\in\cD_{\fn_{Q^{v}}\cap \fh}$ be the positive density that corresponds to the Haar measure on $N_{Q^{v}}\cap H,$
which was chosen to be the push-forward of the Lebesgue measure on $\fn_{Q^{v}}\cap \fh$ along the exponential map (see text below Lemma \ref{Lemma N_(Q,X) x (N_Q cap H) to N_Q diffeo}). Then we choose $\omega_{\scriptscriptstyle H/H_{Q^{v}}}$ to be the unique density such that
\begin{equation}\label{eq omega_(G/H) otimes omega_(H/HQv) otimes omega_v=omega_(G/HL)}
\omega_{\scriptscriptstyle G/H}\otimes\omega_{\scriptscriptstyle H/H_{Q^{v}}}\otimes\omega_{v}
=\omega_{\scriptscriptstyle G/H_{L}}.
\end{equation}

The inner product $ B|_{\faq}$ on $\faq$ induces a linear isomorphism
$B:  \faq \to \faq^{*}.$
If $Q \in \cP(A),$ we define the cone $\GammaQ \subseteq   \faq$ as in  (\ref{eq def Gamma_Q}).
Then
$B(\Gamma(Q))$ equals the cone
spanned by the elements $\alpha+\sigma\theta\alpha$, with $\alpha\in\Sigma(Q)_{-}$, see (\ref{e: defi gS Q minus}).
Let $\Omega_Q \subseteq   \faq$ be defined as in (\ref{eq def Omega_Q})  and let
$\widehat \Omega_{Q}$ denote its hull in $\faqc$ with respect to the functions $\Re\, \langle\dotvar,\alpha\rangle$ with $\alpha\in\Sigma(\faq)\cap B( \Gamma(Q))$, i.e.,
$$
\widehat \Omega_{Q}
:= \{\lambda \in \faqc^{*} : \Re \langle\lambda,\alpha\rangle\leq \sup \,\Re \langle\Omega_Q,\alpha\rangle,
    \;\;\forall \alpha\in\Sigma(\faq)\cap B(\Gamma(Q))\}.
$$
Since $\langle\alpha,\lambda\rangle\leq0$ for all $\alpha\in\Sigma(\faq)\cap B(\GammaQ)$ and $\lambda\in-\Gamma(Q)^{\circ}$, it follows that we can describe the given hull by means of inequalities as follows:
$$
\widehat \Omega_{Q}
=\{\lambda \in \faqc^{*} :
    \Re \langle\lambda,\alpha\rangle\leq\max_{P\in \cP_{\sigma}(A,Q)}\langle-\rho_{P},\alpha\rangle,
    \;\; \forall \alpha\in\Sigma(\faq)\cap B(\Gamma(Q))\}.
$$
We define the following closed subsets of $\faqc$,
\begin{equation}\label{eq def Upsilon_Q}
\Upsilon_{Q} = \bigcap_{v \in \cW} v \Omega_{v^{-1}Qv},
\qquad \widehat\Upsilon_{Q} =  \bigcap_{v \in \cW} v \widehat\Omega_{v^{-1}Qv}.
\end{equation}

Given $v \in \cW$ we will use the notation
$$
\vH: = v H v^{-1},\qquad \mbox{\rm and} \qquad \vH_Q: = \vH \cap Q.
$$
Furthermore, we define the density ${}^v\omega_v $ on ${}^v\fh/{}^v\fh_Q$ by
$$
{}^v\omega:= \Ad(v^{-1})^{*}   \omega_{H/H_{Q^v}}.
$$
Given $\psi_v \in C^\infty(M/\vH \cap M: \tau_M^0)$ and $\lambda\in\Upsilon_{Q}$ we define
the function  $\psi_{v, Q,\lambda}: G \to V_{\tau}$ by
\begin{equation}
\label{e: defi psi v Q}
\psi_{v, Q,\lambda} (kman)
= a^{\lambda - \rho_{Q} - \rho_{Q,\fh}}\,\tau(k) \psi_v(m)
\qquad(k\in K, m\in M, a\in A, n\in N_{Q}).
\end{equation}
Then for  every $x \in G$ the function
$$
y \mapsto \psi_{v, Q,\lambda}(xy)\;dl_y(e)^{-1*}\, ({}^v\omega )
$$
defines a $V_{\tau}$-valued density on $\vH/\vH_{Q},$
which is integrable by \cite[Prop.\ 8.2]{vdBanKuit_EisensteinIntegrals}.

We define
$$
E_{v H v^{-1}}(Q:\psi_v: \lambda) : G \to V_{\tau}
$$
for $x\in G$ by
\begin{align*}
E_{v H v^{-1}}(Q:\psi_v: \lambda)(x)
&= \int_{\vH / \vH_Q}\; \psi_{v,Q,\lambda}(xy)\,dl_y(e)^{-1*}({}^v \omega)\\
&= \int_{H/ H_{Q^{v}}}\; \psi_{v,Q,\lambda}(xvhv^{-1})\,dl_h(e)^{-1*}\omega_{\scriptscriptstyle H/H_{Q^{v}}}.
\end{align*}
For $\psi \in \cA_{M,2}$ and $\lambda\in\Upsilon_{Q}$ we define the Eisenstein integral
\begin{equation}
\label{e: defi Eisenstein}
\Etau(Q:\psi:\lambda) = E_\tau(Q: \psi: \gl) : \;\; G \to V_{\tau}
\end{equation}
by
$$
\Etau(Q:\psi:\lambda)(x)
=\sum_{v \in \cW}  E_{v H v^{-1}}(Q: \pr_v \psi : \lambda)(xv^{-1})
\qquad(x\in G).
$$
It is readily verified that this function belongs to $C^\infty(G/H: \tau).$
The map $\Etau(Q:\psi:\dotvar)$ extends to a meromorphic $C^{\infty}(G/H:\tau)$-valued function on $\faqc^{*}$ and is holomorphic on an open neighborhood of $\widehat{\Upsilon}_{Q},$ see \cite[Cor.~8.5]{vdBanKuit_EisensteinIntegrals}.

Let $Q_1,Q_{2}\in \cP(A)$. Then there exists a unique meromorphic $\End(\cA_{M,2}(\tau))$-valued function
$\ctau(Q_{2}:Q_1:\dotvar)$ on $\faqc^{*}$ such that
\begin{equation}
\label{eq relation E_Q2 and E_Q1}
\Etau(Q_{2}:\psi:\lambda) = \Etau(Q_1:\ctau(Q_1:Q_{2}:\lambda)\psi:\lambda)
\end{equation}
for generic $\lambda\in\faqc^{*}$, see \cite[Cor.\ 8.14]{vdBanKuit_EisensteinIntegrals}.

In order to describe the relation of these Eisenstein integrals with those defined in terms
of a parabolic subgroup from the set $\cP_\gs(\Aq),$ see the text preceding (\ref{e: relation cP gs}), we need to introduce a bit more notation.

Let  $\Mzerohat$ denote the collection of (equivalence classes of) finite dimensional
irreducible unitary representations of $M_0.$ For $\xi \in \Mzerohat$ and
$v\in\cW$ we define the finite dimensional Hilbert space
$$
V(\xi,v):=\cH_{\xi}^{M_0 \cap v H v^{-1} }.
$$
The formal direct sum of these gives a finite dimensional Hilbert space
$$
V(\xi)
:=\bigoplus_{v\in\cW} V(\xi,v).
$$
We define
$C(K: \xi: \tau)$ to be the space of functions
$f: K \to \cH_{\xi} \otimes V_{\tau}$ transforming according to the rule:
$$
f(m k_{0} k )
= \big(\xi(m)\otimes\tau(k)^{-1}\big)f(k_{0}), \qquad (k,k_{0} \in K, m \in M_{0}).
$$
Let $\bar V(\xi)$ denote the conjugate space of $V(\xi)$ and consider
the natural map
\begin{equation}
\label{e: defi map psi T}
T\mapsto \psi_T,\;\; C(K:\xi: \tau) \otimes \bar V(\xi)\to\cA_{M,2}(\tau),
\end{equation}
which for $v\in\cW$ and $T=f\otimes\eta\in C(K:\xi: \tau) \otimes \bar V(\xi)$ is given by
$$
\big(\psi_{T}\big)_{v}(m)
=\langle f(e),\xi(m)\eta_{v}\rangle_{\xi}
\qquad(m\in M_{0}).
$$
Then the sum of the maps $T\mapsto(\dim \xi)^{\frac12}\psi_{T}$
over all $\xi \in \Mzerohat$ gives a surjective isometry
\begin{equation}
\label{e: direct sum over xi}
\bigoplus_{\xi \in \Mzerohat }\;  C(K:\xi: \tau) \otimes \bar V(\xi)  \;\;{\buildrel\simeq \over \longrightarrow}
\;\;\cA_{M,2}(\tau);
\end{equation}
see \cite[Lemma 3]{vdBanSchlichtkrull_FourierTransformOnASemisimpleSymmetricSpace}. Note that
only finitely many terms in the direct sum are non-zero.

Let now $Q \in \cP(A)$ and $R \in \cP_\gs(\Aq).$ Then we define the $C$-functions
$C_{R|Q}(s: \dotvar)$ for $s \in W(\faq)$ as
in \cite[Thm.~8.13]{vdBanKuit_EisensteinIntegrals}. These are $\End(\cA_{M,2}(\tau))$-valued meromorphic
functions on $\faqdc,$ with meromorphic inverses.
Moreover, by uniqueness of asymptotics, they are uniquely determined by the requirement that
\begin{equation}
\label{e: asymptotics E Q}
E(Q:\psi:\gl)(av)   \sim \sum_{s \in W(\fa_{\fq})} a^{s\gl - \rho_R}\; [C_{R|Q}(s:\gl)\psi ]_v(e) ,\quad (a \to \infty
\;\;{\rm in} \;\; \Aq^+(R))
\end{equation}
for all $\psi \in \cA_{M,2}(\tau),$ $v \in \cW$  and generic $\gl \in i \faqd.$

For $P_0\in \cP_\gs(\Aq)$ and $\psi \in \cA_{M,2}(\tau)$ we denote by $E(P_0:\psi:\gl)$
the Eisenstein integral as defined in \cite[Sect.\ 2]{vdBanSchlichtkrull_FourierTransformOnASemisimpleSymmetricSpace}.
Then $\gl \mapsto \Etau(P_0:\psi :\gl)$ is a meromorphic function on $\faqdc$
with values in $C^\infty(G/H:\tau).$ Given $R \in \cP_\gs(\Aq)$ the $C$-functions
$$
C_{R|P_0}(s:\gl) \in \End(\cA_{M,2}(\tau)),
$$
for $s \in W(\faq)$  are defined as in  \cite[Eqn.\ (46)] {vdBanSchlichtkrull_FourierTransformOnASemisimpleSymmetricSpace}.  These are meromorphic
functions with values in $\End(\cA_{M,2}(\tau))$ and with meromorphic inverses. Moreover, they
are uniquely determined by the asymptotic behavior of the Eisenstein integral $E(P_0:\psi:\gl)$, described
by (\ref{e: asymptotics E Q}) with everywhere $Q$ replaced by $P_0.$
\begin{Lemma}
Let $P_0 \in \cP_\gs(\Aq)$ and assume that $P \in \cP(A)$ is $\fq$-extreme, and satisfies
$P \subseteq   P_0.$ Then for generic $\gl \in \faqdc,$
\begin{equation}
\label{e: Eis P zero is Eis P}
E(P_0: \psi:\gl) = E(P:\psi:\gl).
\end{equation}
Furthermore, for all $R \in \cP_\gs(\Aq),$ $s \in W(\faq)$ and generic $\gl \in \faqdc,$
\begin{equation}
\label{e: C P zero is C P}
C_{R|P_0}(s:\gl) = C_{R|P}(s:\gl).
\end{equation}
\end{Lemma}
\begin{proof}
The first assertion is made in \cite[Cor.\ 8.6]{vdBanKuit_EisensteinIntegrals}.  The second assertion
follows by uniqueness of asymptotics.
\end{proof}
For $\psi\in \cA_{M,2}(\tau)$ we define the normalized Eisenstein integral $E^\circ(\bar P_0: \psi: \dotvar)$
as in  \cite[Sect. 5,6]{vdBanSchlichtkrull_FourierTransformOnASemisimpleSymmetricSpace}.
It is
a meromorphic  $C^\infty(H/H:\tau)$-valued function of $\gl \in \faqdc.$ Furthermore, for any
$R \in \cP_\gs(\Aq)$ we have
\begin{equation}
\label{l: expression in terms of normalized Eisenstein integral}
E(R:  \psi: \gl) = E^\circ(\bar P_0: C_{\bar P_0|R}(1: \gl) \psi : \gl),
\end{equation}
see
\cite[Eqn.\ (58)]{vdBanSchlichtkrull_FourierTransformOnASemisimpleSymmetricSpace}.

\begin{Lemma}
\label{l: relation EQ with Enorm}
Let $Q \in \cP(A)$ and $P_0 \in \cP_\gs(\Aq).$
Then
\begin{equation}\label{e: rel EQ and Enorm}
\Etau(Q:\psi:\lambda)
=\Etau^{\circ}(\bar{P_{0}}: C_{\bar P_0|Q}(1:\gl)\psi:\lambda),
\end{equation}
for all $\psi \in \cA_{M,2}(\tau)$ and generic $\gl \in \faqdc.$
\end{Lemma}

\begin{proof}
Let $P \in \cP_\gs(A)$ be such that $P_0 \supseteq P.$ Then it follows from
(\ref{eq relation E_Q2 and E_Q1}) and (\ref{e: Eis P zero is Eis P}) that
$$
E(Q:\psi: \gl) = E(P: C(P:Q:\gl)\psi: \gl ) = E(P_0: C(P:Q:\gl)\psi: \gl).
$$
Using
(\ref{l: expression in terms of normalized Eisenstein integral})
with $R = \bar P_0,$
we infer that
\begin{equation}
\label{e: Eis Q as nEis bar P zero}
E(Q:\psi: \gl) = E^\circ (\bar P_0: C_{\bar P_0| P_0}(1:\gl) C(P:Q:\gl)\psi: \gl).
\end{equation}
By application of (\ref{e: C P zero is C P}) and \cite[Cor.\ 8.14 (a)]{vdBanKuit_EisensteinIntegrals}
with $R = \bar P_0,$ we find that
\begin{eqnarray}
C_{\bar P_0| P_0}(1:\gl) C(P:Q:\gl) &=& C_{\bar P_0| P}(1:\gl) C(P:Q:\gl)\nonumber \\
&=& C_{\bar P_0| Q}(1:\gl). \label{e: formula for C bar P zero Q}
\end{eqnarray}
Substituting the latter expression in (\ref{e: Eis Q as nEis bar P zero}), we obtain (\ref{e: rel EQ and Enorm}).
\end{proof}

Our next goal is to describe the $C$-function in (\ref{e: rel EQ and Enorm}) in terms of a standard intertwining
operator in case $Q$ and $P_0$ are suitably related. For this, we need to introduce additional
notation.

Let $\xi \in \Mzerohat.$ We define  $C^\infty(K:\xi)$
to be the space of functions  $f: K \to \cH_\xi$ transforming according to the rule
\begin{equation}
\label{e: transformation rule C K xi}
f(mk): = \xi(m) f(k),\qquad (k \in K, m \in M_0 \cap K).
\end{equation}
Furthermore, we put $\xi_M = \xi|_M$ and define $C^\infty(K:\xi_M)$ to be the space
of functions $f: K \to \cH_\xi$ transforming according to the same rule
(\ref{e: transformation rule C K xi}) but for $k \in K$ and $m \in M.$ Since $M \subseteqq M_0,$
we have a natural inclusion map
$$
\inj: C(K:\xi) \to C(K:\xi_M).
$$
Following \cite[Sect.\ 4]{vdBanKuit_EisensteinIntegrals}, we denote by
$$
\proj: C(K:\xi_M) \to C(K:\xi)
$$
the transpose of this map for the natural sesquilinear pairings coming from the $L^2$-inner product
on $L^2(K, \cH_\xi, dk).$

If $P_0 \in \cP_\gs(\Aq),$ and $\gl \in \faqdc,$ we denote the realization of the normalized
induced representation $\Ind_{P_0}^G(\xi \otimes \gl \otimes 1)$ of $G$ in $C^\infty(K: \xi)$
according to the compact picture by $\pi_{P_0, \xi, \gl}.$ Given a second parabolic subgroup
$P_1 \in \cP_\gs(\Aq)$ we denote by
$$
A(P_1: P_0: \xi: \gl): C^\infty(K:\xi) \to C^\infty(K:\xi)
$$
the (meromorphic continuation of) the standard intertwining operator
which intertwines the representations
$\Ind_{P_j}^G(\xi \otimes \gl \otimes 1),$ for $j=0,1,$ respectively.

Likewise, if $Q \in \cP(A), \xi \in \Mzerohat$ and $\mu \in \fadc,$
we denote the realization of the normalized induced representation
$\Ind_Q^G(\xi_M \otimes \mu \otimes 1)$ of $G$ in $C^\infty(K: \xiM)$ according to
the compact picture by $\pi_{Q, \xi_M, \mu}.$

Given $Q_1, Q_2 \in \cP(A)$ we denote by
$$
A(Q_1:Q_2:\xiM: \mu ): C^\infty(K:\xiM) \to C^\infty(K:\xiM),
$$
the (meromorphic continuation of) the standard intertwining operator which intertwines the representations
$$
\Ind_{Q_j}^G(\xi_M \otimes \mu \otimes 1),
$$ for $j=2,1,$ respectively.

The two types of parabolically induced representations are related by the maps
$\inj$ and $\proj$ defined above.
Let $P \in \cP(A)$ be a $\fq$-extreme parabolic subgroup, and let $P_0 $ be
the unique parabolic subgroup in $\cP_\gs(\Aq)$ containing $P.$
Then $\inj$ intertwines $\pi_{P_0, \xi, \gl}$ with $\pi_{P, \xiM, \gl - \rho_{P,\fh}}$
and $\proj$ intertwines $\pi_{P, \xiM, \gl + \rho_{P,\fh}}$ with $\pi_{P_0, \xi, \gl},$
for every $\gl \in \faqdc.$ We refer to \cite[Sect.~4]{vdBanKuit_EisensteinIntegrals} for further details.

We denote by
\begin{equation}
\label{e: defi Pi Sigma}
\Pi_{\Sigma,\R}(\faqd)
\end{equation}
the set of polynomial functions $\faqc^{*} \to \C$ that can be expressed as non-zero products of affine functions of the form $\lambda\mapsto \langle\lambda,\alpha\rangle-c$, where $\alpha\in\Sigma\setminus\fahd $ and $c\in\R$.

Finally, we arrive at the mentioned description of the $C$-function in
(\ref{e: rel EQ and Enorm}).

\begin{Prop}\label{Prop relation E(Q) and E^circ(cP0)}
Let $Q \in \cP(A),$ $P \in \cP_\gs(A, Q),$ see (\ref{e: defi P gs A Q}), and let $P_0$ be the unique minimal
$\gs\Cartan$-stable parabolic subgroup containing $P.$
Then the following assertions are valid.
\begin{enumerate}
\itema
If $\xi \in \Mzerohat$ and
$T \in C^{\infty}(K:\xi:\tau)\otimes \barV(\xi)$ then, for generic $\gl \in \faqdc,$
\begin{equation}
\label{e: C by intertwining operator}
C_{\bar{P}_{0}|Q}(1:\lambda) \psi_{T}
=\psi_{[\proj \after A(\gs P:Q:\xiM:-\lambda+\rho_{P,\fh})\after \inj \otimes I]T}.
\end{equation}
\itemb
The function $\ctau_{\bar{P}_{0}|Q}(1:\dotvar)$ is holomorphic on the set
\begin{equation}
\label{e: set of holomorphy for C}
\big\{\lambda\in\faqc^{*}:\Re\langle-\lambda+\rho_{Q,\fh},\alpha\rangle>0
    \textnormal{ for all }\alpha\in\Sigma(P_{0})\cap\Sigma(Q)\big\}.
\end{equation}
\itemc
Let $B\subseteq    \faq^{*}$ be open and bounded. There exists a $p\in\Pi_{\Sigma,\R}(\faqd)$ such that
$$
\lambda\mapsto p(\lambda)\ctau_{\bar{P}_{0}|Q}(1:\lambda)
$$
is holomorphic and of polynomial growth on $B+i\faq^{*}$.
\end{enumerate}
\end{Prop}

\begin{proof}
We first turn to (a).
From (\ref{e: formula for C bar P zero Q}) we recall that
\begin{equation}
\label{e: C bar P zero Q as composition}
C_{\bar P_0|Q} (1:\gl) =  C_{\bar P_0 | P_0}(1:\gl)\, C(P:Q: \gl).
\end{equation}
Let $\xi$ and $\psi$ be as in assertion (a).
Then it follows from \cite[Prop.\ 8.7]{vdBanKuit_EisensteinIntegrals} that
\begin{equation}
\label{e: C and S}
C(P:Q: \gl) \psi_T = \psi_{S(\gl)}
\end{equation}
with
\begin{equation}
\label{e: formula S gl}
S(\gl) =  [\proj \after A(Q:P:\xiM: - \gl + \rho_{P,\fh})^{-1}\after \inj \otimes I]T.
\end{equation}
On the other hand, by
\cite[Prop.\ 3.1]{vdBanSchlichtkrull_FourierTransformOnASemisimpleSymmetricSpace},
\begin{equation}
\label{e: C  and S prime}
C_{\bar P_0 | P_0}(1:\gl) \psi_{S(\gl)} = \psi_{S'(\gl)},
\end{equation}
with
\begin{equation}
\label{e: formula S prime gl}
S'(\gl) = [A(\bar P_0 : P_0 : \xi : - \gl) \otimes I]S(\gl).
\end{equation}
From (\ref{e: C bar P zero Q as composition}), (\ref{e: C and S})
and (\ref{e: C and S prime}) we obtain that
\begin{equation}
\label{e: second C and S prime}
C_{\bar P_0|Q}(1:\gl) \psi_T = \psi_{S'(\gl)};
\end{equation}
we will prove (a) by determining $S'(\gl).$

It follows from \cite[Lemma 8.10]{vdBanKuit_EisensteinIntegrals}, that the following diagram commutes, for generic $\gl \in \faqdc,$
$$
\begin{array}{ccc}
C(K:\xiM) &\buildrel {\scriptscriptstyle A(\gs P : P: \xiM: -\gl +\rho_{P,\fh})} \over \longrightarrow & C(K:\xiM)\\
{\scriptscriptstyle \proj } \downarrow && \downarrow {\scriptscriptstyle \proj} \\
C(K:\xi) &\buildrel  {\scriptscriptstyle  A(\bar P_0 : P_0:  \xi: - \gl)} \over  \longrightarrow & C(K:\xi)
\end{array}
$$
Taking the commutativity of this diagram into account,
we infer by combining
(\ref{e: formula S gl}) and (\ref{e: formula S prime gl}) that
\begin{eqnarray}
\nonumber
\lefteqn{
S'(\gl) = }\\
&=& [ \proj \!\after A(\gs P\!:  \!P \!: \! \xiM \! : \! -\gl  \!+  \!\rho_{P,\fh}) A(Q \!: \!P \!: \!\xiM \!: \! - \gl  \!+  \!\rho_{P,\fh})^{-1} \!\after \inj \otimes I]T\:\:\:\:
\label{e: formula for S' gl with prod}
\end{eqnarray}
Since $P\succeq Q$, we have
$
\Sigma(\sigma P)\cap\Sigma(P)
=\Sigma(P,\sigma)
\subseteq   \Sigma(Q)\cap\Sigma(P).
$
By application of
\cite[Cor.\ 7.7]{KnappStein_IntertwiningOperatorsForSemisimpleGroups_II} we find that
\begin{align}
\nonumber
&A(\gs P :P:\xiM:-\lambda+\rho_{Q,\fh})\\
\label{e: product intertwining operators}
&\qquad=A(\gs P \!:Q \!: \xiM :-\lambda+\rho_{P,\fh})\after A(Q  \!: \! P \! :\xiM : -\lambda+\rho_{Q,\fh}).
\end{align}
The identity in (\ref{e: C by intertwining operator})
now follows from
(\ref{e: second C and S prime}), (\ref{e: formula for S' gl with prod})
and (\ref{e: product intertwining operators}).
Thus, (a) holds.

We turn to (b) and (c).  Let $\xi \in \Mzerohat$ and let $\End(C(K:\xiM))$ denote the space of bounded linear endomorphisms of the Banach space $C(K:\xiM).$ Then as a $\End(C(K:\xi_M))$-valued function, the standard intertwining operator  $A(\gs P\!:\!Q:\xiM : \mu)$ depends holomorphically on $\mu \in \fa^{*}_\iC$ satisfying
\begin{equation}
\label{e: condition for holomorphy A}
\Re\langle \mu ,\alpha\rangle > 0,
\qquad\big(\alpha\in\Sigma(\gs \bar P)\cap\Sigma(Q)\big).
\end{equation}
Indeed, this is a straightforward consequence of the convergence of the integral
defining the intertwining operator, asserted in \cite[Thm.~4.2]{KnappStein_IntertwiningOperatorsForSemisimpleGroups_II}.

Since $\gS(P_0) = \gS(P)\setminus \fahd,$ we have $\gS(P_0) \cap \gS(Q) \subseteq   \gS(\gs \bar P) \cap \gS(Q).$
Thus, if $\gl \in \faqdc$ belongs to the set (\ref{e: set of holomorphy for C}) then $\mu = - \gl + \rho_{P,\fh}$ satisfies (\ref{e: condition for holomorphy A}). We infer that $A(\gs P: Q : \xi_M: - \gl + \rho_{P,\fh})$
is a holomorphic $\End(C(K:\xi_M))$-valued function of $\gl$ in the set
(\ref{e: set of holomorphy for C}).  In view of (\ref{e: C by intertwining operator}) we now infer (b).

Finally, we turn to (c).
Let $\xi \in \Mzerohat$ and let $\cA_{M,2}(\tau)_\xi$ denote the image of $C(K:\xi \:\tau) \otimes \bar V(\xi)$
under the map (\ref{e: direct sum over xi}).
We may select a finite set of $K$-types
 $\vartheta \subseteq   \widehat K$  such that $C(K:\xi: \tau) \subseteq   C(K:\xi)_\vartheta \otimes V_\tau.$
In view of \cite[Thm.\ 1.5]{vdBanSchlichtkrull_estimates_cfunction}  there exists
a polynomial function $q_\xi: \fadc \to \C$ which is a product of linear factors of the form
$\mu  \mapsto \inp{\mu}{\ga} - c,$ with $\ga \in \gS(\gs \bar P) \cap \gS(Q)$ and $c \in \R$ such that
\begin{equation}
\label{e: intertwiner on K types}
\mu \mapsto q_\xi(\mu) A(\gs P: Q : \xi_M: \mu)|_{C(K:\xi_M)_\vartheta}
\end{equation}
is holomorphic and polynomially bounded on the set $-B + \rho_{Q,\fh} + i\fad.$
 It follows that the function
of $\gl \in \faqdc$ arising from (\ref{e: intertwiner on K types}) by  the substitution $\mu = -\gl + \rho_{Q,\fh}$ is holomorphic and polynomially bounded on $B + i\faqd.$ Define
$$
p_\xi(\gl): = q_\xi(-\gl + \rho_{Q,\fh}).
$$
Then $p_\xi \in \Pi_{\gS, \R}(\faqd)$  because
$\gS(\gs\bar P) \cap \gS(Q) \subseteqq \gS\setminus \fahd,$ and in view of (\ref{e: C by intertwining operator})
it follows that
$$
\gl \mapsto p_\xi(\gl) C_{\bar P_0|Q}(1: \gl)_{\cA_{M,2}(\tau)_\xi}
$$
is holomorphic and polynomially bounded on $B + i\faqd.$
The result now follows by finiteness of the sum (\ref{e: direct sum over xi}).
\end{proof}

\subsection{The $\tau$-spherical Fourier transform}
Let $Q \in \cP(A)$ and let $(\tau, \Vtau)$ be a finite dimensional unitary representation of $K.$ For $\phi\in C_{c}^{\infty}(G/H:\tau)$, we define the $\tau$-spherical Fourier transform  $\Ft_{Q,\tau}\phi$  to be the meromorphic function $\faqc^{*}\to\oC(\tau)$ determined by
$$
\langle \Ft_{Q,\tau}\phi(\lambda),\psi\rangle
=\int_{G/H}\langle\phi(x)\, ,\, \Etau(Q:\psi:-\overline{\lambda})(x)\rangle_{\tau}\,dx
$$
for $\psi\in\cA_{M,2}(\tau)$ and generic $\lambda\in\faqc^{*}$.

\begin{Prop}
Let $\phi\in C_{c}^{\infty}(G/H:\tau)$. Then
$\Ft_{Q,\tau}\phi$ is holomorphic on an open neighborhood of $-\widehat{\Upsilon}_{Q}$,
see  (\ref{eq def Upsilon_Q}).
\end{Prop}

\begin{proof}
This follows directly from \cite[Cor.~8.5]{vdBanKuit_EisensteinIntegrals}.
\end{proof}

Before proceeding we will first discuss how this Fourier transform is related to the
$\tau$-spherical Fourier transform $\nFt_{\bar P_0}\phi$ defined in
\cite[Eqn. (59)]{vdBanSchlichtkrull_FourierTransformOnASemisimpleSymmetricSpace},
for $P_0,$ hence $\bar P_0,$ a minimal $\gs\Cartan$-stable parabolic subgroup from $\cP_\gs(\Aq).$
The last mentioned transform is defined to be the meromorphic
function  $\faqc^{*}\to\oC(\tau)$ given by
$$
\inp{\nFt_{\bar P_0}\phi(\gl) }{\psi} = \int_{G/H} \inp{f(x)}{E^\circ(\bar P_0:\psi: -\bar \gl)(x)}_\tau \; dx,
$$
for $\psi \in \oC(\tau)$ and generic $\gl \in \faqdc.$

\begin{Prop}\label{Prop F_Q phi=c F^0_cP0 phi}
Let $P_0 \in \cP_\gs(\Aq)$ and $\phi\in C_{c}^{\infty}(G/H:\tau).$ Then
$$
\Ft_{Q,\tau}\phi(\lambda)
=\ctau_{\bar P_{0}|Q}(1: -\bar{\lambda})^{*}\nFt_{\bar P_0}\phi(\lambda).
$$
for generic $\lambda\in\faqc^{*}.$
\end{Prop}
\begin{proof}
The identity follows directly from Lemma \ref{l: relation EQ with Enorm}.
\end{proof}

Given $R>0$ we write $B_{R}$ for the open ball in $\faq$ with center $0$ and radius $R$. Furthermore, we define
$$
C_{R}^{\infty}(G/H:\tau):= \{ \phi \in C_{c}^{\infty}(G/H:\tau): \; \supp\phi\subseteq    K\exp(B_{R})\cdot H\}.
$$
\begin{Prop}\label{Prop descr of singularities of E(Q:psi:.) and Paley-Wiener estimate for F_(Q,tau)}
Let $B\subseteq    \faq^{*}$ be open and bounded.
There exists a $p\in\Pi_{\Sigma,\R}(\faqd)$ such that $\lambda\mapsto p(-\lambda)\Etau(Q:\psi:-\lambda)$ is holomorphic on $B+i\faq^{*}$ for every $\psi\in\oC(\tau)$. Moreover, if $\phi\in C_{c}^{\infty}(G/H,\tau)$, then
$$
\lambda\mapsto p(\lambda)\Ft_{Q,\tau}\phi(\lambda)
$$
is holomorphic on $B+i\faq^{*}$.

Let $R>0$. There exist a constant $C_{R}>0$ and for every $N\in \N$  a
continuous seminorm $\nu_{N}$ on $C_{R}^{\infty}(G/H:\tau)$ such that
\begin{equation}\label{eq Paley-Wiener estimate for F_(Q,tau)}
\|p(\lambda)\Ft_{Q,\tau}\phi(\lambda)\|
\leq(1+\|\lambda\|)^{-N}e^{C_{R}\|\Re\lambda\|}\nu_{N}(\phi)
\end{equation}
for every $\phi\in C_{R}^{\infty}(G/H:\tau)$ and all $\lambda\in B+i\faq^{*}$.
\end{Prop}

\begin{proof}
We fix a $\fq$-extreme parabolic subgroup $P \in \cP(A)$ such that $P \succeq Q.$ Let $P_0$ be the unique
subgroup in $\cP_\gs(\Aq)$ such that $P_0 \supseteq P.$
By \cite[Prop.\ 3.1]{vdBanSchlichtkrull_FourierInversionOnAReductiveSymmetricSpace} there exists a $p_1\in\Pi_{\Sigma,\R}(\faqd)$ such that
$$
\lambda\mapsto p_1(-\lambda)E^{\circ}(\bar{P_{0}}:\psi:-\lambda)
$$
is holomorphic on $B+i\faq^{*}$ for every $\psi\in\oC(\tau)$. This implies that $p_1\nFt_{\bar{P_{0}}}\phi$
is holomorphic on $B+i\faq^{*}$ for every $\phi\in C_{c}^{\infty}(G/H:\tau)$. Furthermore, by \cite[Lemma 4.4]{vdBanSchlichtkrull_FourierInversionOnAReductiveSymmetricSpace} there exist a constant $C_{R}>0$ and for every $N\in \N$  a continuous seminorm $\nu_{N}$ on $C_{R}^{\infty}(G/H:\tau)$ such that
$$
\|p_1(\lambda)\nFt_{\bar P_0,\tau}\phi(\lambda)\|
\leq(1+\|\lambda\|)^{-N}e^{C_{R}\|\Re\lambda\|}\nu_{N}(\phi)
$$
for every $\phi\in C_{R}^{\infty}(G/H:\tau)$ and all $\lambda\in B+i\faq^{*}$.

Choose $p_{2}\in\Pi_{\Sigma,\R}(\faqd)$ as in Proposition \ref{Prop relation E(Q) and E^circ(cP0)}(b,c)
and put $p = p_1p_2.$ Then the result follows in view of
Proposition \ref{Prop F_Q phi=c F^0_cP0 phi}, by combining the above assertions with those of  Proposition \ref{Prop relation E(Q) and E^circ(cP0)}(b,c).
\end{proof}

\section{The $\tau$-spherical Harish-Chandra transform}
\label{s: tau spherical HC transform}
\subsection{Definition and relation with the spherical Fourier transform}
We assume that $Q \in \cP(A)$ and that $(\tau, \Vtau)$ is a finite
dimensional unitary representation of $K.$
 Recall the definition of the character $\gd_Q$ on $L$
by  (\ref{e: defi Delta Q}),  see also
(\ref{e: Delta Q and rho}). The following definition makes use of the notation (\ref{e: defi Q v}).

\begin{Defi}
\label{d: tau HCT}
For a function $\phi\in C_{c}^{\infty}(G/H:\tau)$ we define its $\tau$-spherical Harish-Chandra transform  $\Ht_{Q,\tau}\phi$ to be the function $A_{\fq}\to \oC(\tau)$ given by
\begin{equation}\label{eq def H_(Q,tau)}
\Big(\Ht_{Q,\tau}\phi(a)\Big)_{\!v}(m)
:=\Deltach_{Q}(a)\int_{N_{Q^{v}}/H_{N_{Q^{v}}}}\phi(mavn)\,dn
\end{equation}
for $v\in\cW$, $m\in M$ and $a\in A_{\fq}$.
\end{Defi}

It is easily seen that $\Ht_{Q,\tau}$ defines a continuous linear map $C_{c}^{\infty}(G/H:\tau)\to C^{\infty}(A_{\fq})\otimes\oC(\tau)$.
The $\tau$-spherical Harish-Chandra transform $\Ht_{Q,\tau}$ is related to the Harish-Chandra transform introduced in  Definition \ref{Defi Radon and HC-transform}.
Namely, the following result is valid.

\begin{Lemma}\label{Lemma relation between H_(Q,tau) and H_Q}
Let $\phi\in C_{c}^{\infty}(G/H:\tau)$. Then for  $a\in A_{\fq}$ and $\psi\in\oC(\tau)$
\begin{equation}\label{eq relation H_(Q,tau) and H_Q}
\big\langle\Ht_{Q,\tau}\phi(a),\psi\big\rangle
=\sum_{v\in\cW}\Ht_{Q^{v}}\Big(\big\langle\phi(\dotvar),\tau(v^{-1})\psi_{v}(e)\big\rangle_{\tau}
\Big)(v^{-1}av).
\end{equation}
\end{Lemma}

\begin{proof}
Let $\phi\in C_{c}^{\infty}(G/H:\tau)$, $\psi\in\oC(\tau)$ and $a\in A_{\fq}$.
Recall that $\vH$ denotes $vHv^{-1}$ for $v\in\cW$.
Then
\begin{align*}
&\langle\Ht_{Q,\tau}\phi(a),\psi\rangle
=\sum_{v\in\cW}\int_{M/(M\cap\vH)}\Big\langle\big(\Ht_{Q,\tau}\phi(a)\big)_{v}(m),\psi_{v}(m)\Big\rangle_\tau\\
&\qquad=\sum_{v\in\cW}\int_{M/(M\cap\vH)}\Deltach_{Q}(a)
    \int_{N_{Q^{v}}/H_{N_{Q^{v}}}}\big\langle\phi(mavn),\psi_{v}(m)\big\rangle_\tau       \\
&\qquad=\sum_{v\in\cW}\int_{M/(M\cap\vH)}\Deltach_{Q^{v}}(v^{-1}av)
    \int_{N_{Q^{v}}/H_{N_{Q^{v}}}}\big\langle\phi(mvv^{-1}avn),\psi_{v}(m)\big\rangle_{\tau}.
\end{align*}
We now use that $\tau$ is unitary and that the measure on $M/(M\cap\vH)$
is normalized, and thus we conclude that the last expression is equal to
$$
\sum_{v\in\cW}\Deltach_{Q^{v}}(v^{-1}av)
    \int_{N_{Q^{v}}/H_{N_{Q^{v}}}}\langle\phi(v^{-1}avn),\tau(v^{-1})\psi_{v}(e)\rangle_\tau.
$$
Finally, the claim follows from the definition of the Harish-Chandra transform (Definition \ref{Defi Radon and HC-transform}).
\end{proof}

\begin{Cor}{\ }
\begin{enumerate}
\itema
Let $P \in \cP_\gs(A,Q).$ Then the spherical Harish-Chandra transform $\Ht_{Q,\tau}$ is a continuous linear map
$C_c^\infty(G/H:\tau) \to L^1(\Aq, \gd_P^{-1}da) \otimes \oC(\tau).$
\itemb
Let
$\phi \in C_c^\infty(G/H:\tau)$ be supported in $K \exp C \cdot H,$ with $C \subseteqq \faq$ compact, convex and invariant under the action of $N_{K\cap H}(\faq).$
Then
\begin{equation}
\label{e: inclusion support HQtau}
\supp \Ht_{Q,\tau}(\phi) \subseteq \bigcup_{v \in \cW} \exp (C  + v\Gamma(Q^v))
\end{equation}
\end{enumerate}
\end{Cor}

\begin{proof}
It follows from Proposition \ref{Prop Delta_P_0^(-1/2)H_Q maps C_c^infty to C(L/L cap H)} that  $\gd_{P^v}^{-1} (\Ht_{Q^v} \otimes I)$ defines a continuous linear map
$C^\infty_c(G/H: \tau) \to L^1(L/H_L: \tau_{M}).$ Since $L/H_L \simeq M/M\cap H \times \Aq),$
it follows that restriction to $\Aq$ defines a continuous linear map $ L^1(L/H_L: \tau_{M})
\to \Aq \otimes \Vtau.$ In view of  Lemma  \ref{Lemma relation between H_(Q,tau) and H_Q} assertion (a)
of the corollary now follows.

For (b), assume that $\phi \in C_c^\infty(G/H: \tau)$ has a support as stated.
Then by Proposition \ref{Prop support of H phi} the support of $(\Ht_{Q^v} \otimes I)(\phi)|_{\Aq} $ is contained in
$\exp(C + \Gamma(Q^v)). $ In view of the $N_{K\cap H}(\faq)$-invariance
of $C,$ the inclusion (\ref{e: inclusion support HQtau}) now follows by application of Lemma \ref{Lemma relation between H_(Q,tau) and H_Q}.
\end{proof}

It follows from this corollary that for $\phi \in C_c^\infty(G/H\colon \tau)$  the Euclidean Fourier transform
$$
\eFt_{\!\Aq}(\Ht_{Q,\tau} \phi)(\gl) = \int_{\Aq} \Ht_{Q, \tau}\phi(a) a^{-\gl} \; da,
$$
is well defined for $\gl$ in the subset $- \Upsilon_Q \subseteqq \faqdc,$ with absolutely convergent integral,
and defines a holomorphic $\oC(\tau)$-valued function on the interior of this set.

\begin{Prop}\label{Prop Ft_(Q,tau)=F_A circ H_Q}
Let $\phi\in C_{c}^{\infty}(G/H:\tau)$. Then for $\lambda\in-\Upsilon_{Q},$
$$
\Ft_{Q,\tau}\phi(\lambda)
=\eFt_{\!A_{\fq}}\big(\Ht_{Q,\tau}\phi\big)(\lambda).
$$
\end{Prop}

Before turning to the proof of this result, we first prove a lemma.

\begin{Lemma}\label{Lemma int_(H_Q / G)=int_K int_A_q int_(N/N cap H)}
Let $\omega\in\cD_{\fg/\fh_{Q}}$. Let $\psi\in C(G:H_{Q}:\Delta_{\scriptscriptstyle G/H_{Q}})$ and assume that the associated density $\psi_{\omega}\in\cD_{G/H_{Q}}$ is integrable. Then
$$
\int_{G/H_{Q}}\psi_{\omega}
=\int_{K}\int_{A_{\fq}}\int_{N_{Q}/H_{N_{Q}}}a^{2\rho_{Q}}\psi(kan)\,dn\,da\,dk
$$
up to a positive constant which only depends on the normalization of the measures and the densities.
\end{Lemma}

\begin{proof}
In this proof we will need to introduce several densities. For each quotient $S/T$ of a Lie group $S$ by a closed subgroup $T$ that appears below, we choose a positive density $\omega_{\scriptscriptstyle S/T}\in\cD_{\fs/\ft}$. We leave it to the reader to check that these densities may be normalized in such a manner that the stated equalities
are valid.

By Theorem \ref{Thm Fubini theorem for densities},
\begin{equation}\label{eq int_(G/HQ)=int_(G/Q)int_(Q/HQ)}
\int_{G/H_{Q}}\psi_{\omega}
=\int_{G/Q}
        I_g(\psi)\,
        dl_{g}([e])^{-1*}\omega_{\scriptscriptstyle G/Q}.
\end{equation}
where
\begin{equation}
\label{e: defi I one}
I_g(\psi) =
\int_{Q/H_{Q}} \psi(gq)\,\Delta_{\scriptscriptstyle G/Q}(q)\,dl_{q}([e])^{-1*}\omega_{\scriptscriptstyle Q/H_{Q}}.
\end{equation}
Since the canonical map
$
\zeta:  K/M\to G/Q
$
is a $K$-equivariant diffeomorphism we may rewrite the integral on the right-hand side of
(\ref{eq int_(G/HQ)=int_(G/Q)int_(Q/HQ)}) as an integral over $K/M$ of the pull-back density
$$
\zeta^{*}\big(g\mapsto I_g(\psi)dl_{g}([e])^{-1*}\omega_{\scriptscriptstyle G/Q}\big)_k =
I_k(\psi) \, dl_k(e)^{-1*}\omega_{K/M}.
$$
Now $k\mapsto I_k(\psi)$ is right $M$-invariant, and $k \mapsto dl_k(e)^{-1*}\omega_{K/M}$
defines a left $K$-equivariant density on $K/M.$
Hence,
\begin{equation}
\label{e: second integral psi omega}
\int_{G/Q} \psi_\omega = \int_{K} I_k(\psi) \,dk.
\end{equation}
Next, we fix $k \in K.$ Applying  Theorem \ref{Thm Fubini theorem for densities} to the integral for
$I_k(\psi),$ given by (\ref{e: defi I one}) with $g = k,$ we infer that
\begin{equation}
\label{e: rewrite I one}
I_k (\psi) =
\int_{Q/H_{L}N_{Q}} J_y(l_k^*\psi \, \Delta_{\scriptscriptstyle G/Q})
        \,dl_{y}([e])^{-1*}\omega_{\scriptscriptstyle Q/ \LH N_{Q}},
\end{equation}
where
\begin{eqnarray}
\nonumber
J_y(l_k^*\psi \, \Delta_{\scriptscriptstyle G/Q}  )   &=&
\int_{H_L N_{Q}/H_{Q}}
        \psi(kyx)\,\Delta_{\scriptscriptstyle G/Q}(yx)\Delta_{\scriptscriptstyle Q/ \LH  N_{Q}}(x)
        \,dl_{x}([e])^{-1*}\omega_{\scriptscriptstyle \LH N_{Q}/H_{Q}} \\
    &=& \label{e: rewrite I two}
               \int_{H_L  N_{Q}/H_{Q}}
        \psi(kyx)\,\Delta_{\scriptscriptstyle G/Q}(yx)        \,dl_{x}([e])^{-1*}\omega_{\scriptscriptstyle \LH N_{Q}/H_{Q}} .
\end{eqnarray}
In the latter equality we have used that    $\Delta_{\scriptscriptstyle Q/ \LH  N_{Q}} = 1.$
Indeed, by nilpotency of $N_Q$ it is evident that $ \Delta_{\scriptscriptstyle Q/ \LH  N_{Q}}|_{N_Q} = 1.$
On the other hand,  $\Delta_{\scriptscriptstyle Q/ \LH  N_{Q}}|_{\LH} =
\Delta_{L/H_L} = 1$ by unimodularity of $L$ and $H_L.$

To complete the proof we will rewrite both integrals  (\ref{e: rewrite I one}) and (\ref{e: rewrite I two}), respectively.
Starting with the first, we note that the map
$$
\eta:A_{\fq}\times M/H_{M}\to Q/\LH  N_{Q};\qquad (a,m)\mapsto am\LH  N_{Q}
$$
is a $\Aq \times  M$-equivariant diffeomorphism. Hence, the density $\eta^{*}\big(y\mapsto\,dl_{y}([e])^{-1*}\omega_{\scriptscriptstyle Q/L_{H}N_{Q}}\big)$ is  left
 $A_{\fq}\times M$ invariant. Accordingly, the integral (\ref{e: rewrite I one})
 may be rewritten as
\begin{eqnarray}
\nonumber
I_k (\psi)  & = & \int_{M/H_{M}}\int_{A_{\fq}} J_{am}(l_k^*\psi \, \Delta_{\scriptscriptstyle G/Q}  )
        \,da\,d\bar m \\
        \label{e: second rewrite I one}
        &=&
        \int_{M}\int_{A_{\fq}} J_{ma}(l_k^*\psi \, \Delta_{\scriptscriptstyle G/Q}  )
        \,da\,d m
\end{eqnarray}
Likewise, the map
$$
 \vartheta:    N_{Q}/H_{N_{Q}}\to \LH  N_{Q}/H_{Q};\qquad nH_{N_{Q}}\mapsto nH_{Q}
$$
is a left $N_Q$-equivariant  diffeomorphism. Therefore,   $\vartheta^{*}\big(x\mapsto dl_{x}([e])^{-1*}\omega_{\scriptscriptstyle \LH  N_{Q}/H_{Q}}\big)$ is an $N_{Q}$-invariant density on $N_{Q}/H_{N_{Q}}$.
Accordingly, we find that (\ref{e: rewrite I two}) may be rewritten as
\begin{equation}
\label{e: second rewrite I two}
J_y(l_k^*\psi \, \Delta_{\scriptscriptstyle G/Q}  )
=
\int_{N_{Q}/H_{N_Q}}
        \psi(kyn)\,\Delta_{\scriptscriptstyle G/Q}(yn) \, dn
\end{equation}
Combining (\ref{e: second integral psi omega}), (\ref{e: second rewrite I one}) and (\ref{e: second rewrite I two}),
we obtain  that
\begin{align*}
\int_{G/H_{Q}}\psi_{\omega}
&=\int_{K}\int_{M}\int_{A_{\fq}}\int_{N_{Q}/H_{N_{Q}}}\, \gD_{G/Q}(man)
\,\psi(kman)\,dn\,da\,dm\,dk\\
&=\int_{K}\int_{A_{\fq}}\int_{N_{Q}/H_{N_{Q}}}a^{2\rho_{Q}}\psi(kan)\,dn\,da\,dk.
\end{align*}
\end{proof}

\begin{proof}[Proof of Proposition \ref{Prop Ft_(Q,tau)=F_A circ H_Q}]
For each $v\in\cW$ let $\omega_{\scriptscriptstyle H/H_{Q^{v}}}\in\cD_{\fh/\fh_{Q^{v}}}$ be as in (\ref{eq omega_(G/H) otimes omega_(H/HQv) otimes omega_v=omega_(G/HL)}).
Let $\phi\in C_{c}^{\infty}(G/H:\tau)$ and  $\psi\in\oC(\tau). $ Then for $\lambda\in-\Upsilon_{Q}$,
\begin{align*}
&\langle\Ft_{Q,\tau}\phi(\lambda),\psi\rangle
=\int_{G/H}\langle\phi(x),\Etau(Q:\psi:-\bar{\lambda})(x)\rangle_\tau\,dx\\
&\qquad=\sum_{v\in\cW}\int_{G/H}\langle\phi(x),E_{\vH}(Q:\psi_{v}:-\bar{\lambda})(xv^{-1})\rangle_\tau\,dx\\
&\qquad=\sum_{v\in\cW}\int_{G/H}
    \left(\int_{H/H_{Q^{v}}}\;\langle\phi(x),\psi_{v,Q,-\bar{\lambda}}(xhv^{-1})\rangle_\tau
    \,dl_h(e)^{*-1}\,\omega_{\scriptscriptstyle H/H_{Q^{v}}}\right)
    \,dl_x(e)^{*-1}\,\omega_{\scriptscriptstyle G/H}.
\end{align*}
Here  $\psi_{v,Q,-\bar{\lambda}}$  is defined as in (\ref{e: defi psi v Q}).
We now apply Theorem \ref{Thm Fubini theorem for densities} to the term for $v$ in order
to rewrite the repeated integral as a single integral over $G/H_{Q^{v}}$ and obtain
$$
\langle\Ft_{Q,\tau}\phi(\lambda),\psi\rangle
 =  \sum_{v\in\cW}\int_{G/H_{Q^{v}}}\langle\phi(y),\psi_{v,Q,-\overline{\lambda}}(yv^{-1})\rangle_\tau
    \,dl_{y}(e)^{*-1}\,\omega_{\scriptscriptstyle G/H_{Q^{v}}}.
$$
By Lemma \ref{Lemma int_(H_Q / G)=int_K int_A_q int_(N/N cap H)} this expression is equal to
$$
\sum_{v\in\cW}
\int_{K}\int_{A_{\fq}}\int_{N_{Q^{v}}/H_{N_{Q^{v}}}}a^{2\rho_{Q^{v}}}
    \langle\phi(kan),\psi_{v,Q,-\overline{\lambda}}(kanv^{-1})\rangle_\tau
    \,dn\,da\,dk
$$
By  $\tau$-sphericality and unitarity of $\tau$ it follows that each integrand is independent of $k.$
Furthermore, by our chosen normalization of Haar measure, $dk(K) = 1$ so that the integral over $K$ can
be removed. By substituting $a^v:= v^{-1}  a v$ for $a$ and using the right $A N_{Q}$-equivariance of
$\psi_{Q,v, -\bar \gl},$ we thus find
\begin{align*}
\langle\Ft_{Q,\tau}\phi(\lambda),\psi\rangle
 & =
 \sum_{v\in\cW}
\int_{A_{\fq}}\int_{N_{Q^{v}}/H_{N_{Q^{v}}}}a^{2\rho_{Q}}
    \langle\phi(a^v  n), \tau(v)^{-1} \psi_{v,Q,-\overline{\lambda}}(a)\rangle_\tau
    \,dn\,da\,dk
     \\
    &=
     \sum_{v\in\cW}
\int_{A_{\fq}}\int_{N_{Q^{v}}/H_{N_{Q^{v}}}}
    a^{-\gl + \rho_Q - \rho_{Q,\fh}}\langle\phi(a^v n), \tau(v)^{-1} \psi_{v}(e)\rangle_\tau
    \,dn\,da\,dk
    \\
& =
\int_{A_{\fq}} a^{-\gl } \sum_{v \in \cW} \; \Ht_{Q^v}\left(\inp{\phi(\dotvar)}{\tau(v)^{-1} \psi(e) }_\tau\right)(a^v) \; da
\end{align*}
Using Lemma \ref{Lemma relation between H_(Q,tau) and H_Q}
we finally obtain
$$
\langle\Ft_{Q,\tau}\phi(\lambda),\psi\rangle =
\int_{A_{\fq}} a^{-\gl } \inp{\Ht_{Q,\tau} (\phi)(a)}{\psi} \; da.
$$
Since $\psi$ was arbitrary, the result follows.
\end{proof}

\subsection{Invariant differential operators}
In this section we assume that $P_{0}$ is a parabolic subgroup from $\cP_{\sigma}(A_{\fq})$
and write $P_{0}=M_{0}A_{0}N_{0}$ for its Langlands decomposition;
then  $A_{0}\subseteq A$ and $\faq = \fa_0 \cap \fq.$
Furthermore,  $M_{0}/M_{0}\cap H=M/M\cap H$ as homogeneous spaces for $M$, see \cite[Lemma 4.3]{vdBanKuit_EisensteinIntegrals}.
Accordingly,
\begin{equation}
\label{e: fg deco fnzero fl fh}
\fg  =\fn_{0} \oplus (\fl  + \fh),
\end{equation}
where $\fl = \fm \oplus \fa$ is the Lie algebra of $L = MA.$
Let $\DGH$ be the algebra of invariant differential operators
on $G/H.$ Then the right-regular representation of $G$ on $C^\infty(G)$
induces an isomorphism
\begin{equation}
\label{e: r and DGH}
r: U(\fg)^H/(U(\fg)^H \cap U(\fg)\fh) \;\;  {\buildrel\simeq\over \longrightarrow}\;\; \DGH,
\end{equation}
see \cite[Sect.\ 2]{vdBan_PrincipalSeriesII} for details. Let
$$
r_0: U(\fm_0)^{H_{M_0}}/(U(\fm_0)^{H_{M_0}} \cap \fh_{M_0})
 \;\;  {\buildrel\simeq\over \longrightarrow}\;\;
\D(M_0/M_0\cap H)
$$
be the similar isomorphism onto the algebra of left $M_0$-invariant differential
operators on $M_0/M_0\cap H.$ Let $\D(\Aq)$ denote the algebra
of bi-invariant differential operators on $\Aq.$ Then the right regular representation
induces an algebra isomorphism $U(\faq) = S(\faq) \simeq \D(\Aq).$
We define the canonical algebra embedding
$
\mu: \DGH \embeds  \D(M_0/M_0\cap H) \otimes \D(\Aq)
$
as in \cite[Sect.\ 2]{vdBan_PrincipalSeriesII}. It is independent of the choice of parabolic subgroup
$P_0.$ We will give a suitable description of $\mu$ in terms of $P_0,$
which is somewhat different from the one in \cite{vdBan_PrincipalSeriesII}.

To prepare for this, let $\fm_{0n}$ be the ideal of $\fm_0$ generated by $\fm_0 \cap \fa$
and let $M_{0n}$ be the corresponding analytic subgroup of $M_0.$ Then
$M_0 = MM_{0n}$ and $M_{0n}$ acts trivially on $M_0/M_0\cap H,$
see \cite[Lemma 4.3]{vdBanKuit_EisensteinIntegrals}.
Therefore, the inclusion $M \to M_0$ induces a natural isomorphism
\begin{equation}
\label{e: identification of DMH}
\D(M_0/M_0\cap H) \simeq \D(M/M\cap H),
\end{equation}
via which we shall identify their elements.
As before, the right regular representation induces an isomorphism
$
U(\fm)^{H_M} / U(\fm)^{H_M}\cap U(\fm)\fh_M \simeq \D(M/M\cap H).$ Furthermore,
since $\fm_{0n} \subseteq \fh,$ the inclusion $\fm \embeds \fm_0$ induces an isomorphism
$$
U(\fm)^{H_M} / U(\fm)^{H_M} \cap U(\fm)\fh_M \simeq U(\fm_0)^{H_{M_0}}/(U(\fm_0)^{H_{M_0}} \cap \fh_{M_0})
$$
which is compatible with $r_0$ and the identification (\ref{e: identification of DMH}).
Accordingly, we may view $\mu$ as an algebra embedding
$$
\mu:  \DGH \embeds \D(M/M\cap H) \otimes \D(\Aq).
$$
\begin{Rem}
\label{r: multiplication with exponential}
In the formulation of the following result, we will write $e^{\pm \rho_{P_0}}$ for the continuous linear endomorphism
of $C^\infty((M/M \cap H)\times \Aq)$ given by multiplication with the similarly denoted function
$e^{\pm \rho_{P_0}}: (m,a) \mapsto a^{\pm \rho_{P_0}}.$
\end{Rem}

\begin{Lemma}
\label{l: char of mu}
Let $D \in \DGH$ and let $D_0 \in \D(M/M\cap H) \otimes  \D(\Aq)$ be the element
determined by $\mu(D) = e^{-\rho_{P_0}} \after D_0 \after e^{\rho_{P_0}},$ see
Remark \ref{r: multiplication with exponential}.
Let $u \in U(\fg)^H$ be a representative of $D$ and let
$u_0 \in U(\fm)^{H_M} \otimes\D(\Aq)$ be a representative of $D_0.$ Then
\begin{enumerate}
\itema $u - u_0 \in \fn_{P_0} U(\fg) \oplus U(\fg)\fh.$
\itemb Furthermore, if $Q \in \cP(A)$ satisfies $\gS(Q, \gs \Cartan) \subseteq \gS(P_0),$
then
$$
u - u_0 \in (\fn_{P_0}\cap \fn_Q )U(\fg) + U(\fg) \fh.
$$
\end{enumerate}
\end{Lemma}

\begin{proof}
We start with (a).
Note that $U(\fm)^{H_M} \subseteq U(\fm_0)^{H_{M_0}} + U(\fm_0)\fh_{M_0}.$
Thus, if $v_0$ is a representative for $D_0$ in $U(\fm_0)^{H_{M_0}} \otimes \D(\Aq),$
then $u_0 - v_0 \in U(\fg)\fh$ and it suffices to prove the assertion (a) with $v_0$
in place of $u_0.$ The resulting assertion immediately follows from the definition of $\mu$ in
\cite[Sect.~2]{vdBan_PrincipalSeriesII}.

We turn to (b). In view of  (\ref{e: fg deco fnzero fl fh}) and the PBW theorem,
it suffices to show that the image $u_1$ of $u - u_0$ in $\fn_{P_0} U(\fn_{P_0}) \otimes U(\fl) / U(\fl)\fh_L$
belongs to $(\fn_{P_0} \cap \fn_{Q})U(\fn_{P_0})\otimes U(\fl)/U(\fl)\fh_L. $
The element $u_1$ is invariant under $\ad(\fa_\fh),$ as both $u$ and $u_0$ are.
Since $\fah$ centralizes $\fl,$ we have
$$
u_1 \in [\fn_{P_0} U(\fn_{P_0}) ]^{\fah} \otimes U(\fl)/U(\fl)\fh_L.
$$
By the PBW theorem we have the following direct sum decomposition into $\ad(\fa_\fh)$-invariant subspaces,
$$
\fn_{P_0} U(\fn_{P_0}) = (\fn_{P_0} \cap \fn_Q)U(\fn_{P_0}) \oplus (\fn_{P_0} \cap \bar \fn_Q)U(\fn_{P_0} \cap \bar \fn_Q).
$$
The $\fah$-weights of the second summand are all of the form $\mu = \ga_1 + \cdots + \ga_k,$ with $k \geq 1$
and $\ga_j \in \gS(P_0) \cap \gS(\bar Q).$ The latter set is contained in $\gS(\bar Q, \gs),$ because
$\gS(Q, \gs \Cartan) \subseteq \gS(P_0).$ Let $X \in \fa^+(\bar Q).$ Then it follows that the roots
of $\gS(\bar Q, \gs)$ are positive on the element $Y = X + \gs(X)$ of $\fah.$ Hence $\mu(Y) > 0;$
in particular $\mu \neq 0.$ We thus see that
$$
\left[\fn_{P_0} U(\fn_{P_0})\right]^{\fah} = \left[(\fn_{P_0} \cap \fn_Q) U(\fn_{P_0})\right]^{\fah}.
$$
The result follows.
\end{proof}

\begin{Rem}
In view of the PBW theorem, the map  $\mu$ is entirely determined
either by the description in (a), or by the description in (b). For $\fh$-extreme
$Q$ the proof of (b) is basically a reformulation of the argument given in
the proof of \cite[Lemma 2.4]{AndersenFlenstedJensenSchlichtkrull_CuspidalDiscreteSerieseForSemisimpleSymmetricSpaces}.
\end{Rem}

Let $w\in\cW$ (see the definition preceding Lemma \ref{l: about NorKfaq}).
Then $\Ad(w)$ preserves $\fm$ and $\fa_{\fq}$. The action of $\Ad(w)$ on $\fm$ and $\fa_{\fq}$ induces an isomorphism of algebras
$$
\Ad(w):\D(M/H_M)\otimes\D(A_{\fq})\to\D(M/wH_Mw^{-1})\otimes\D(A_{\fq}).
$$
Accordingly, we define the algebra embedding
$$
\mu_{w}:\D(G/H)\to\D(M/wH_Mw^{-1})\otimes \, \D(A_{\fq})
$$
by
$$
\mu_{w}:=\Ad(w)\circ\mu.
$$

Let $(\tau,V_{\tau})$ be a finite dimensional unitary representation of $K.$
For each $w\in\cW$ the  natural action of $\D(M/wH_Mw^{-1})$ on $C^{\infty}(M/wH_Mw^{-1}:\tau_{M}^{0})$ induces an algebra homomorphism
$$
r_{w}:\D(M/wH_Mw^{-1})\to\End\big(C^{\infty}(M/wH_Mw^{-1}:\tau_{M}^{0})\big)
$$
In the following we will view
$\End(\oC(\tau)) \otimes \D(\Aq)$ as the algebra of invariant
differential operators with coefficients in $\End(\oC(\tau)),$
which naturally acts on $C^\infty(\Aq) \otimes \oC(\tau).$
Accordingly, we define the algebra homomorphism
\begin{equation}
\label{e: intro mu tau}
\mu(\dotvar:\tau) :\D(G/H)\to  \End\big(\oC(\tau)\big) \otimes \D(\Aq)
\end{equation}
by
$$
\big( \mu(D:\tau ) \Psi \big)_{w}
=[(r_{w}\otimes I)\after \mu_{w}(D)]\Psi_{w}
\qquad\big(\Psi\in C^\infty(\Aq) \otimes \oC(\tau), w\in \cW\big),
$$
for $D \in \DGH.$
\begin{Prop}\label{Prop H(D phi)=mu(D) H phi}
Let $Q\in\cP(A).$
If $D\in\D(G/H)$ and $\phi\in C_{c}^{\infty}(G/H:\tau)$ then
\begin{equation}\label{eq H(D phi)=mu(D)H phi}
\Ht_{Q,\tau}(\phi)
= \mu (D:\tau) \Ht_{Q,\tau}\phi.
\end{equation}
\end{Prop}

\begin{proof}
Let $v \in \cW.$
Fix  $P_{0} \in \cP_\gs(\Aq)$ such that $\gS(Q,\gs\Cartan) \subseteq \gS(P_0).$
In view of Lemma \ref{l: Q  P zero and P} there exists a unique $P \in \cP_\gs(A)$
such that $Q \preceq P \subseteq P_0.$
Then $\gS(Q^v, \gs\Cartan ) \subseteq \gS(P^v,\gs\Cartan) = \gS(P_0^v).$
Let $D \in \DGH$ and let $u$ and $u_0$ be associated with $D$ as in Lemma
\ref{l: char of mu}, but with $P_0^v, Q^v$ in place of $P_0, Q.$ Then
$$
u - u_0 \in (\fn_{P_0^v} \cap   \fn_{Q^v})U(\fg) + U(\fg)\fh
$$
and
$$
\mu(D) = d_v^{-1} \after D_0 \after d_v,
$$
where $D_0 = R_{u_0}$ and $d_v(a) = a^{v^{-1}\rho_{P_0}},$ for $a \in \Aq;$
see Remark \ref{r: multiplication with exponential}.

Let  $X\in\faq$ be such that $\alpha(X)>0$ for every $\alpha\in\Sigma(P_{0}^v)$. Then $X$ satisfies (\ref{eq condition on X})
for the pair $(P^v, Q^v).$
By Lemma \ref{l: deco NPX and NQX} we infer that
$N_{Q^v,X}=N_{Q^v}\cap N_{P^v} = N_{Q^v} \cap N_{P_0^v}.$
By Definition \ref{d: tau HCT} and
Corollary \ref{Cor int_N/(N cap H)=int_N_Q,X} it now follows that
$$
\big(\Ht_{Q,\tau}\phi(a)\big)_{v}(m)
=a^{\rho_{Q}-\rho_{Q,\fh}}\int_{N_{Q^{v}}\cap N_{P_{0}^{v}}}\phi(mavn)\,dn,
$$
for all $\phi\in C^\infty_c(G/H:\tau),$ $m \in M$ and $a \in \Aq.$
In the integral on the right-hand side, the function
$\phi$ should be viewed as a function in $C^{\infty}(G)\otimes V_{\tau}$
of compact support modulo $H$, i.e., with support in $G$ that
has compact image in $G/H.$ Accordingly, we define
$$
T\phi(m,a):= a^{\rho_{Q}-\rho_{Q,\fh}}\int_{N_{Q^{v}}\cap N_{P_{0}^{v}}}\phi(mavn)\,dn
\qquad ((m,a) \in M\times \Aq),
$$
for any such function $\phi.$
Note that $T\phi \in C^\infty(M \times \Aq) \otimes \Vtau.$
It is readily verified that
$$
T(R_Z \phi)(m,a) = 0
$$
for $\phi \in C^\infty_c(G/H:\tau)$ and  $Z \in (\fn_{P_0^v} \cap \fn_{Q^v}) U(\fg) + U(\fg)\fh.$
Therefore,
\begin{equation}
\label{e: Ht D and T}
\big(\Ht_{Q,\tau}D\phi(a)\big)_{v}(m) =  T(R_{u_0}\phi)(m,a).
\end{equation}
For any function $\phi \in C^\infty(G:\tau)$ of compact support modulo $H$ we have
\begin{align*}
T(\phi)(m,a)
&=a^{\rho_{Q}-\rho_{Q,\fh}}\int_{N_{Q}\cap N_{P_{0}}}\phi(manv)\,dn\\
&=a^{\rho_{Q}-\rho_{Q,\fh}}\big|\det\Ad(a)|_{\fn_{Q}\cap\fn_{P_{0}}}\big|^{-1}
    \int_{N_{Q}\cap N_{P_{0}}}\phi(nmav)\,dn.
\end{align*}
for $(m,a) \in M \times \Aq.$ Since
\begin{align*}
a^{\rho_{Q}-\rho_{Q,\fh}}\big|\det\Ad(a)|_{\fn_{Q}\cap\fn_{P_{0}}}\big|^{-1}
&=\big|\det\Ad(a)|_{\fn_{Q}\cap \fn_{P_{0}}}\big|^{-\frac{1}{2}}
    \big|\det\Ad(a)|_{\fn_{Q}\cap \theta\fn_{P_{0}}}\big|^{\frac{1}{2}}\\
&=\big|\det\Ad(a)|_{\fn_{Q}\cap \fn_{P_{0}}}\big|^{-\frac{1}{2}}
    \big|\det\Ad(a)|_{\theta\fn_{Q}\cap \fn_{P_{0}}}\big|^{-\frac{1}{2}}\\
&=a^{-\rho_{P_{0}}},
\end{align*}
we infer, writing $d(a) = a^{\rho_{P_0}},$ that
$$
T\phi(m,a)
= d(a)^{-1}
\int_{N_Q \cap N_{P_0}}\phi(nv v^{-1} mav)\,dn.
$$
Let now $\phi \in C_c^\infty(G/H),$ so that $T\phi \in C^\infty(M/vH_Mv^{-1} \times \Aq) \otimes \Vtau.$ Then
\begin{eqnarray*}
T(R_{u_0}\phi)(m,a)
&=&
[d^{-1} \after R_{\Ad(v)u_0} \after d ](T\phi)(m,a)\\
&=& [d^{-1} \after \Ad(v)(R_{u_0}) \after d ](T\phi)(m,a)\\
&=& \Ad(v) [d_v^{-1} \after R_{u_0} \after d_v]  (T\phi)(m,a)\\
&=& \mu_v(D) (T\phi)(ma).
\end{eqnarray*}
In view of (\ref{e: Ht D and T}), we finally conclude that
\begin{eqnarray*}
\big(\Ht_{Q,\tau}\phi(a)\big)_{v}(m) & = &
\mu_v(D)(T\phi)(ma)\\
& = &
\big([(r_w \otimes I)\after \mu_v(D) ] (\Ht_{Q,\tau}\phi)_v(a)\big)(m) \\
&= &
[\mu(D:\tau)(\Ht_{Q,\tau}\phi)(a)]_v(m).
\end{eqnarray*}
\end{proof}

\section{Extension to the Schwartz space}\label{section Extension to Harish-Chandra-Schwartz functions}
Throughout this section, we assume that $Q\in\cP(A)$ and that $P_{0}$ is a minimal $\sigma\theta$-stable parabolic subgroup that contains $A$ and satisfies $\Sigma(Q,\sigma\theta)\subseteq \Sigma(P_{0})$,
see Lemma \ref{l: Q  P zero and P}.

We define
$$
\faq^{*+}(P_{0})
:=\{\lambda\in\faq^{*}:\langle\lambda,\alpha\rangle>0\;\;\forall\alpha\in\Sigma(P_{0})\}
$$
and
$$
A_{\fq}^{+}(P_{0}) =  \{ a \in A_{\fq}: a^\ga  > 1\;\;\forall \alpha\in\Sigma(P_{0})\}.
$$

\subsection{Tempered term of the $\tau$-spherical Harish-Chandra transform}\label{subsection Tempered Term}

Let $(\tau,V_{\tau})$ be a finite dimensional unitary representation of $K$ as before. It is convenient
to denote by $\Etau(Q:\dotvar)$ the meromorphic map $\faqc^{*}\to\Hom\big(\cA_{M,2},C^{\infty}(G/H:\tau)\big)$ given by $$
\Etau(Q:\lambda)\psi
=\Etau(Q:\psi:\lambda)\qquad\big(\lambda\in\faqc^{*},\psi\in\cA_{M,2}(\tau)\big).
$$
By Proposition \ref{Prop descr of singularities of E(Q:psi:.) and Paley-Wiener estimate for F_(Q,tau)},
the singular locus of $\Etau(Q: - \dotvar )$
equals the union of a locally finite collection  $\HypQtau$
of hyperplanes of the form $\{\lambda\in\faqc^{*}:\langle\lambda,\alpha\rangle=c\}$ with $\alpha\in\Sigma\setminus\fa_{\fh}^{*}$ and $c\in\R.$
Each such hyperplane $H$ can
be written as $H:= \mu + \ga^\perp_\iC,$ where $\mu \in \faqd$ is real and where $\ga^\perp_\iC$ denotes the complexification of the real hyperplane $\ga^\perp \subset \faqd.$ We note
that $H_\R:= H \cap \faqd$ equals $\mu + \ga^\perp$ and that $H$ may be viewed
as the complexification of $H_\R.$
Moreover, we agree to write
$$
\HypRQtau: = \{ H_\R \colsep H \in \HypQtau\}.
$$
We note that for $\mu \in \faqdp(P_0) \setminus \cup \HypRQtau$ the function
$\Etau(Q: - \dotvar )$ is regular on $\mu + i\faqd.$ Furthermore, if
$\phi\in C_{c}^{\infty}(G/H:\tau)$
then from the Paley-Wiener type estimate (\ref{eq Paley-Wiener estimate for F_(Q,tau)}) in
Proposition \ref{Prop descr of singularities of E(Q:psi:.) and Paley-Wiener estimate for F_(Q,tau)} we infer that
$$
\lambda\mapsto \Ft_{Q,\tau}\phi(\lambda)\,a^{\lambda}
$$
is integrable on $\mu+i\faq^{*}$ for every $a\in A_{\fq}$.
In view of  the estimates in the same proposition, it follows by application of Cauchy's integral formula  that the map
$$
\faq^{*+}(P_{0})\setminus  \cup \HypRQtau  \; \ni
\mu\mapsto   \int_{\mu+i\faqd}
             \Ft_{Q,\tau}\phi(\lambda)a^{\lambda}\,d\lambda
$$
is locally constant, hence constant on each connected component of
$\faq^{*+}(P_{0})\setminus \cup \HypRQtau$.
Here $d\gl$ denotes the choice of (real) measure on $\mu + i\faqd$ obtained by
transferring $(2\pi)^{-\dim \faq}$ times the Lebesgue
measure on $\faqd$ under the map $\gl \mapsto \mu + i \gl.$

Since $E(Q:- \dotvar)$ is holomorphic on an open neighborhood of the
closed convex set $-\Upsilon_Q,$
see (\ref{eq def Upsilon_Q}), it follows that there exists a connected component $C_1$
of $\faq^{*+}(P_{0})\setminus \cup \HypRQtau$ such that
$$
C_1 \supseteq \faq^{*+}(P_0) \cap (-\Upsilon_Q).
$$
\begin{Lemma}
\label{l: HC transform as contour integral}
Let $\mu \in C_1.$ Then
\begin{equation}
\label{e: HC transform as contour integral}
\Ht_{Q,\tau} (\phi)(a)  =
\int_{\mu+i\faq^{*}}\Ft_{Q,\tau}\phi(\lambda)a^{\lambda}\,d\lambda  \qquad (a \in \Aq)
\end{equation}
\end{Lemma}

\begin{proof}
As the expression on the right-hand side of the equation is independent of $\mu \in C_1,$
we may
assume that $\mu \in - \Upsilon_Q.$ Then in view of
Proposition \ref{Prop Ft_(Q,tau)=F_A circ H_Q},
$$
\int_{\mu+i\faq^{*}}\Ft_{Q,\tau}\phi(\lambda)a^{\lambda}\,d\lambda
=
\int_{i\faq^*} \eFt_{\Aq}(\Ht_{Q,\tau}(\phi))(\mu + \gl)  a^\mu  a^{\lambda}\, d\gl
= \Ht_{Q,\tau}(a)
$$
where the latter equality is valid by application of the Fourier inversion formula.
\end{proof}

We intend to analyze $\Ht_{Q,\tau}(\phi)$ by applying a contour shift
to the integral on the right-hand side of (\ref{e: HC transform as contour integral})
with $\mu$ tending  to zero in a suitable way.
This will  result in residual terms. In the $\gs$-split rank one case, these are point residues,
which will be
analyzed in the next section. For general $\gs$-split rank, one may hope to analyze them by
using a multi-dimensional residue calculus in the spirit of
\cite{vdBanSchlichtkrull_ResidueCalculus}.

It is readily  seen that there exists a unique connected component $C_0$
of $\faq^{*+}(P_0)\setminus \HypRQtau$
with
$0 \in \overline{C}_0$.
For $\phi\in C_{c}^{\infty}(G/H:\tau)$ we define $\It_{Q,\tau}\phi:A_{\fq}\to\oC(\tau)$ by
$$
\It_{Q,\tau}\phi(a)
=\lim_{\epsilon\downarrow0}   \int_{\epsilon\nu+i\faq^{*}}
    \Ft_{Q,\tau}\phi(\lambda)a^{\lambda}\,d\lambda\qquad(a\in A_{\fq}).
$$
Here $\nu$ is any choice of element of $\faq^{*+}(P_{0})$; the definition is independent
of this choice and
\begin{equation}
\label{e: defi It Q tau}
\It_{Q,\tau}\phi(a)
=\int_{\mu+i\faq^{*}}\Ft_{Q,\tau}\phi(\lambda)a^{\lambda}\,d\lambda
\end{equation}
for $\mu\in C_0$.  The function
$\It_{Q,\tau}\phi:$ $\Aq \to \oC(\tau)$ will be called the tempered term
of the Harish-Chandra transform.

We define
\begin{equation}
\label{e: temp C infty Aq}
C^{\infty}_{\temp}\big(A_{\fq}\big)
\end{equation}
to be the space of
 smooth functions on $A_{\fq}$ which are tempered as distributions on $\Aq,$
viewed as
a Euclidean space, i.e., belong
to the dual of the Euclidean Schwartz space $\cS(\Aq).$
We equip the space (\ref{e: temp C infty Aq})
with the coarsest locally convex topology such that the
inclusion maps into $C^{\infty}(A_{\fq})$ and $\cS'(A_{\fq})$ are both continuous.
Here $C^\infty(\Aq)$ and $\cS(\Aq)$
are equipped with the usual Fr\'echet topologies and $\cS'(\Aq)$ is equipped
with the strong dual topology.

\begin{Prop}
\label{Prop JQ: C(G/H:tau) to C^infty(Aq)}
If $\phi\in C_{c}^{\infty}(G/H:\tau)$, then
$\It_{Q,\tau}\phi\in C^{\infty}_{\temp}(A_{\fq})\otimes\oC(\tau)$.
The map $C_{c}^{\infty}(G/H:\tau)\to C^{\infty}_{\temp}(A_{\fq})\otimes\oC(\tau)$ thus obtained has a
unique extension
to a continuous linear map
$$
\It_{Q,\tau}:\cC(G/H:\tau)\to C^{\infty}_{\temp}(A_{\fq})\otimes\oC(\tau).
$$
\end{Prop}

\begin{proof}
Let $B\subseteq  \faq^{*}$ be a bounded neighborhood of $0$.
Let $p\in\Pi_{\Sigma,\R}(\faqd)$
be as in Proposition \ref{Prop relation E(Q) and E^circ(cP0)} (c).
Then $p(-\dotvar)$ belongs to
$\Pi_{\Sigma,\R}(\faqd),$ hence admits a
decomposition as a product of  a polynomial from $\Pi_{\gS, \R}(\faqd)$
which vanishes nowhere on
$i \faqd$ and  a polynomial $p_h \in \Pi_{\Sigma, \R}(\faqd)$
which is homogeneous. Then
$\lambda\mapsto p_{h}(-\lambda) C_{\bar P_0 | Q}(1 : \lambda)$
is holomorphic on an
open neighborhood of
$i \faqd$ in $\faqdc.$

According to \cite[Lemma 6.2]{vdBanSchlichtkrull_MostContinuousPart}
the Fourier transform
$\nFt_{\bar P_{0}}$
 extends to a continuous linear
 map from $\cC(G/H:\tau)$ to $\cS(i\faq^{*})\otimes\oC(\tau)$.
Hence, in view of  Proposition \ref{Prop F_Q phi=c F^0_cP0 phi},
also the map $\phi\mapsto p_{h}\Ft_{Q,\tau}\phi$ extends to a continuous linear map
 $\cC(G/H: \tau )\to\cS(i\faq^{*})\otimes\oC(\tau)$
 and  for all $\phi \in \cC(G/H:\tau)$ we have
 \begin{equation}
 \label{e: FQtau and FbarPzero}
 [p_{h} \Ft_{Q,\tau}\phi](\gl)  = p_{h}(\lambda) C_{\bar P_0 : Q}(1 : - \bar\lambda)^*\,\nFt_{\bar P_0}(\phi)(\gl),
 \qquad (\gl \in i\faqd).
\end{equation}
We now see that, for $\phi\in\cC(G/H:\tau),$
\begin{equation}\label{eq def K_Q}
\Kt_{Q,\tau}\phi(a)
:=
\int_{i\faq^{*}}
    p_{h}(\lambda)\Ft_{Q,\tau}\phi(\lambda)a^{\lambda}\,d\lambda
    \qquad(a\in A_{\fq})
\end{equation}
defines an element of $\cS(A_{\fq})\otimes \oC(\tau)$ and the map
$$
\Kt_{Q,\tau}:\cC(G/H:\tau)\to\cS(A_{\fq})\otimes\oC(\tau)
$$
thus obtained is continuous linear.

Let $\nu\in\faq^{*+}(P_{0})$.
It follows from \cite[Thm.~3.1.15]{Hormander_AnalysisOfPDOs-I} that the limit
\begin{equation}
\label{e: defi distribution v}
v(f)
:=\lim_{\epsilon\downarrow0}
    \int_{i\faq^{*}}\frac{f(\lambda)}{p_{h}(\lambda+\epsilon \nu)}\,d\lambda
\end{equation}
exists for every $f\in \cS(i\faq^{*}),$  and that accordingly $v$ defines a
distribution on $i\faq^{*}$. This distribution is
homogeneous, hence tempered, see \cite[Thm.~7.1.18]{Hormander_AnalysisOfPDOs-I}.
Put $u:=\eFt_{\!\!A_{\fq}}^{-1}v$.
Then $u$ is a tempered distribution on $\Aq,$ hence
the convolution operator $f \mapsto u* f$ defines a continuous linear map
$\cS(A_{\fq})\to C_{\temp}^{\infty}(A_{\fq})$.
Thus, to finish the proof, it suffices to prove the claim that
for every $\phi\in C_{c}^{\infty}(G/H:\tau),$
\begin{equation}\label{eq J_Q= u*K_Q}
\It_{Q,\tau}\phi
=u*\Kt_{Q,\tau}\phi.
\end{equation}

We set
$\Phi := p_{h}\Ft_{Q,\tau}\phi$
and note that
$\Kt_{Q,\tau}\phi=\eFt_{\!\!A_{\fq}}^{-1}\Phi$.
Therefore,
$$
u*\Kt_{Q,\tau}\phi
=(\eFt_{\!\!A_{\fq}}^{-1}v) *  \eFt_{\!\!A_{\fq}}^{-1}\Phi
=\eFt_{\!\!A_{\fq}}^{-1}\big(\Phi v\big).
$$
Let $\psi\in C^{\infty}_{c}(i\faq^{*})$. Then
\begin{align*}
\Phi v(\psi)
&=\lim_{\epsilon\downarrow0}\int_{i\faq^{*}}
 \frac{\Phi(\lambda)\psi(\lambda)}{p_{h}(\lambda+\epsilon\nu)}\,d\lambda\\
&=\lim_{\epsilon\downarrow0}\Bigg(
    \int_{i\faq^{*}}
           \frac{\Phi(\lambda+\epsilon\nu)\psi(\lambda)}{p_{h}(\lambda+\epsilon\nu)}\,d\lambda
    -\int_{i\faq^{*}}\Big(\int_{0}^{\epsilon}\partial_{t}\Phi(\lambda+t\nu)\,dt\Big)
        \frac{\psi(\lambda)}{p_{h}(\lambda+\epsilon\nu)}\,d\lambda
    \Bigg)\\
&=\lim_{\epsilon\downarrow0}\Bigg(
    \int_{i\faq^{*}}\Ft_{Q,\tau}\phi(\lambda+\epsilon\nu)\psi(\lambda)\,d\lambda
    -\int_{0}^{\epsilon}\int_{i\faq^{*}}
        \frac{\partial_{t}\Phi(\lambda+t\nu)\psi(\lambda)}{p_{h}(\lambda+\epsilon\nu)}\,d\lambda\,dt
    \Bigg).
\end{align*}
The function
$$
F:   \;\; (t,\epsilon)\mapsto
    \int_{i\faq^{*}}
    \frac{\partial_{t}\Phi(\lambda+t\nu)\psi(\lambda)}{p_{h}(\lambda+\epsilon\nu)}\,d\lambda
$$
is continuous on $[0,1]\times \,]0,1]. $ Moreover, since
$
f:  t \mapsto \partial_t\Phi(\dotvar +t \nu) \psi(\dotvar )
$
is a continuous function $[0,1] \to C_c^\infty(i\faqd),$ it follows that in
the Banach space $C([0,1]),$
$$
F(\dotvar, \epsilon) \to v(f(\dotvar)) \qquad (\epsilon \downarrow 0 ).
$$
We thus see that $F$  extends continuously to $[0,1]\times[0,1]$. This in turn implies that
$$
\lim_{\epsilon\downarrow0}
    \int_{0}^{\epsilon} F(t, \epsilon) \, dt  = 0,
$$
hence
$$
\Phi v(\psi)
=\lim_{\epsilon\downarrow0}
    \int_{i\faq^{*}}\Ft_{Q,\tau}\phi(\lambda+\epsilon\nu)\psi(\lambda)\,d\lambda.
$$
Now let $\chi\in C_{c}^{\infty}(A_{\fq})$. Then
\begin{align*}
\big(u*\Kt_{Q,\tau}\phi\big)(\chi)
&=\Phi v\Big(\lambda\mapsto\int_{A_{\fq}}\chi(a)a^{\lambda}\,da\Big)\\
&=\lim_{\epsilon\downarrow0}
     \int_{i\faq^{*}}\int_{A_{\fq}}\Ft_{Q,\tau}\phi(\lambda+\epsilon\nu)\chi(a)a^{\lambda}\,da\,d\lambda\\
&=\lim_{\epsilon\downarrow0}
     \int_{A_{\fq}}\Bigg(\int_{i\faq^{*}}\Ft_{Q,\tau}\phi(\lambda+\epsilon\nu)a^{\lambda+\epsilon\nu}
        \,d\lambda\Bigg)a^{-\epsilon\nu}\chi(a)\,da.
\end{align*}
If $\epsilon$ is sufficiently small, then
$$
\It_{Q,\tau}\phi(a)
=
\int_{i\faq^{*}}\Ft_{Q,\tau}\phi(\lambda+\epsilon\nu)a^{\lambda+\epsilon\nu}\,d\lambda,
$$
see (\ref{e: defi It Q tau}), hence
$$
\big(u*\Kt_{Q,\tau}\phi\big)(\chi)
= \lim_{\epsilon\downarrow 0} \int_{A_{\fq}}\It_{Q,\tau}\phi(a)\, a^{-\epsilon \nu} \chi(a) \,da
 =\int_{A_{\fq}}\It_{Q,\tau}\phi(a)\chi(a)\,da.
$$
This establishes the claim (\ref{eq J_Q= u*K_Q}).
\end{proof}

\begin{Rem}\label{r: Case absence of residues}
Assume that the Eisenstein integral $E(Q:-\dotvar) = E_\tau(Q:- \dotvar)$
is holomorphic on $\,]0,1]\cdot \xi$ for an element $\xi \in C_1.$
Then the chambers $C_0$ and $C_1$ are equal, and it follows
that $\Ht_{Q,\tau}\phi=\It_{Q,\tau}\phi$. In view of Proposition \ref{Prop JQ: C(G/H:tau) to C^infty(Aq)},  the spherical Harish-Chandra transform $\Ht_{Q,\tau}$
extends to a continuous linear map $\cC(G/H:\tau) \to C^\infty_{\temp}(\Aq) \otimes \oC(\tau).$

Now assume that the above condition of holomorphy is fulfilled for $(\tau, V_\tau)$ equal to the trivial representation $(1, \C)$ of $K.$ Then $\cC(G/H:\tau) = \cC(G/H)^K$ and $\oC(\tau) = \C^\cW$
and it follows by application of Lemma \ref{Lemma relation between H_(Q,tau) and H_Q} that the
restriction of  $\Ht_Q$  to $C_c^\infty(G/H)^K$ extends to a continuous linear map
$\cC(G/H)^{K}\to C^{\infty}(L/H_{L})^{M}.$
By application of Proposition \ref{Prop If H_Q extends to C(G/H)^K, then R_Q extends to C(G/H) with conv integrals}
it now follows that $\Ht_{Q}$ extends to a continuous linear map $\cC(G/H)\to C^{\infty}(L/H_{L})$ and is given by absolutely convergent integrals. Moreover, the image of $\Ht_{Q}$ consists of tempered functions.

Finally, assume that $\Sigma_{-}(Q)=\emptyset.$  Then $\Upsilon_{Q}=\faqc^{*}$. This implies that $\Ft_{Q,\tau}\phi$ is holomorphic on $\faqc^{*}$ for every $\phi\in C_{c}^{\infty}(G/H:\tau)$.
Now a stronger statement can be obtained than in the above more general setting.
 The polynomial $p$ in the proof for Proposition \ref{Prop JQ: C(G/H:tau) to C^infty(Aq)} can be taken equal to the constant function $1$. The distribution $u$ is then equal to the Dirac measure at the origin of $i\faq^{*}$ and as a consequence, $\It_{Q,\tau}$ is equal to $\Kt_{Q,\tau}$.
In particular, it follows that $\Ht_{Q,1}$ extends to a continuous linear map
$\cC(G/H)^K \to \cC(L/H_L)^M$ and is given by absolutely convergent integrals.
In view of Lemma \ref{Lemma relation between H_(Q,tau) and H_Q} it follows that
$\Ht_Q$ maps $\cC(G/H)^K$ continuous linearly into $\cC(L/H_L)^M.$

We will now apply domination to show that in this case $\Ht_Q$
is a continuous linear map from $\cC(G/H)$ to $\cC(L/H_L).$ In the above we established already
that for $\phi \in \cC(G/H)$ the function $\Ht_Q \phi \in C^\infty(L/H_L)$ is given
by absolutely convergent integrals.
For the purpose of estimation, let $\gf \mapsto \widehat\gf$ be a map as in
Proposition \ref{Prop domination by K-inv Schwartz functions}.  Let $u \in U(\fl).$ Then there exists
a $u' \in U(\fl)$ such that $L_u \after \gd_Q = \gd_Q\after L_{u'}$
on $C^\infty(L/H_L).$ Thus, for $\phi \in \cC(G/H)$ we have
\begin{equation}
\label{e: L u on Ht Q}
L_u \Ht_Q(\phi) = \gd_Q L_{u'} \Rt_Q(\phi) = \Ht_Q(L_{u'}\phi),
\end{equation}
by equivariance of the Radon transform. Let $N \in  \N.$
There exists a continuous seminorm $\nu$ on $\cC(G/H)$ such that for all
$\phi \in \cC(G/H)^K$ and $l \in L,$
\begin{equation}
\label{e: estimate HQ f}
(1 + \|l\|)^N |\Ht_Q(\phi)(l) | \leq \nu(\phi).
\end{equation}
It now follows by application of
Proposition \ref{Prop domination by K-inv Schwartz functions} that there exists a continuous seminorm $\mu$ on $\cC(G/H)$ such that

\begin{equation}
\label{e: nu of widehat phi}
\qquad \nu(\widehat \phi ) \leq \mu(\phi) \qquad\qquad  (\phi \in \cC(G/H)).
\end{equation}
Combining the equality (\ref{e: L u on Ht Q}) with the estimates
(\ref{e: estimate HQ f}) and (\ref{e: nu of widehat phi}), we find
\begin{eqnarray*}
(1 + \|l\|)^N |L_u \Ht_Q(\phi)(l) |  & \leq  & ( 1 + \|l\|)^N  \Ht_Q(|L_{u'} \phi|) \\
& \leq  & ( 1 + \|l\|)^N  \Ht_Q(\widehat{ L_{u'} \phi} ) \\
& \leq & \nu( \widehat{L_{u'} \phi}) \leq \mu(L_{u'} \phi).
\end{eqnarray*}
This establishes the continuity.
\end{Rem}

\begin{Ex}[\bf Group case]
\label{ex: Whittaker transform}
We use the notation of Example \ref{Ex Group case - HC-transform def compared to HC's def}. Assume that $\bp P$ and $\bp Q$ are minimal parabolic
subgroups of $\bp G$ containing $\bp A.$
Since $\Sigma_{-}(\bp P\times \bp P)=\emptyset$, the final analysis
in Remark \ref{r: Case absence of residues} applies to $\Ht_{\bp P\times\bp P}$.
Let $\bp \xi \in \bp \widehat M$ and define $\xi \in \widehat M$ by $\xi := \bp \xi \otimes \bp \xi^\vee.$
For $\bp \gl \in \bp \fadc$ we set $\gl =  (\bp \gl, - \bp \gl) \in \faqdc.$
Let $(\tau , \Vtau)$ be a finite dimensional unitary representation
of $K = \bp K \times \bp K.$
We recall from \cite[Eqn.~(8.16)]{vdBanKuit_EisensteinIntegrals}
that the $C$-function $C(\bp Q \times \bp Q : \bp P \times \bp Q :\gl) \in \End(\oC(\tau))$
is defined by the relation
\begin{equation}
\label{e: C QQ PQ}
E(\bp P \times \bp Q : \gl) = E(\bp Q \times \bp Q : \gl) \after C(\bp Q \times \bp Q : \bp P \times \bp Q :\gl).
\end{equation}
It follows from \cite[Cor. 9.6, 9.8]{vdBanKuit_EisensteinIntegrals} that
$$
C(\bp Q \times  \bp Q : \bp P \times \bp Q :\gl) \psi_{f \otimes I_{\bp \xi}} =
\psi_{[A(\bp Q : \bp P : \bp \xi : - \bp \gl) \otimes I)f]\otimes I_{\bp \xi}},
$$
for $\xi \in \widehat M$ and $f \in C^\infty(K:\xi: \tau).$ The intertwining operator
$A(\bp Q : \bp P: \bp \xi : \bp \gl)$ depends holomorphically on $\gl = (\bp \gl, -\bp \bar \gl)$
in the region
$$
U: = \{ \gl \in \faqdc :   \inp{\bp \gl}{\ga} > 0\;\; \textnormal{for all} \;\;\ga \in \gS(\bp P) \cap \gS( \bp \bar Q)\;\}.
$$
Clearly, $U$ is contained in $\fa_{\fq\iC}^{*+}(\bp P \times \bp \bar P).$
It follows that the $C$-function in
(\ref{e: C QQ PQ}) depends holomorphically on $\gl \in -U.$ From (\ref{e: C QQ PQ}) it follows that
$$
\Ft_{\bp P\times\bp Q,\tau}\phi(\lambda)
=\ctau(\bp Q\times\bp Q:\bp P\times\bp Q:-\bar{\lambda})^{*}\circ\Ft_{\bp Q\times\bp Q,\tau}\phi(\lambda).
$$

Hence,  $\Ft_{\bp P\times\bp Q}\phi$ is holomorphic on $\faq^{*+}(\bp P\times\bp\bar{P})$ for every $\phi\in C_{c}^{\infty}(G/H:\tau)$.
It follows from Remark \ref{r: Case absence of residues} that $\Ht_{\bp P\times\bp Q}$ extends to $\cC(G/H)$ and is given by absolutely convergent integrals. Moreover, $\Ht_{\bp P\times\bp Q}$ maps $\cC(G/H)$ to the space of smooth tempered functions on $L/H_{L}$.

The convergence of the integrals for $\Ht_{\bp P \times \bp Q}$ also follows from combining
\cite[Thm.~7.2.1]{Wallach_RealReductiveGroupsI} and \cite[Lemma 15.3.2]{Wallach_RealReductiveGroupsII}.

\begin{Rem}
\label{r: error in Wallach}
We should inform the reader that
\cite[Lemma 15.3.2]{Wallach_RealReductiveGroupsII} has an additional assertion that a certain
transform $f^P$ is of Schwartz behavior. However, the proof of this assertion is not correct.
In fact, in the right-hand side of the inequality at the top of page 377,
a factor $(1 + \|\log (aa_1)\|)^{2d}$ is missing.

From the given proof it can be concluded
that the map $f \mapsto f^{P}$ is given by absolutely convergent integrals, that it is continuous from
$\cC(N_0\bs G: \chi)$ to $C^\infty(N_0\cap M_P\bs M_P: \chi|_{N_0\cap M_P}),$ and that its
image consists of tempered smooth functions.
However, the second part of Lemma 15.3.2
of \cite{Wallach_RealReductiveGroupsII} cannot be true in the generality stated. Indeed, combined  with \cite[Thm.\  7.2.1]{Wallach_RealReductiveGroupsI} the validity of the lemma would imply that
 $\Ht_{\bp P\times \bp\bar{P}}$ maps $\cC(G/H)$ to $\cC(L/H_{L})$. The latter assertion
is already incorrect for  $\bp G=\mathrm{SL}(2,\R).$
This is established in the result below.
\end{Rem}

\begin{Lemma}
\label{l: counterexample wallach sl2}
Let $\bp G=\mathbf{SL}(2,\R)$ and let $\phi\in \cC(G/H)$. Assume that $\phi\geq0$ and
\begin{equation}\label{e: positive value phi}
\phi\left(\left(\scriptstyle{
      \begin{array}{cc}
        0 & -1 \\
        1 & 0 \\
      \end{array}}
    \right),e\right)
>0.
\end{equation}
Then
$$
\liminf_{\stackrel{a\to\infty}{a\in A_{\fq}^{+}(\bp P\times \bp \bar P)}}
    \Ht_{\bp P\times \bp \bar P}\,\phi(a)>0
$$
\end{Lemma}

\begin{proof}
For $t>0$ and $x,y\in\R$, we define
$$
a_{t}=\left(
        \begin{array}{cc}
          e^{t} & 0 \\
          0 & e^{-t} \\
        \end{array}
      \right),\;\;\;
n_{x}=\left(
        \begin{array}{cc}
          1 & x \\
          0 & 1 \\
        \end{array}
      \right),\;\;\;
\overline{n}_{y}=\left(
        \begin{array}{cc}
          1 & 0 \\
          -y & 1 \\
        \end{array}
      \right).
$$
To shorten notation, we write
$$
I(t):=\Ht_{\bp P\times \bp \bar P}\phi\big(a_{\nicefrac{t}{2}},a_{-\nicefrac{t}{2}}\big).
$$
Using the identification $G/H \buildrel \simeq\over\to  \bp G$ induced by $(x,y) \mapsto xy^{-1}$ to view $\phi$ as a function on $\bp G$ we obtain
$$
I(t)
=e^{t}\int_{\R}\int_{\R}\phi(a_{\nicefrac{t}{2}}n_{x}\overline{n}_{y}^{-1}a_{\nicefrac{t}{2}})\,dx\,dy
=e^{t}\int_{\R}\int_{\R}
    \phi\left(
        \begin{array}{cc}
          e^{t}(1+xy) & x \\
          y & e^{-t} \\
        \end{array}
      \right)\,dx\,dy.
$$
Let $0<\epsilon<1$ and $\eta>0$. We define the domain
$$
D_{t}
:=\{(x,y)\in\R^{2}:-\epsilon<e^{t}(1+xy)<\epsilon, 1< y< 1+\eta\}.
$$
Note that $(x,y)\in D_{t}$ if and only if
$$
1< y< 1+\eta\text{ and }\frac{-e^{-t}\epsilon-1}{y}<x<\frac{e^{-t}\epsilon-1}{y}.
$$
Hence, for $t>0,$
$$
D_{t}\subseteq R_{\epsilon, \eta}: = \big[-\epsilon-1,\frac{\epsilon-1}{1+\eta}\big]\times\big[1,1+\eta\big].
$$
We note that $R_{\epsilon, \eta}$ tends to $\{(-1, 1)\}$ for $\epsilon, \eta \downarrow 0$
and define
$$
C_{\epsilon,\eta}
:=\left\{\left(\scriptstyle{
      \begin{array}{cc}
        a & b \\
        c & d \\
      \end{array}}
    \right)\in\bp G:|a|\leq\epsilon, \; |d|\leq\epsilon, \; (b,c) \in R_{\epsilon, \eta} \right\}.
$$
Then $C_{\epsilon,\eta}$ is compact and tends to the singleton consisting of the matrix
in (\ref{e: positive value phi}). We may therefore
take $\epsilon$ and $\eta$ so close to zero that
the function $\phi$ is strictly positive on $C_{\epsilon,\eta}.$
We thus obtain, for $t>-\log\epsilon,$ that
\begin{align*}
I(t)
&
\geq e^{t}\int_1^{1+\eta}  \int_{\frac{-e^{-t}\epsilon-1}{y}}^{\frac{e^{-t}\epsilon-1}{y}}
    \phi\left(
        \begin{array}{cc}
          e^{t}(1+xy) & x \\
          y & e^{-t} \\
        \end{array}
      \right)\,dx\,dy\\
&\geq e^{t}\inf_{(u,v)\in D_{t}}
    \phi\left(
        \begin{array}{cc}
          e^{t}(1+uv) & u \\
          v & e^{-t} \\
        \end{array}
      \right)
\int_1^{1+\eta}\int_{\frac{-e^{-t}\epsilon-1}{y}}^{\frac{e^{-t}\epsilon-1}{y}}\,dx\,dy\\
&\geq 2\epsilon\,\log(1+\eta)\inf_{C_{\epsilon,\eta}}\phi \; >0.
\end{align*}
\end{proof}
\end{Ex}

\subsection{The residual term for spaces of split rank one}
\label{subsection Residues}
We retain the notation of the previous subsection. In particular, $Q \in \cP(A)$
and $P_0 \in \cP_\gs(\Aq)$ is such that $\gS(Q, \gs \Cartan) \subseteq \gS(P_0).$
As mentioned in the previous subsection,
the difference between $\Ht_{Q,\tau}\phi$ and $\It_{Q,\tau}\phi$ is equal to a finite sum of
residual integrals. These become point residues
in case $\dim \faq = 1.$ For the rest of this subsection, we make the
\medbreak\noindent
{\bf Assumption:\ } \em  $G/H$ is of split rank one, i.e., $\dim \faq =1.$
\medbreak\noindent\rm
By our assumption on the split rank, each hyperplanes from the set
$\HypQtau$
defined in the beginning
of Subsection \ref{subsection Tempered Term} consists of a single point in $\faqd.$
Furthermore, the union $\cup \HypQtau$ is a discrete subset of $\faqd,$
by
Proposition \ref{Prop descr of singularities of E(Q:psi:.) and Paley-Wiener estimate for F_(Q,tau)}.

We define
$$
S_{Q,\tau}: = \faqdp(P_0) \cap\big(\cup \HypQtau\big).
$$

\begin{Lemma}\label{Lemma S_(Q,tau) is real and finite}
The set $S_{Q,\tau}$ is finite.
\end{Lemma}

\begin{proof}
The Eisenstein integral $\Etau(Q:-\dotvar)$ is holomorphic on $-\widehat{\Upsilon}_{Q}$.
The latter set contains a set of the form $\xi + \faq^{*+}(P_{0})+i\faq^{*}$,
with $\xi \in \faq^*.$
Hence, $S_{Q,\tau}$ is contained in the set
$\faqdp(P_0)\setminus (\xi + \faqdp(P_0) + i\faqd)$ which is bounded.
Since $S_{Q,\tau}$ is discrete, the result follows.
\end{proof}

For a meromorphic function $f: \faqdc \to \C$ and a point $\mu \in \faqdc$ we define the
residue
\begin{equation}
\label{e: definition of residue}
\Res_{\gl = \mu}\; \gf(\gl) :=  \Res_{z = 0}\;  \gf(\mu + z \omega).
\end{equation}
Here $\omega$ is the unique vector in $\faq^{*+}(P_0)$ of unit length (relative to
the Killing form), $z$ is a variable in the complex plane, and the residue
on the right-hand side of (\ref{e: definition of residue}) is the usual residue from complex analysis, i.e., the coefficient
of $z^{-1}$ in the Laurent expansion of $z \mapsto \gf(\mu + z\omega)$ around $z = 0.$

\begin{Lemma}
\label{l: Ht minus Jphi as residues}
Let $\phi \in C_c(G/H:\tau).$ Then
$$
\Ht_{Q,\tau}\phi(a)
= \It_{Q,\tau}\phi(a)
+
\sum_{\mu \in S_{Q,\tau}}   \Res_{\gl = \mu} \; a^\gl\, \Ft_{Q,\tau}\phi(\gl)
$$
\end{Lemma}

\begin{proof}
By the chosen normalization of the measure $d\gl$  on $\mu + i\faqd,$
$$
\int_{\nu  \omega  + i\faqd}  \Ft_{Q,\tau}(\phi)(\gl) \; d\gl =  \frac{1}{2\pi i} \int_{ \nu  + i\R}
\Ft_{Q,\tau}(\phi)(\tau \omega) \; d\tau .
$$
In view of the estimates (\ref{eq Paley-Wiener estimate for F_(Q,tau)})
the result now follows by a straightforward application of the Cauchy integral formula.
\end{proof}

\begin{Lemma}
\label{l: residues one}
Let $\psi \in \oC(\tau)$ and $\mu \in S_{Q,\tau}$.
Then for all $\phi \in C_c(G/H:\tau),$
\begin{equation}
\label{e: residue of Ft Q}
\Res_{\gl = \mu} \;\inp{a^\gl  \,\Ft_{Q,\tau}(\phi)(\gl)}{\psi} = a^\mu \inp{ \phi} { \Restau (Q:  \mu: a:\dotvar )(\psi)},
\end{equation}
where $\Restau(Q:\mu)$ is the function $\Aq \times G/H \to \Hom(\oC(\tau), \Vtau)$ given by
$$
\Restau(Q:\mu:a:x) (\psi) =  - \Res_{\gl = -\mu} \big( a^{-\gl -\mu} E(Q:\psi:\gl)(x)\big).
$$
\end{Lemma}

\begin{proof}
First, assume that  $\Phi: \faqdc \to \C$ is a meromorphic function.
Then it is readily verified that
\begin{eqnarray*}
\Res_{\gl = \mu} a^\gl \overline{ \Phi(-\bar\gl)} &=&
a^\mu \Res_{z = 0} \overline{ a^{\bar z \omega} \Phi(- \mu - \bar z \omega) }\nonumber\\
&=&- a^\mu \overline{ \Res_{z = 0}  a^{- z \omega} \Phi(- \mu + z \omega) }\nonumber \\
&=&- a^\mu \overline {  \Res_{\gl = -  \mu} a^{-\gl -   \mu } \Phi(\gl)}
\end{eqnarray*}
By using conjugate linearity of the pairing $C^\infty_c(G/H:\tau) \times C^\infty(G/H:\tau) \to \C$
in the second factor, it now follows that the expression on the left-hand side of
(\ref{e: residue of Ft Q}) equals
$$
 \Res_{\gl = \mu}  a^\gl \,
 \inp{\phi}{E(Q:\psi:- \bar \gl)}  = - a^\mu \inp{\phi} {\Res_{\gl = - \mu} a^{-\gl -  \mu}
 E(Q:\psi: \gl)}.
 $$
 The latter expression equals the right-hand side of (\ref{e: residue of Ft Q}).
 It is clear that $\Restau(Q: \mu)$ is a function in
 $C^\infty(\Aq \times G/H) \otimes V_\tau \otimes \oC(\tau)^*,$ which is
 $\tau$-spherical in the second variable.
 \end{proof}

 The following result will be a useful tool for understanding the nature
 of the residues.
 We will use the notation $P_d(\faq)$ for the space of
 polynomial functions $\faqdc \to \C$ of degree at most $d,$
 and $P_d(\Aq)$ for the space of functions $\Aq \to \C$ of the form $a \mapsto p(\log a),$
 with $p \in P_d(\faq).$

\begin{Lemma}
\label{l: residue preparation}
Let $\mu \in S_{Q,\tau}$ and $\psi \in \oC(\tau).$ Then the following assertions hold.
\begin{enumerate}
\itema Let $\phi$ be a holomorphic function defined on a neighborhood
of $\mu$ such that the $C^\infty(G/H:\tau)$-valued function
$\gl \mapsto \phi(\gl) E(Q: \psi: \gl)$ is holomorphic at $\mu.$ Then
for every  $u \in S(\faqd),$
\begin{equation}
\label{e: derivative DGH finite}
\left.  \partial_u\right|_{\gl = \mu}  \phi(\gl) E(Q:\psi:\gl)
\end{equation}
is a $\DGH$-finite function in $C^\infty(G/H:\tau).$
\itemb
Let $\gf$ be a meromorphic function in a neighborhood of $\mu.$
Then
\begin{equation}
\label{e: res with meromorphic multiplier}
\Res_{\gl = \mu}\;  \gf(\gl) E(Q:\psi:\gl)
\end{equation}
is a $\DGH$-finite function in $C^\infty(G/H:\tau).$
\end{enumerate}
\end{Lemma}

\begin{proof}
Assertion (a) follows by applying the argument of the proof
of  \cite[Lemma 6.3]{vdBanSchlichtkrull_FourierInversionOnAReductiveSymmetricSpace}.

For (b) we may fix a non-trivial polynomial function $q$ such that $\phi: = q(\gl) \gf $
is holomorphic at $\mu$ and satisfies the hypothesis of (a). Then there exists $u \in S(\faqd)$
such that (\ref{e: res with meromorphic multiplier}) equals (\ref{e: derivative DGH finite}).
\end{proof}

\begin{Lemma}
\label{l: residues two}
Let $\mu\in S_{Q,\tau}$ and let $d\geq 0$ be the pole order of $\gl \mapsto E(Q:\gl)$ at $-\mu.$
Then there exists a finite dimensional subspace  $\findim \subseteq C^\infty(G/H:\tau),$ consisting of $\DGH$-finite functions, such that
$$
\Restau(Q:\mu) \in P_{d-1}(\Aq)\otimes \findim \otimes\oC(\tau)^*.
$$
\end{Lemma}

\begin{proof}
There exists a non-zero polynomial function $q$ on $\faqdc$ of degree
$d$ such that
the $C^\infty(G/H)^K \otimes \oC(\tau)^*$ valued meromorphic function
$\gl \mapsto q(\gl) E(Q: \gl)$ is regular at $-\mu.$
It follows that there exists a $u \in S(\faqd)$ of order at most $d-1$ such that
$$
\Restau(Q:\mu:a:x)\psi = \left. \partial_u\right|_{\gl = -\mu}
a^{-\gl-\mu } q(\gl) E(Q:\psi:\gl)(x)
$$
for all $a \in \Aq, x \in G/H$ and $\psi \in \oCtau.$
By application of the Leibniz rule we infer that there exist finitely many
$u_1 , \ldots , u_k \in S(\faqd)$  and $p_1, \ldots, p_k \in P_{d-1}(\faq),$
such that
$$
\Restau(Q:\mu:a: \dotvar) = \sum_{j=1}^k p_j(\log a) \left. \partial_{u_j}\right|_{\gl = -\mu}
q(\gl) E(Q:\gl),
$$
for all $a\in \Aq.$
The assertion now readily follows by application of Lemma \ref{l: residue preparation}.
\end{proof}

In the sequel, we will need the following version of Lemma \ref{l: Ht minus Jphi as residues}.

\begin{Cor}
Let $\phi \in C_c(G/H:\tau).$ Then,
for every $\psi \in \oC(\tau),$
\begin{equation}
\label{e: cor Ht minus Jphi as residues}
\inp{\Ht_{Q,\tau}\phi(a)}{\psi}
= \inp{\It_{Q,\tau}\phi(a)}{\psi}
+
\sum_{\mu \in S_{Q,\tau}}   a^\mu \inp{\phi}{\Restau(Q:\mu: a:\dotvar)\psi}.
\end{equation}

\end{Cor}
\begin{proof}
This follows from combining Lemma \ref{l: Ht minus Jphi as residues} with
(\ref{e: residue of Ft Q}).
\end{proof}

\subsection{$\fh$-compatible parabolic subgroups}
The residual terms in (\ref{e: cor Ht minus Jphi as residues})
will turn out to have a special relation to representations
of the discrete series if we select the parabolic subgroup $Q$ so that its
positive system $\gS(Q)$ satisfies a certain geometric restriction.

 Although we will only apply the definitions and results of the present  subsection
to symmetric spaces of split rank $1$, everything in this subsection
is in fact valid for symmetric spaces of higher split rank
 as well.

\begin{Defi}
\label{d: fh compatible}
A parabolic subgroup  $Q\in\cP(A)$ is said to be
$\fh$-compatible if
$$
\langle\alpha,\rho_{Q,\fh}\rangle\geq0\quad \textnormal{for all}\quad\alpha\in\Sigma(Q),
$$
see (\ref{e: defi rho 2x}) for the definition of $\rho_{Q,\fh}.$
We write $\cPH(A)$ for the subset of $\cP(A)$ consisting of all such
parabolic subgroups.
\end{Defi}

\begin{Lemma}{\ }
\label{l: H compatible parabolics}
\begin{enumerate}
\itema
If $P_0 \in \cP_\gs(\Aq)$ then there exists an $\fh$-extreme $Q \in \cPH(\Aq)$ such that
$\gS(Q, \gs \Cartan) \subseteq \gS(P_0).$
\itemb
If $Q \in \cPH(A),$ then $\gS(Q, \gs\Cartan) \perp \rho_{Q,\fh}.$
\itemc
If $P,Q \in \cP(A)$ and $P \succeq Q,$ then
$
P \in \cPH(A) \implies Q \in \cPH(A).
$
\end{enumerate}
\end{Lemma}

\begin{proof}
We start with (a). Let $P_0 \in \cP_\gs(\Aq).$ We fix $X$ in the associated positive
chamber $\faq^{*+}(P_0)$ and select
a positive system $\gS_\fh^+$ for the root system $\gS \cap \fahd.$
Put
$$
\rho_\fh : = \frac 12 \sum_{\ga \in \gS_\fh^+} m_\ga \ga.
$$
Then $\rho_\fh$ belongs to the positive chamber for $\gS_\fh^+$ in $\fahd.$
There exists $Y \in \fah$ such that for all $\ga \in \gS$
\begin{equation}
\label{e: inner prod ga and rho h}
\inp{\ga}{\rho_\fh} >0 \implies \ga(Y) >0.
\end{equation}

Replacing $Y$ by a small perturbation if necessary,
we may in addition assume that $Y$ belongs to the set
$\fah^{\rm reg}$ of elements $Z \in \fah$ such that for all $\ga \in \gS$ we have
$\ga(Z) = 0 \implies \ga|_{\fah} = 0.$

We fix $\epsilon > 0 $ sufficiently close to zero so that for all $\ga \in \gS,$
\begin{equation}
\label{e: sign of ga Y + eX}
\ga (Y) > 0 \implies \ga(Y + \epsilon X)  > 0,
\end{equation}
and so that $Y + \epsilon X$ is a regular element in $\fa.$
Let $Q \in \cP(A)$ be the
unique parabolic subgroup such that $Y + \epsilon X$ belongs to the positive chamber $ \fa^+(Q).$
We claim that $Q$ satisfies the requirements.
To see this, we start with the observation that for $\ga \in \gS\setminus\faqd$
we have $\ga(Y) \neq 0,$ so that
\begin{equation}
\label{e: equality of signs}
\sign \ga(Y) = \sign \ga(Y + \epsilon X).
\end{equation}
For such $\ga$ it follows by application of
(\ref{e: equality of signs})  to both $\ga$ and $\gs \ga$ that
$$
\sign \gs \ga (Y + \epsilon X) = \sign \ga(Y + \epsilon X).
$$
Thus, we see that for $\ga \in \gS\setminus \faqd$ we have
$\ga \in \gS(Q)$ if and only if $\gs \ga \in \gS(Q).$ Thus,
$\gS(Q) \setminus \faqd = \gS(Q, \gs)$ and we infer that $Q$ is $\fh$-extreme.

If $\ga \in \gS(Q)\cap \faqd,$ then
$$
\epsilon \ga(X) = \ga(Y + \epsilon X) > 0
$$
and we obtain that $\ga \in \gS(P_0).$
Hence, $\gS(Q, \gs\Cartan) = \gS(Q)\cap \faqd \subseteq \gS(P_0).$

Next, assume that  $\ga \in \gS$ satisfies $\inp{\ga}{\rho_\fh} \neq 0.$ Then
in view of (\ref{e: inner prod ga and rho h}) and (\ref{e: sign of ga Y + eX}) we have
\begin{equation}
\label{e: equality three signs}
\sign \inp{\ga}{\rho_\fh} = \sign \ga(Y) = \sign \ga(Y + \epsilon X).
\end{equation}
In particular, the above is valid for $\ga \in \gS_\fh^+.$
From this we see that
$\gS^+_\fh = \gS(Q)\cap \fahd,$ so that $\rho_{Q,\fh} = \rho_\fh.$

For the proof of (a), it remains to be shown that $Q$ is $\fh$-compatible.
Let $\ga \in \gS(Q).$ If $\inp{\ga}{\rho_{Q\fh} } \neq 0,$ then it follows
from (\ref{e: equality three signs}) that $\inp{\ga}{\rho_{Q,\fh}} = \inp{\ga}{\rho_\fh} > 0.$ This establishes (a).

We turn to (b). Let $\ga \in \gS(Q, \gs \Cartan).$ Then $\inp{\ga}{\rho_{Q,\fh}} \geq 0.$
On the other hand, $\gs\Cartan\ga \in \gS(Q)$ hence
$$
0 \leq \inp{\gs\Cartan \ga}{\rho_{Q,\fh}}  = - \inp{\ga}{\rho_{Q,\fh}}
$$
and we see that $\ga \perp \rho_{Q,\fh}.$

Finally, assume that $P,Q$ satisfy the hypotheses of (c) and that
$P \in \cPH(A).$ From $P\succeq Q$ it follows that $\gS(P)\cap \fa_\fh^* = \gS(Q) \cap \fa_\fh^*,$
hence $\rho_{P,\fh} = \rho_{Q, \fh}.$ Since $P$ is $\fh$-compatible and since (b) holds,
we see that every root $\ga$ from the set
$$
\Psi:= \gS(P) \cup - \gS(P, \gs \Cartan)
$$
 satisfies $\inp{\ga}{\rho_{Q,\fh}} \geq 0.$
We will finish the proof by showing that $\gS(Q) \subseteqq \Psi.$ Let $\ga \in \gS(Q).$ Then either
$\ga \in \gS(Q,\gs \Cartan)$ or $\ga \in \gS(Q, \gs).$ In the first case,
$\ga \in \gS(P, \gs\Cartan) \subseteqq \Psi.$
Thus, assume $\ga \in \gS(Q,\gs).$ If $\ga \in \gS(P, \gs)$ then $\ga \in \Psi.$ Thus, it remains
to consider the case that $\ga \notin \gS(P,\gs).$ Since
$-\ga \in \gS(P,\gs)$ would imply $-\ga \in \gS(Q,\gs) \subseteqq \gS(Q),$ contradiction,
both $\ga$ and $-\ga$ do not belong to $\gS(P, \gs).$ It follows that one of them must belong to
$\gS(P, \gs \Cartan),$ hence $\ga\in \Psi.$
\end{proof}

\begin{Ex}[\bf Group case]\label{Ex Group case - H-compatible parabolic subgroups}
We use notation as in Example \ref{Ex Group case - HC-transform def compared to HC's def}. Every element of $\cP(\bp A\times\bp A)$ is of the form $P=\bp P\times\bp Q$ where $\bp P$ and $\bp Q$ are minimal parabolic subgroups containing $\bp A$. All roots are non-zero on $\faq$, hence $\rho_{P,\fh}=0$. Each of these parabolic subgroups is therefore $\fh$-compatible.
\end{Ex}

The importance of the notion of $\fh$-compatibility
comes from the following result,
which implies that if $Q \in \cPH(A),$ then certain singularities of
$\Etau(Q:\dotvar)$ are caused by singularities of  $\Etau^{\circ}(\bar P_{0}:\dotvar)$, see also Lemma \ref{l: relation EQ with Enorm}.

\begin{Prop}\label{Prop c-function holomorphic on a_q*(P0,0)}
Let $P_{0}$ be a minimal $\sigma\theta$-stable parabolic subgroup containing $A$.
Let $Q \in \cPH(A)$ and assume
that $\Sigma(Q:\sigma\theta)\subseteq   \Sigma(P_{0})$.
Then $C_{\bar{P}_{0}|Q}(1 :\dotvar)$ is holomorphic on $-\faq^{*+}(P_{0})+i\faq^{*}$.
\end{Prop}

\begin{proof}
Assume that $\lambda\in-\faq^{*+}(P_{0})+i\faq^{*}.$
Then $\Re\,\langle\lambda,\alpha\rangle<0$ for all $\alpha\in\Sigma(P_{0})$ and $\langle\rho_{Q,\fh},\alpha\rangle\geq0$ for all $\alpha\in\Sigma(Q)$. Hence,
 $$
 \Re\,\langle-\lambda+\rho_{Q,\fh},\alpha\rangle>0 \quad \textnormal{for all} \;\;\alpha\in\Sigma(P_{0})\cap\Sigma(Q).
 $$
The result now follows by application of Proposition \ref{Prop relation E(Q) and E^circ(cP0)} (b).
\end{proof}

We end this section with a result about $\cW$-conjugates of $\fh$-compatible parabolic subgroups.

\begin{Lemma}
\label{l: H compatibility stable under cW}
If $Q\in\cPH(A)$, then $Q^{v}\in\cPH(A)$ for every $v\in N_K(\faq) \cap N_K(\fah).$
In particular, this is valid for $v \in \cW.$
\end{Lemma}

\begin{proof}
Since $v$ normalizes both $\fa$ and $\fah,$ it follows that
$$
\Sigma(Q^{v})\cap\fa_{\fh}^{*}
= v^{-1}  \cdot(\Sigma(Q)\cap\fa_{\fh}^{*}),
$$
hence $\rho_{Q^{v},\fh}= v^{-1} \cdot\rho_{Q,\fh}$. The lemma now follows from the fact that $v$ acts isometrically on $\fa^{*}$.
\end{proof}

\subsection{Residues for the trivial $K$-type}
\label{ss: residues trivial K type}
In this subsection we retain the
\medno
{\bf Assumption:\  } \em  $G/H$ is of split rank 1.
\medbreak\rm
Our goal is to analyze the residues of Eisenstein integrals $E(Q:\gl)$ as introduced in Lemma \ref{l: residues one}, for $Q \in \cPH(A)$ and for $(\tau, V_\tau)$ equal to the trivial
representation $(1, \C)$ of $K.$ To emphasize that $\tau = 1,$ we
 denote the associated Eisenstein
integrals with $E_1(Q: \gl),$ see also (\ref{e: defi Eisenstein}).

As before, we assume that $P_{0}$ is a minimal $\sigma\theta$-stable parabolic subgroup containing
$A$ and such that $\Sigma(Q,\sigma\theta)\subseteq  \Sigma(P_{0})$.

If $\pi$ is a discrete series representation for $G/H$, we agree to
 write $\cC(G/H)_{\pi}$ for the closed subspace of
 $\cC(G/H)$ spanned by left $K$-finite and right $H$-fixed generalized matrix coefficients of $\pi$.

\begin{Prop}\label{p: R Q  in span of finite sum of discrete series}
Assume that $Q\in\cPH(A)$. There exist finitely many spherical discrete series representations $\pi_1,\dots,\pi_{k}$ such that for all $\mu\in S_{Q,1}$,  all $\psi\in\oC(1)$ and all
$a \in \Aq,$
$$
\Resone(Q: \mu: a :\dotvar )(\psi) \in\bigoplus_{i=1}^{k}\;\cC(G/H)_{\pi_{i}}.
$$
\end{Prop}

We will prove this proposition through a series of results. First, we need to prepare
for these.
Let $W(\faq)$ be the Weyl group of the root system of $\faq$ in $\fg.$
Then $W(\faq)=\{-1,1\},$ since $\dim(\faq)=1$ by assumption.
The map
\begin{equation}
\label{e: iso oC one}
\oC(1) \longrightarrow \C^\cW, \;\;\;\; \psi \mapsto (\psi_w(e): w\in \cW)
\end{equation}
is a linear isomorphism via which we shall identify the indicated spaces.

Let $\psi \in \oC(1).$ Then from  \cite[Thm.\ 8.13]{vdBanKuit_EisensteinIntegrals}
it follows that, for $R \in \cP_\gs(\Aq),$ $ w \in \cW,$  $b \in \Aq^+(R)$ and generic $\gl \in \faqdc,$
$$
E_1(Q:\psi: \gl)(bw) = \sum_{s \in \{\pm 1\}} \Phi_{R, w}(s\gl: b) [C_{R | Q}(s :\gl)\psi]_w(e).
$$
Here $\Phi_{R, w}(\gl, \dotvar),$ for generic $\gl \in \faqdc,$ is a certain function on $\Aqp(R)$ defined as in  \cite[Thm.\ 11.1]{vdBanSchlichtkrull_ExpansionsForEisensteinIntegralsOnSemisimpleSymmetricSpaces}, for $\tau = 1.$  We recall that the functions are related by the equations
\begin{equation}
\label{e: relations Phi R w}
\Phi_{R,w}(\gl, a) = \Phi_{w^{-1}Rw, 1}(w^{-1}\gl, w^{-1}aw),
\end{equation}
for generic $\gl \in \faqdc$ and all $a \in \Aqp(R),$ see \cite[Lemma 10.3]{vdBanSchlichtkrull_ExpansionsForEisensteinIntegralsOnSemisimpleSymmetricSpaces}.
It follows from these relations and \cite[Eqn.\  (15)]
{vdBanSchlichtkrull_ExpansionsForEisensteinIntegralsOnSemisimpleSymmetricSpaces}
that the function $\Phi_{R,w}(\gl, \dotvar)$ for generic $\gl \in \faqdc$ has a converging series expansion of the form
\begin{equation}
\label{e: expansion Phi R w}
\Phi_{R, w} (\gl , a) = a^{\gl - \rho_{R}} \sum_{k\geq 0}
\Gamma_{R,w,k}(\gl) a^{-k\ga}, \qquad (a \in \Aqp(\bar P_0)),
\end{equation}
where $\ga$ denotes the unique indivisible root in $\gS(R, \faq).$
The coefficients $\Gamma_{R,w,k},$ for $k \in \N,$ are meromorphic functions
on $\faqdc,$ which are uniquely determined by the following conditions,
see
\cite[Prop. 5.2, Eqn.\ (19)]{vdBanSchlichtkrull_ExpansionsForEisensteinIntegralsOnSemisimpleSymmetricSpaces},
taking into account that
$\tau(L_\ga^\pm) = 0$ and $\gg =0,$ because $\tau$ is trivial.
\begin{enumerate}
\item[{\rm (1)}]
$
\Gamma_{R,w,0}(\gl) = 1,
$
for all $\gl \in \faqdc.$
\item[{\rm (2)}] The function
$
\gl \mapsto \left(\prod_{l=1}^k  \inp{2 \gl - l\ga}{\ga}\right) \cdot \Gamma_{R,w,k}(\gl)
$
is entire holomorphic on $ \faqdc.$
\end{enumerate}
In the proof ahead, we will need the following additional properties of the functions $\Phi_{R,w}$
and their expansions.

\begin{Lemma}
\label{l: behavior Phi}
Let $\Omega$ be a bounded open subset of $\faqd.$ Then there exists a polynomial function
$q \in \Pi_{\gS, \R}(\faqd),$ see (\ref{e: defi Pi Sigma}), such that the following assertions
are valid.
\begin{enumerate}
\itema For every $k \geq 0$ the function $q \Gamma_{R,w,k}$ is holomorphic on
$\Omega + i\faq.$
\itemb The power series
$$
\sum_{k\geq 0} q(\gl) \Gamma_{R,w,k}(\gl) z^k
$$
converges absolutely on $D = \{z \in \C : |z| < 1\},$
locally uniformly  in $(\gl, z) \in (\Omega + i\faqd) \times D.$ In particular, it defines a holomorphic function on the mentioned set.
\itemc
The assignment $a \mapsto q(\gl) \Phi_{R, w}(\gl: a)$ defines a smooth function on $\Aq,$ depending holomorphically on $\gl \in \Omega \times i\faqd.$
\itemd For all  $\gl \in \Omega+i\faqd$ and $a \in A_\fq^+(R),$
$$
q(\gl) \Phi_{R, w} (\gl , a) = q(\gl) a^{\gl - \rho_{R}} \sum_{k\geq 0} \Gamma_{R,w,k}(\gl) a^{-k\ga}.
$$
\end{enumerate}
\end{Lemma}

\begin{proof}
In view of the relations (\ref{e: relations Phi R w}), we may assume that $w =1.$
By boundedness of $\Omega$ there exists $n \geq 1$ such that
$\inp{2 \gl - l\ga}{\ga}  \neq 0$ for all $l >  n$ and $\gl \in \Omega + i\faqd.$ In view of conditions
(1) and (2) above, we see that the
polynomial
$$
q(\gl) = \prod_{l=1}^n  \inp{2 \gl - l\ga}{\ga}
$$
satisfies the requirements of (a).

In view of (a), it follows from the estimate for the coefficients given in \cite[Thm.\ 7.4]{vdBanSchlichtkrull_ExpansionsForEisensteinIntegralsOnSemisimpleSymmetricSpaces},
that for every compact subset $U \subseteqq \Omega + i\faqd$ there exists a constant $C >0$ and
an integer $\kappa > 0$ such that
$$
|q(\gl) \Gamma_{R,1,k}(\gl) | \leq C (1 + k)^\kappa, \qquad (\gl \in U).
$$
From this, (b) follows readily. Finally, (c) and (d) are immediate
consequences of (a), (b) and (\ref{e: expansion Phi R w}).
\end{proof}

Based on the lemma, we can prove another preparatory result.

\begin{Prop}
\label{p: residue in Schwartz space}
Let $\mu \in \faqdp(P_0).$
Then for every holomorphic function $f:\faqdc \to \C$
and every $\psi \in \oC(\tau)$  the function
\begin{equation}
\label{e: residue f Eis one}
\Res_{\gl = -  \mu} f(\gl) E_1(Q:\psi:\gl) : \;G/H \to \C
\end{equation}
is $\DGH$-finite and contained in $\cC(G/H)^{K}.$
\end{Prop}

\begin{Rem}
Actually the result is valid for any holomorphic function $f$ defined on an open neighborhood
of $-\mu,$ but we will not need this.
\end{Rem}

\begin{proof}[Proof of Proposition \ref{p: R Q  in span of finite sum of discrete series}]
Let $F$ be the function (\ref{e: residue f Eis one}). It follows from Lemma \ref{l: residue preparation} (b) that $F$ belongs to $C^\infty(G/H)^K$ and is $\DGH$-finite.

We will complete the proof by showing that $F$ belongs to the Schwartz space.
Fix $w \in \cW,$ then by \cite[Thm.~6.4]{vdBan_AsymptoticBehaviourOfMatrixCoefficientsRelatedToReductiveSymmetricSpaces}
it suffices to establish,
for any $\geps >0,$ the existence of a constant $C= C_\geps >0$ such that
\begin{equation}
\label{e: estimate F}
|F(bw)| \leq C b^{-\mu + \geps \ga - \rho_{\bar P_0}}  \qquad (b \in \Aqp(\bar P_0)).
\end{equation}
Taking $R = \bar P_0$ we obtain that, for $w \in \cW,$
for $b \in \Aq^+(\bar P_0)$ and generic $\gl \in \faqdc,$
$$
E_1(Q:\psi:\gl)(bw) = \sum_{s \in \{\pm 1\}} \Phi_{\bar P_0, w}(s\gl: b) [C_{\bar P_0| Q}(s :\gl)\psi]_w(e).
$$
Here $C_{\bar P_0| Q}(1 : \dotvar)$ is holomorphic at the element $-\mu \in \faqdp(\bar P_0).$

For $k \geq 1$ we define the polynomial
function $q_k :=  \inp{\dotvar + \mu}{\ga}^k$ on $\faqdc.$
Let $\Omega$ be a bounded open neighborhood of $-\mu$ in $\faqd.$
Then we may select
an integer $k_1 \geq 0$ such that the polynomial $q_{k_{1}}$
satisfies all properties of Lemma \ref{l: behavior Phi} for $R = \bar P_0$ and all $w \in \cW.$
In addition we may fix $k_2 \geq 0$ such that the function
\begin{equation}
\label{e: holomorphy C}
\gl \mapsto q_{k_2}(\gl) C_{\bar P_0|Q}(s: \gl)
\end{equation}
is holomorphic on an open neighborhood $\nbhood$ of  $-\mu$
in $ \Omega + i\faqd$  for each $s \in \{\pm 1\}.$
Put $q = q_{k_{1}} q_{k_2},$
then it follows that $q(\gl)  E_1(Q:\psi:\gl)(bw)$
is smooth in $ (\gl, b) \in \nbhood  \times \Aqp(\bar P_0)$
and holomorphic in the first of these variables.

For each $s \in \{\pm 1\},$ we define the disjoint decomposition
$$
\N = \Nt(s,+) \cup  \Nt(s,-)
$$
by
$$
k \in \Nt(s, +) \iff  \inp{- 2 s \mu - k\ga}{\ga}  > 0.
$$
Accordingly, we put
$$
\Phi_{\bar P_0, w}^{\pm} (s \gl, b)
= b^{-\rho_{\bar P_0} } \sum_{k \in \Nt(s,\pm)}\Gamma_{\bar P_0, w, k}(s\gl) b^{s\gl - k\ga}.
$$
Then
$$
\Phi_{\bar P_0, w}(s \gl, b) = \Phi_{\bar P_0, w}^{+} (s \gl, b) + \Phi_{\bar P_0, w}^{-} (s \gl, b).
$$
It is easily seen that $\Nt(s,+ )$ is finite and without gaps in $\N.$ Furthermore,
we may shrink $\nbhood$ to arrange that
$\inp{2s \gl - k\ga}{\ga} > 0$ for all $s = \pm 1,$ $k \in \Nt(s,+)$ and $\gl \in \nbhood.$
In view of property (2) below (\ref{e: expansion Phi R w}) this implies that the function
$$
\Phi_{\bar P_0, w}^+(s\gl, b)
$$
is holomorphic in $\gl \in \nbhood,$ for each $s = \pm 1.$
Furthermore, if $s = -1,$ then $\inp{-s\mu}{\ga} < 0$ and we see that
$\Nt(-1,+)= \emptyset$ so that in fact
$$
\Phi_{\bar P_0, w}^+(- \gl, b) = 0.
$$
We agree to write
$$
\Psi^\pm(\gl,b)) =  \sum_{s \in \{\pm 1\}} \Phi_{\bar P_0, w}^\pm(s\gl: b) [C_{\bar P_0| Q}(s :\gl)\psi]_w(e),
$$
so that
$$
E_1(Q:\psi:\gl)(bw) = \Psi^+(\gl,b) + \Psi^-(\gl,b).
$$
Taking into account that  $C_{\bar P_0|Q}(1:\dotvar)$ is holomorphic
at the point $-\mu, $
It follows from the above that $\Psi^+(\gl, b)$ is holomorphic at $\gl= - \mu,$ so that
$$
 \Res_{\gl = -\mu}\;[ f(\gl) \Psi^+(\gl, b)] = 0 \qquad (b \in \Aqp(P_0)).
$$
We infer that
\begin{equation}
\label{e: residue F as Psi minu}
F(bw) = \Res_{\gl = -\mu} f(\gl) \Psi^-(\gl, b).
\end{equation}
We will now derive an estimate for $f(\gl)q(\gl)\Psi^-(\gl, b)$ by looking at the exponents of the
series expansion.
If $s =1,$ then it follows that for $k \in \Nt(s,-)$ we have
$$
\inp{-s\mu - k\ga}{\ga} \leq 0 + \inp{\mu}{\ga} <  0.
$$
Shrinking $U$ we may arrange that for all
$\gl \in U$  and all $k \in \Nt(s,-),$
\begin{equation}
\label{e: estimate inp s gl}
\Re\inp{s\gl - k \ga}{\ga} \leq \inp{\mu + \epsilon \ga}{\ga}.
\end{equation}
As described in Lemma \ref{l: behavior Phi},
the series for $q_1(\gl)\Phi_{\bar P_0, w}(s\gl: b)$
is essentially a power series in $b^{-\ga},$ with holomorphic dependence on
$\gl.$
Hence, we may shrink $U$ to arrange that there exists a constant $C >0$
 such that for $\gl \in U$ and $b\in \Aq$ with $b^\ga \geq 2$ we have
\begin{equation}
\label{e: estimate Phi}
|q_1(\gl) \Phi_{\bar P_0, w}^{-}(s\gl, b)| \leq C b^{-\mu -\rho_{\bar P_0} + \epsilon\ga}.
\end{equation}
On the other hand, if $s = -1,$ then for all $k \in \N$ we have
$$
\inp{-s\mu - k\ga}{\ga} \leq \inp{\mu}{\ga} < 0
$$
and we see that by shrinking $U$ even further if necessary, we may arrange
that (\ref{e: estimate inp s gl}) is valid for the present choice of $s,$ all $k \in \N$ and all
$\gl \in U.$ This leads to the estimate (\ref{e: estimate Phi}) for $s = -1.$
From the estimates obtained, combined with the holomorphy of the function
(\ref{e: holomorphy C}), for $s = \pm 1,$ we infer that there exists a constant $C_1 > 0$
such that
$$
|f(\gl)q(\gl) \Psi^-(\gl, b)| \leq C_1 b^{-\rho_{\bar P_0} + (- \mu  + \epsilon\ga) }
$$
for $\gl \in U$ and $b^\ga \geq 2.$
Using the integral formula for the residue in (\ref{e: residue F as Psi minu})
and taking into account that $q(\gl)^{-1}$ is bounded
on a circle around $-\mu$ in $U,$ we  infer that there exists
$C >0$ such that
(\ref{e: estimate F}) is valid for all $b \in \Aq$ with $b^\ga \geq 2.$ Since
$F$ is continuous, a similar estimate holds for all $b \in \Aqp(\bar P_0).$
\end{proof}

We finally come to the proof of Proposition \ref{p: R Q  in span of finite sum of discrete series}.

\begin{proof}[Proof of Proposition \ref{p: R Q  in span of finite sum of discrete series}]
Let $\mu \in S_{Q,1}.$
It follows from Lemma \ref{l: residues two} that there exists an integer
$d \geq 0$ and a finite dimensional subspace
$\findim \subseteqq C^\infty(G/H)^K,$
consisting of $\DGH$-finite functions, such that
\begin{equation}
\label{e: Resone in P tensor S}
\Resone(Q:\mu)\psi \in P_d(\Aq) \otimes \findim
\end{equation}
for all $\psi \in \oC(1). $ On the other hand, it follows by application of  Proposition
\ref{p: residue in Schwartz space} that $\Resone(Q:\mu:a:\dotvar)$
is a $\DGH$-finite function in $\cC(G/H)^K,$
for every $a \in \Aq.$
Hence, (\ref{e: Resone in P tensor S}) is valid with $\findim$ a finite dimensional subspace of $\cC(G/H)^K,$
consisting of $\DGH$-finite functions, which are therefore in particular $Z(\fg)$-finite.

It now follows
that the $(\fg,K)$-span of $\findim$ in $\cC(G/H)$ is a $(\fg,K)$-module of
finite length in view of a well known result of Harish-Chandra; see \cite[p.~312, Thm.~12]{Varadarajan_HarmonicAnalysisOnRealReductiveGroups} and \cite[p.~112, Thm.~4.2.1]{Wallach_RealReductiveGroupsI}.)
The closure of this span in $L^{2}(G/H)$ is therefore a finite direct sum of irreducible subrepresentations. Since $\findim$ consists of left $K$-invariant functions, each of these irreducible subrepresentations is spherical.
The result follows.
\end{proof}
\begin{Rem}
The proof of Proposition \ref{p: R Q  in span of finite sum of discrete series}
relies heavily on the assumption that $\faq$ is one-dimensional, which makes it possible to analyse
which exponents vanish from the expansion involved, by taking residues.
\end{Rem}

\subsection{Convergence for symmetric spaces of split rank one}
\label{ss: Convergence}
In this subsection we retain the
\medno
{\bf Assumption: \ }   $G/H$ is of split rank one.
\begin{Thm}\label{Thm convergence}
Let $Q\in\cPH(A)$. Then $\Ht_{Q}$ has a unique extension to a continuous linear map $\cC(G/H)\to C^{\infty}(L/H_{L}).$ Moreover,  for every $\phi\in\cC(G/H),$
$$
\Ht_{Q}\phi(l)
=\Deltach_{Q}(l)\int_{N_{Q}/H_{N_{Q}}}\phi(ln)\,dn
\qquad(l\in L)
$$
with absolutely convergent integrals.

Furthermore, the Radon transform $\Rt_Q$ has a unique
extension to a continuous linear map $\cC(G/H) \to C^\infty(G/N_Q)$ and for every $\phi \in \cC(G/H)$
$$
\Rt_Q \phi(g) = \int_{N_Q/H_{N_Q}} \phi(gn)\; dn
\qquad(g \in G),
$$
with absolutely convergent integrals.
\end{Thm}

\begin{proof}
As in (\ref{e: iso oC one}) we identify  $\oC(1) \simeq \C^{\cW}$. Let $\psi_0\in  \oC(1)$  be the unique element determined by
$(\psi_0)_w = \gd_{1w}$ for $w \in \cW.$

By Lemma \ref{Lemma relation between H_(Q,tau) and H_Q} we see that, for $\phi\in C_{c}^{\infty}(G/H)^{K},$
$$
\big\langle\Ht_{Q,1}\phi(a),\psi_{0}\big\rangle
=\Ht_{Q}\phi(a)
\qquad(a\in A_{\fq}).
$$
It follows from Proposition \ref{p: R Q in span of finite sum of discrete series} that the functions
$\Resone(Q: \mu: a: \dotvar )(\psi_0),$ for $\mu \in S_{Q,\tau},$
belong to $\cC(G/H)$ for every $a \in \Aq.$ Furthermore, by
Lemma \ref{l: residues two}, these functions depend
polynomially on $a.$
If we combine this with
Proposition \ref{Prop JQ: C(G/H:tau) to C^infty(Aq)}, we infer that the expression on
 the right-hand side of (\ref{e: cor Ht minus Jphi as residues})
 is well defined for $\phi\in\cC(G/H: \oneK )=\cC(G/H)^{K}$  and depends linear continuously
on it with values in
$$
C^{\infty}(A_{\fq})=C^{\infty}(L/H_{L})^{M}.
$$
It follows that the restriction of $\Ht_{Q}$ to $C_{c}^{\infty}(G/H)^{K}$ extends to a continuous linear map from $\cC(G/H)^{K}$ to $C^{\infty}(L/H_{L})^{M}$.
The theorem now follows by application of Proposition \ref{Prop If H_Q extends to C(G/H)^K, then R_Q extends to C(G/H) with conv integrals}.
\end{proof}

\begin{Rem}
\label{r: convergence Radon in AFJS}
For the hyperbolic spaces $\SO(p,q+1)_{e}/ \SO(p,q)_e,$ this result is due to
\cite{AndersenFlenstedJensenSchlichtkrull_CuspidalDiscreteSerieseForSemisimpleSymmetricSpaces}.
\end{Rem}

In the following we assume more generally that $(\tau, V_{\tau})$ is a finite dimensional unitary representation of $K$.

\begin{Cor}\label{Cor H_(Q,tau) extends to Schwartz functions}
Let $Q\in\cPH(A)$. Then $\Ht_{Q,\tau}$ extends to a continuous linear map
$$
\Ht_{Q,\tau}:\cC(G/H:\tau)\to C^{\infty}(A_{\fq})\otimes\oC(\tau).
$$
Moreover, (\ref{eq def H_(Q,tau)}) and (\ref{eq relation H_(Q,tau) and H_Q}) are valid for every
$\phi \in \cC(G/H:\tau)$ with the extensions of $\Ht_{Q,\tau}$ and $\Ht_{Q^v}$
to the associated Schwartz spaces; the appearing integrals are absolutely convergent.
\end{Cor}

\begin{proof}
Since $C^\infty_c(G/H:\tau)$ is dense in $\cC(G/H:\tau),$
this follows immediately from combining Theorem \ref{Thm convergence} with Lemma
\ref{l: H compatibility stable under cW}.
\end{proof}

By Proposition \ref{Prop JQ: C(G/H:tau) to C^infty(Aq)} the map $\It_{Q,\tau}$ extends to a continuous linear map from  $\cC(G/H:\tau)$ to $C^{\infty}(A_{\fq})\otimes\oC(\tau)$ as well.
Hence, it follows from Equation (\ref{e: cor Ht minus Jphi as residues})
that for every $a\in A_{\fq}$ and $\psi\in\oC(\tau),$
$$
\sum_{\mu \in S_{Q,\tau}} a^\mu \Restau(Q: \mu : a : \dotvar)\psi
$$
is a smooth function defining a tempered distribution. In particular, writing
$C^\infty_\temp(G/H): = C^\infty(G/H) \cap \cC'(G/H),$ we obtain

\begin{Cor}
\label{c: R tempered}
Let $Q\in\cP_{\fh}(A)$ and $\mu \in S_{Q,\tau}.$ Then
$$
\Restau(Q:\mu) \in P(\Aq) \otimes C^\infty_\temp(G/H) \otimes \Hom(\oC(\tau), \Vtau).
$$
\end{Cor}

We now observe that Proposition \ref{Prop JQ: C(G/H:tau) to C^infty(Aq)}, Corollary \ref{Cor H_(Q,tau) extends to Schwartz functions} and Corollary \ref{c: R tempered}
imply the following result.

\begin{Cor}
\label{c: HQtau as tempered term plus residues for Schwartz}
Let $Q\in\cPH(A)$.
Then equation
{\rm (\ref{e: cor Ht minus Jphi as residues})}
holds for every $\phi\in\cC(G/H:\tau)$, $\psi\in\oC(\tau)$ and $a\in A_{\fq}$.
\end{Cor}

\section{Cusp forms and discrete series representations}\label{section Cusp forms}
A well known
 result of Harish-Chandra asserts for the case of the group that
 the closed span in the Schwartz space of the bi-$K$-finite matrix coefficients of
 the representations from the discrete series equals the space
of so-called cusp forms. (See \cite{HC_ds_2}, \cite[Thm.\ 10]{HC_harman_70}, \cite[Sect.\ 18,27]{Harish-Chandra_HarmonicAnalyisOnRealReductiveGroupsI} and \cite[Thm.\ 16.4.17]{Varadarajan_HarmonicAnalysisOnRealReductiveGroups}.)
 In the present section we generalize this result for the case of split rank one.

\subsection{The kernel of $\It_{Q,\tau}$}
\label{ss: kernel of I Q tau}
In the present subsection
we do not make any assumption on the dimension of $\faq$.
Let $(\tau,V_{\tau})$ be a finite dimensional unitary representation of $K$ as before.

\begin{Thm}\label{Thm ker(F^0_(P0,tau))=ker(J_(Q,tau))}
Let $Q\in\cP(A)$ and $P_{0} \in \cP_\gs(A).$
Then
the following are equal as subspaces of
$\cC(G/H:\tau),$
\begin{equation}
\label{e: kernels equal}
\ker(\nFt_{\bar P_{0}})
=\ker(\It_{Q,\tau}).
\end{equation}
\end{Thm}

\begin{proof}
Let $\phi\in\cC(G/H:\tau).$ Then the definition
of $\Kt_{Q,\tau}\phi$ in (\ref{eq def K_Q}) is meaningful and the equality (\ref{eq J_Q= u*K_Q}) is valid. If $\nFt_{\bar P_0} \phi = 0$ then it follows from (\ref{e: FQtau and FbarPzero}) that
$\Kt_{Q,\tau}\phi = 0,$ which in turn implies that $\It_{Q,\tau}\phi = 0.$
This shows that the space on the left-hand side of (\ref{e: kernels equal}) is contained in the space
on the right-hand side.

For the converse inclusion, let $p_{h}$, $u$ and $v$ be as in the proof of
Proposition \ref{Prop JQ: C(G/H:tau) to C^infty(Aq)}. Let $D$
be the bi-invariant differential operator on $\Aq$ which satisfies
$$
\eFt_{A_{\fq}}(D\delta_e)
=p_{h},
$$
with $\gd_e$ the Dirac measure at the point $e$ on $\Aq.$ Then
$$
\eFt_{A_{\fq}}\big(D(\It_{Q,\tau}\phi)\big)
=p_{h}\eFt_{A_{\fq}}\big(\It_{Q,\tau}\phi\big)
=p_{h}\eFt_{A_{\fq}}\big(\Kt_{Q,\tau}\phi\big)v.
$$
If $\xi \in \faqd$ is a real linear functional, then by looking at
real and imaginary parts, we see that $| \inp{\gl}{\xi} | \leq |\inp{\gl + \epsilon \nu}{\xi}|$
for all $\gl \in i\faqd$, $\epsilon>0$ and $\nu\in\fa_{\fq}^{*}$.
This implies that
$$
\left| \frac{p_h(\gl)}{p_h(\gl + \epsilon \nu)} \right| \leq 1,\qquad (\gl \in \faqd).
$$
By a straightforward application of the Lebesgue dominated convergence theorem it
now follows that $p_h v = 1,$ and we infer that
$$
D(\It_{Q,\tau}\phi)
=\Kt_{Q,\tau}\phi.
$$

Let now $\phi \in \cC(G/H:\tau)$ and assume that $\It_{Q,\tau}\phi = 0.$ Then $\Kt_{Q,\tau}\phi = 0$
and it follows from (\ref{eq def K_Q}) that $[p_h\Ft_{Q,\tau}]\phi = 0.$
In view of (\ref{e: FQtau and FbarPzero}) this implies that
$$
p_{h}(\lambda) C_{\bar P_0 : Q}(1 : - \bar\lambda)^*\,\nFt_{\bar P_0}(\phi)(\gl) = 0
$$
for generic $\gl \in i\faqd.$ Since $\nFt_{\bar P_0}\phi$ is smooth by
\cite[Cor.\ 4, p.\ 573]{vdBanSchlichtkrull_FourierTransformOnASemisimpleSymmetricSpace}
and $C_{\bar P_0 : Q}(1 : - \bar\lambda)$ is invertible for generic $\gl \in i\faqd,$
by \cite[Thm.\ 8.13]{vdBanKuit_EisensteinIntegrals}, it follows that $\nFt_{\bar P_0}\phi = 0.$
\end{proof}

Let $\cC_{\mc}(G/H:\tau)$ be the closed subspace of $\cC(G/H:\tau)$ corresponding to the most continuous part of the spectrum.
By \cite[Cor.\ 17.2 and Prop.\ 17.3]{vdBanSchlichtkrull_MostContinuousPart} the kernel of $\nFt_{\bar P_0}$ is equal to the orthocomplement (with respect to the $L^{2}$-inner product) in $\cC(G/H:\tau)$ of $\cC_{\mc}(G/H:\tau)$. Theorem \ref{Thm ker(F^0_(P0,tau))=ker(J_(Q,tau))} therefore implies that $$
\cC_{\mc}(G/H:\tau)\cap\ker(\It_{Q,\tau})=\{0\}.
$$

Let $\cC_{\ds}(G/H\colon \tau)$ be the closed span in $\cC(G/H\colon \tau)$ of the $K$-finite generalized matrix coefficients of the representations from the discrete series for $G/H.$ Then the space $\cC_{\ds}(G/H\colon\tau)$ is contained in the orthocomplement of $\cC_{\mc}(G/H\colon\tau)$ and therefore in the kernel of $\It_{Q,\tau}$ for every $Q\in\cP(A)$.

\begin{Cor}
\label{c: ker It Q tau is ds}
If $\;\dim(\faq)=1$, then $\;\ker(\It_{Q,\tau})
=\cC_{\ds}(G/H:\tau). $
\end{Cor}
\begin{proof}
This follows from the above discussion combined with \cite[Prop.\ 17.7]{vdBanSchlichtkrull_MostContinuousPart}
\end{proof}

\subsection{Residues for arbitrary $K$-types}
\label{ss: residues arbitrary K types}
In this subsection
we will work under the
\medno
{\bf Assumption:\ } {\em $G/H$ is of split rank one.}
\medno
Let $(\tau,\Vtau)$ be a finite dimensional unitary representation of $K$
and let $Q\in\cP_{\fh}(A)$. Let $P_0 \in \cP_\gs(\Aq)$ be such
that $\gS(Q, \gs\Cartan) \subseteq \gS(P_0).$ We recall that the singular set $S_{Q,\tau}$ is a finite subset of $\faqdp(P_0),$ see Lemma \ref{Lemma S_(Q,tau) is real and finite}. The first
main result of this subsection is that the residues appearing in (\ref{e: cor Ht minus Jphi as residues}) are $L^2$-perpendicular
to the part of the Schwartz space corresponding to the most continuous
part of the Plancherel formula.

\begin{Thm}\label{Thm most cont part orthogonal to residues}
Let $\mu\in S_{Q,\tau}$. For every $\phi\in\cC_{\mc}(G/H:\tau)$, $\psi\in\oC(\tau)$ and $a\in A_{\fq}$
\begin{equation}\label{eq <phi, Res(Q:mu:a)>=0}
\langle\phi,\Restau(Q:\mu: a:\dotvar)\psi\rangle
=0.
\end{equation}
\end{Thm}

This result will be proved through a series of partial results.
Let $P_{0}$ be a parabolic subgroup as above and $P_{0}=M_{0}A_{0}N_{0}$ its Langlands decomposition.
Via symmetrization of the associated quadratic form, the symmetric
$G$-equivariant bilinear form $B$ of (\ref{e: defi B})
gives rise to the  Casimir
operator $\Omega$ in the center of $U(\fg).$ The image $\Delta_{G/H} \in \DGH$
of $\Omega$ under the map (\ref{e: r and DGH}) will be called the Laplacian of $G/H.$
Likewise, the restriction of $B$ to $\faq \times \faq$ gives rise to the Laplacian
$\Delta_{\Aq}$
of $\Aq.$ The Casimir $\Omega_{\fm_0}$of $\fm_0\cap \fk$ is defined by
applying symmetrization to the restriction of $B$ to $\fm_0\cap \fk.$

\begin{Lemma}\label{Lemma H(Delta phi)=(Delta-rho^2)phi}
Let $\Delta_{G/H},$ $\Delta_{A_{\fq}}$ and $\Omega_{\fm_{0}}$ be as above.
Then for $\phi\in\cC(G/H:\tau),$
$$
\Ht_{Q,\tau}(\Delta_{G/H}\phi)
=\big(\tau(\Omega_{\fm_{0}})+\Delta_{A_{\fq}}-\langle\rho_{P_{0}},\rho_{P_{0}}\rangle\big)\Ht_{Q,\tau}\phi.
$$
\end{Lemma}

\begin{proof}
First of all, we observe that (\ref{eq H(D phi)=mu(D)H phi}) is valid for Schwartz functions
$\phi \in \cC(G/H),$
in view of Corollary \ref{Cor H_(Q,tau) extends to Schwartz functions}
 and density of $C_c^\infty(G/H:\tau)$ in $\cC(G/H:\tau).$

Next, from \cite[Lemma 5.3]{vdBanSchlichtkrull_ExpansionsForEisensteinIntegralsOnSemisimpleSymmetricSpaces}
it readily follows that
$$
\mu(\Delta_{G/H}:\tau)
=\tau(\Omega_{\fm_{0}})+\Delta_{A_{\fq}}-\langle\rho_{P_{0}},\rho_{P_{0}}\rangle.
$$
Now apply (\ref{eq H(D phi)=mu(D)H phi}).
\end{proof}

We write $\widehat{M_{0}}_{H}$ for the set of $\xi\in\widehat{M_{0}}$ such that $V(\xi)\neq0$. For $\xi\in\widehat{M_{0}}_{H}$ let $\oC(\tau)_{\xi}$ be the image of $C(K:\xi:\tau)\otimes\overline{V(\xi)}$ under the map $T\mapsto \psi_{T}$ defined by
(\ref{e: defi map psi T}).
Then $\oC(\tau)$ decomposes into a finite direct sum of orthogonal subspaces
$$
\oC(\tau)
=\bigoplus_{  \xi\in  \widehat{M}_{0H} } \oC(\tau)_\xi.
$$
The action of $N_K(\faq)$ on $M_0$ by conjugation
naturally induces a (left) action
of the Weyl group $W(\fa_{\fq}) =N_K(\faq)/Z_K(\faq)$
on $\widehat{M_{0}}_{H}$. We agree to use $W$ as abbreviation for $W(\faq)$ in the rest
of this subsection.  Accordingly, for $\xi\in\widehat{M_{0}}_{H}$ we define
$$
\oC(\tau)_{W\cdot \xi}
:=\bigoplus_{w\in W}\oC(\tau)_{w\xi}.
$$
Accordingly, we obtain the orthogonal decomposition
\begin{equation}\label{eq oC(tau)=oplus oC(tau)_(W xi)}
\oC(\tau)
=\bigoplus_{W\cdot\xi\in W\bs\widehat{M_{0}}_{H}}\oC(\tau)_{W\cdot\xi}.
\end{equation}

We define the normalized $C$-function $C^\circ_{\bar P_0|\bar P_0}(s:\dotvar),$ for $s \in W,$
to be the $\End(\oC(\tau))$-valued meromorphic function on $\faqdc$ given by
\begin{equation}
\label{e: defi nC}
C^\circ_{\bar P_0|\bar P_0}(s:\gl) = C_{\bar P_0|\bar P_0}(s:\gl) C_{\bar P_0|\bar P_0}(1:\gl)^{-1}
\end{equation}
for generic $\gl \in \faqdc;$ see \cite[Eqn.\ (55)]{vdBanSchlichtkrull_FourierTransformOnASemisimpleSymmetricSpace}.
Let $\gamma$ be the (unitary) representation of $W$ in $L^{2}(i\fa_{\fq}^{*})\otimes\oC(\tau)$
defined as in \cite[Eqn.~(16.1)]{vdBanSchlichtkrull_MostContinuousPart}, with $\bar P_0$ in place of $P.$  It is given by
$$
[\gamma(s) f](\lambda)
=C_{\bar P_0|\bar P_0}^{\circ}(s^{-1}:\lambda)^{-1}f(s^{-1}\lambda)
$$
for $s\in W$, $f\in L^{2}(i\fa_{\fq}^{*})\otimes\oC(\tau)$ and $\lambda\in i\fa_{\fq}^{*}$.

\begin{Lemma}
Let $\xi \in \widehat{M}_{0H}.$ Then the subspace
$$
\cS(i\fa^{*})\otimes \oC(\tau)_{W\cdot\xi}\subseteq L^2(i\faqd) \otimes \oC(\tau)
$$
 is invariant for $\gamma$.
\end{Lemma}

\begin{proof}
First, the subspace $\cS(i\faqd) \otimes \oC(\tau)$ is $\gg$-invariant by the argument
suggested in
\cite[Rem.\ 16.3]{vdBanSchlichtkrull_FourierTransformOnASemisimpleSymmetricSpace}.

Next, let $f \in \cS(i\faqd) \otimes \oC(\tau)_\xi,$ where $\xi \in \widehat{M}_{0H}.$
Let $s \in W$ and let $w \in N_K(\faq)$ be a representative for $s.$ Then it suffices
to show that
$$
[\gamma(s) f](\gl) \in  \oC(\tau)_{s\xi}\qquad (\gl \in i\faqd).
$$
In turn, for this it suffices to prove the claim that
\begin{equation}
\label{e: C on oC tau}
C_{\bar P_0|\bar P_0}^{\circ}(s^{-1}:\lambda) \oC(\tau)_{s\xi} \subseteq  \oC(\tau)_\xi,
\end{equation}
for any $\xi \in \widehat{M}_{0H}$ and generic $\gl \in i\faqd.$

By (\ref{e: defi nC}) and \cite[Lemma 7]{vdBanSchlichtkrull_FourierTransformOnASemisimpleSymmetricSpace},
\begin{align*}
C_{\bar P_0|\bar P_0}^{\circ}(s^{-1}:\lambda)
&=  C_{\bar P_0|\bar P_0}(s^{-1}:\lambda)\after C_{\bar P_0|\bar P_0}(1:\lambda)^{-1} \\
&= \cL(s^{-1})\after C_{s\bar P_0s^{-1}|\bar P_0}(1:\lambda)\after C_{\bar P_0|\bar P_0}(1:\lambda)^{-1},
\end{align*}
for generic $\lambda\in\fa_{\fq\C}^{*}$. Here $\cL(s^{-1})$ is given by
\cite[Eqn.~(65)]{vdBanSchlichtkrull_FourierTransformOnASemisimpleSymmetricSpace},
$$
\cL(s^{-1})\psi_{T}
=\psi_{[L(w^{-1})\otimes L(s \xi, w^{-1})]T}
\qquad\big(T\in C(K: s \xi:\tau)\otimes \bar V(s \xi)\big)
$$
where $L(w^{-1}):C^\infty(K: s\xi:\tau) \to C^{\infty}(K: \xi: \tau)$ and $L(s\xi,w^{-1}):\bar V(s\xi)\to \bar V(\xi)$ are linear maps.
Accordingly, we see that $\cL(s^{-1})$ maps $\oC(\tau)_{s\xi}$ to $\oC(\tau)_\xi.$

From \cite[Prop.~1]{vdBanSchlichtkrull_FourierTransformOnASemisimpleSymmetricSpace} and the definition of the $B$-matrix in \cite[Prop.~6.1]{vdBan_PrincipalSeriesI} it follows that $C_{s\bar P_0 s^{-1}|\bar P_0}(1\colon\lambda)\circ C_{\bar P_0|\bar P_0}(1\colon\lambda)^{-1}$ preserves $\oC(\tau)_{s\xi}$.
Thus, (\ref{e: C on oC tau}) follows.
\end{proof}

By \cite[Cor.~17.4]{vdBanSchlichtkrull_MostContinuousPart}, the Fourier transform
$\nFt_{\bar P_0}$ defines a topological linear isomorphism from $\cC_\mc(G/H: \tau)$
onto $[\cS(i\faqd) \otimes \oC(\tau)]^W.$
For $\xi\in\widehat{M_0}_{H}$ we define
$$
\cC_{\mc}(G/H:\tau)_{W\cdot \xi}
:=\nFt_{\bar P_0}^{-1} \big([ \cS(i\fa^{*})\otimes \oC(\tau)_{W\cdot\xi}]^{W}\big).
$$
Then in view of (\ref{eq oC(tau)=oplus oC(tau)_(W xi)}) it follows that
\begin{equation}\label{eq C_mc(G/H:tau)=oplus C_mc(G/H:tau)_(W xi)}
\cC_{\mc}(G/H:\tau)
=\bigoplus_{W\cdot \xi\in W\bs \widehat{M}_{0H}} \cC_{\mc}(G/H:\tau)_{W\cdot\xi}
\end{equation}
with only finitely many non-zero summands.

\begin{Lemma}\label{Lemma phi in C_mc(G/H:tau)_(W xi) => H phi in C(Aq)otimes oC(tau)_(W xi)}
Let $\xi\in\widehat{M_{0}}_{H}$. If $\phi\in\cC_{\mc}(G/H:\tau)_{W\cdot \xi}$ then $\Ht_{Q,\tau}\phi\in C^{\infty}(A_{\fq})\otimes\oC(\tau)_{W\cdot\xi}$.
\end{Lemma}

\begin{proof}
Let $\phi\in\cC_{\mc}(G/H:\tau)_{W\cdot \xi}$. Then $\nFt_{\bar P_{0}}\phi\in\cS(i\fa_{\fq}^{*})\otimes \oC(\tau)_{W\cdot\xi}$. The lemma now follows by application of  Propositions \ref{Prop relation E(Q) and E^circ(cP0)}, \ref{Prop F_Q phi=c F^0_cP0 phi} and \ref{Prop Ft_(Q,tau)=F_A circ H_Q}.
\end{proof}

\begin{Lemma}
\label{l: the scalar c xi}
Let $\xi\in\widehat{M_{0}}_{H}$. Then there exists a scalar $c_\xi \in \R$ such that
$\tau(\Omega_{\fm_{0}})$ acts on $\oC(\tau)_{W\cdot\xi}$ by the scalar $c_{\xi}$.
\end{Lemma}

\begin{proof}
From \cite[Cor.~4.4]{vdBanKuit_EisensteinIntegrals}
it follows that the restriction of $\xi$ to $M_{0}\cap K$ is irreducible. Hence, $\xi(\Omega_{\fm_{0}})$ acts by a real scalar $c_\xi$ on $\cH_\xi$.

Let $s \in W.$
Then $s$ has a representative $w \in N_K(\faq).$ As $\Ad(w)$ preserves the restriction of the bilinear form $B$ to $\fk \cap \fm_0,$  it follows that
$\Ad(w)\Omega_{\fm_0} = \Omega_{\fm_0}.$ This implies that $s\xi(\Omega_{\fm_0})$
acts by the scalar $c_\xi$ on $\cH_{\xi}.$ We infer that
$$
c_{s\xi} = c_\xi,\qquad (s \in W).
$$

Let $s \in W,$ $f\in C(K:s\xi:\tau)$ and $\eta\in\bar V(s\xi)$. Then, with notation
as in (\ref{e: defi map psi T}),
\begin{eqnarray*}
\big(\tau(\Omega_{\fm_{0}})\psi_{f\otimes\eta}\big)(m)
&=& \langle\big(1 \otimes \tau(\Omega_{\fm_{0}})\big)f(m),\eta\rangle\\
&=& \langle\big(s\xi(\Omega_{\fm_{0}})\otimes 1\big)f(m),\eta\rangle\\
& = & c_{\xi}\psi_{f\otimes\eta}(m).
\end{eqnarray*}
The assertion now follows.
\end{proof}

For $\xi\in \widehat{M}_{0H}$ we define the differential operator $D_\xi \in \DGH$ by
$$
D_{\xi}
:=\prod_{\mu\in S_{Q,\tau}}
    \Big(\Delta_{G/H}-c_{\xi}+\langle\rho_{P_{0}},\rho_{P_{0}}\rangle-\langle\mu,\mu\rangle\Big)^{m_{\mu}},
$$
where $m_{\mu} - 1 \geq 0$ is the degree of the $ C^\infty(G/H) \otimes \Hom(\oC(\tau), \Vtau)$-valued
polynomial function $a\mapsto\Res_{\tau}(Q:\mu:a:\dotvar)$.

\begin{Lemma}\label{Lemma H(D phi)=prod (Delta-<mu,mu>)^j H phi}
Let $\xi\in\widehat{M_{0}}_{H}$. Then for every $\phi\in\cC_{\mc}(G/H:\tau)_{W\cdot\xi}$
$$
\Ht_{Q,\tau}\big(D_{\xi}\phi)
=\Big(\prod_{\mu\in S_{Q,\tau}}\Big(\Delta_{A_{\fq}}-\langle\mu,\mu\rangle\Big)^{m_{\mu}}\Big)\Ht_{Q,\tau}\phi.
$$
\end{Lemma}

\begin{proof}
The lemma follows directly from Lemma \ref{Lemma H(Delta phi)=(Delta-rho^2)phi},
 Lemma \ref{Lemma phi in C_mc(G/H:tau)_(W xi) => H phi in C(Aq)otimes oC(tau)_(W xi)}
 and Lemma \ref{l: the scalar c xi}.
\end{proof}

\begin{Lemma}\label{Lemma Existence of solution of D chi=phi}
Let $\xi\in\widehat{M_{0}}_{H}$. Then for every $\phi\in\cC_{\mc}(G/H:\tau)_{W\cdot\xi}$ there exists a $\chi\in\cC_{\mc}(G/H:\tau)_{W\cdot\xi}$ such that
$$
D_{\xi}\chi
=\phi.
$$
\end{Lemma}

\begin{proof}
For $\gl \in i\faqd$ we define  $e_\gl: A_{\fq}\to\C,\,$ $a \mapsto a^\gl.$
Moreover, for  $D\in\D(G/H)$
we define $\mu(D:\tau:\lambda)\in\End(\oC(\tau))$ by
$$
\mu(D:\tau: \lambda)\psi
:=\big(\mu (D:\tau)(e^\gl \otimes \psi)\big)(e)
\qquad\big(\psi\in\oC(\tau)\big);
$$
here $\mu(D:\tau)$ is the $\End(\oC(\tau))$-valued differential operator
on $\Aq$ defined in (\ref{e: intro mu tau}).

Then for all $D\in\D(G/H)$, $\phi\in\cC(G/H:\tau)$ and $\lambda\in i\fa_{\fq}^{*}$,
we have
$$
\nFt_{\bar P_0}\big(D\phi)(\lambda)
=\mu(D: \tau: \lambda)\nFt_{\bar P_0}\phi(\lambda);
$$
see \cite[Lemma 6.2]{vdBanSchlichtkrull_MostContinuousPart}.
In particular, it follows that
$$
\nFt_{\bar P_0 }\big(D_{\xi}\phi)(\lambda)
=\prod_{\mu\in S_{Q,\tau}}\Big(\langle\lambda,\lambda\rangle-\langle\mu,\mu\rangle\Big)^{m_{\mu}}
    \nFt_{\bar P_0}\phi(\lambda).
$$
Note that $\langle\lambda,\lambda\rangle-\langle\mu,\mu\rangle\neq0$ for $\lambda\in i\fa_{\fq}^{*}$ and $\mu\in S_{Q,\tau}\subseteq \fa_{\fq}^{*}\setminus \{0\}$.

Now let $\phi\in\cC_{\mc}(G/H:\tau)_{W\cdot \xi}$. Then the function
$
f:i\fa_{\fq}^{*}\to\oC(\tau)$ defined by
$$
f(\gl)=\prod_{\mu\in S_{Q,\tau}}\Big(\langle\lambda,\lambda\rangle-\langle\mu,\mu\rangle\Big)^{- m_{\mu}} \;  \nFt_{\bar P_0}\phi(\lambda)
$$
belongs to the space $\big(S(i\fa_{\fq}^{*})\otimes\oC(\tau)_{W\cdot\xi}\big)^{W}$.
The assertion of  the lemma now follows with $\chi=\nFt_{\bar P_0}^{-1}f$.
\end{proof}

\begin{proof}[Proof of Theorem \ref{Thm most cont part orthogonal to residues}]
Since $\cC_{\mc}(G/H:\tau)$ decomposes as a finite direct sum (\ref{eq C_mc(G/H:tau)=oplus C_mc(G/H:tau)_(W xi)}), it suffices to prove the assertion for $\phi\in\cC_{\mc}(G/H:\tau)_{W\cdot\xi}$.
Let $\phi\in\cC_{\mc}(G/H:\tau)_{W\cdot\xi}$ and let $\chi$ be as in Lemma \ref{Lemma Existence of solution of D chi=phi}. Then by Lemma \ref{Lemma H(D phi)=prod (Delta-<mu,mu>)^j H phi}
$$
\Ht_{Q,\tau}\phi
=\Big(\prod_{\mu\in S_{Q,\tau}}\Big(\Delta_{A_{\fq}}-\langle\mu,\mu\rangle\Big)^{m_{\mu}}\Big)\Ht_{Q,\tau}\chi.
$$
Since $a\mapsto\inp{\chi}{\Restau(Q:\mu: a:\dotvar)\psi}$ is a polynomial function of degree $m_{\mu}-1$, it follows that
$$
\Big(\Delta_{A_{\fq}}-\langle\mu,\mu\rangle\Big)^{m_{\mu}}a^\mu \inp{\chi}{\Restau(Q:\mu: a:\dotvar)\psi}
=0,
$$
hence
$$
\Ht_{Q,\tau}\phi
=\Big(\prod_{\mu\in S_{Q,\tau}}\Big(\Delta_{A_{\fq}}-\langle\mu,\mu\rangle\Big)^{m_{\mu}}\Big)\It_{Q,\tau}\chi.
$$
In particular, $\Ht_{Q,\tau}\phi$ belongs to $C^{\infty}_{\temp}(A_{\fq})\otimes\oC(\tau).$
Since also $\It_{Q,\tau}\phi$ belongs to this space, by Proposition \ref{Prop JQ: C(G/H:tau) to C^infty(Aq)}, we infer that
$$
a \mapsto \sum_{\mu \in S_{Q,\tau}} a^\mu \inp{\Restau(Q:\mu: a: \dotvar)}{\phi}
$$
belongs to this space. Since the latter sum  is also an exponential polynomial function
with non-zero exponents on $\Aq,$ with values in $\oC(\tau),$
it must be zero and we finally conclude
(\ref{eq <phi, Res(Q:mu:a)>=0}).
\end{proof}

We now come to the second theorem of this subsection, which asserts
that in fact the residual functions from Corollary \ref{c: R tempered} are constant
as functions of the variable from $\Aq.$
From Corollary \ref{c: R tempered} we recall that for $\mu \in S_{Q,\tau},$
the function $\Restau(Q,\mu)$ belongs to $P(\Aq) \otimes C^\infty_\temp(G/H: \tau) \otimes
\oC(\tau)^*.$

\begin{Thm}
\label{t: Res constant in a}
Let $\mu \in S_{Q,\tau}.$
\begin{enumerate}
\itema The function $\Restau(Q:\mu)$ is constant with respect to the
variable from $\Aq$ and belongs to $\cC_\ds(G/H:\tau) \otimes \oC(\tau)^*.$
\itemb The meromorphic $C^\infty(G/H) \otimes \Hom(\oC(\tau), V_\tau)$-valued
function $E(Q: - \dotvar)$ on $\faqdc$ has a pole of order $1$ at $\mu.$
\end{enumerate}
\end{Thm}

\begin{proof}
By \cite[Prop.~17.7]{vdBanSchlichtkrull_MostContinuousPart}
$$
\cC(G/H:\tau)
=\cC_{\mc}(G/H:\tau)\oplus\cC_{\ds}(G/H:\tau)
$$
as an orthogonal direct sum.
By Corollary \ref{c: R tempered} the finite dimensional space
$$
\findim =\spn\{\Restau(Q:\mu:a)\psi:\mu\in S_{Q,\tau},a\in A_{\fq},\psi\in\oC(\tau)\}
$$
is contained in $C^{\infty}_{\temp}(G/H:\tau)\subseteq \cC'(G/H:\tau)$. Since $\cC_{\ds}(G/H:\tau)$ is finite dimensional, see \cite{Oshima_Matsuki_ds} and  \cite[Lemma 12.6 \& Rem.\ 12.7]{vdBan_Schlichtkrull_Plancherel1}, there exists for every $\chi\in F$ a function $\vartheta\in\cC_{\ds}(G/H:\tau)$ such that
\begin{equation}\label{eq <phi,chi>=<phi,theta>}
\langle\phi,\chi\rangle
=\langle\phi,\vartheta\rangle
\end{equation}
for every $\phi\in\cC_{\ds}(G/H:\tau)$.
By Theorem \ref{Thm most cont part orthogonal to residues} the space $\cC_{\mc}(G/H:\tau)$ is perpendicular to $F$. Hence,  (\ref{eq <phi,chi>=<phi,theta>}) is valid
for every $\phi\in\cC(G/H:\tau)$, and we conclude that $\chi=\vartheta$. This proves that $F\subseteq\cC_{\ds}(G/H:\tau)$.

The space $\cC_{\ds}(G/H:\tau)$ decomposes as a finite direct sum
$$
\cC_{\ds}(G/H:\tau)
=\bigoplus_{\pi}\cC(G/H:\tau)_{\pi},
$$
where $\pi$ runs over the representations of the discrete series
and $\cC(G/H:\tau)_{\pi}$ is spanned by left $\tau$-spherical and right $H$-fixed generalized matrix coefficients of $\pi$.

Let $\pi$ be a discrete series representation, $\phi\in\cC(G/H:\tau)_{\pi}$
and $\psi \in \oC(\tau).$
We will establish (a) by showing that  $\inp{\phi}{\Restau(Q:\mu: a:\dotvar)\psi}$ is independent of $a\in A_{\fq}.$  Since (\ref{eq oC(tau)=oplus oC(tau)_(W xi)}) is a finite direct sum, we may assume that $\psi\in\oC(\tau)_{W\cdot\xi}$ for a representation
$\xi\in\widehat{M}_{0H}$.

By Corollary \ref{c: HQtau as tempered term plus residues for Schwartz} and Corollary \ref{c: ker It Q tau is ds},
\begin{equation}\label{eq Ht phi_ds =Res phi_ds}
\langle\Ht_{Q,\tau}\phi(a),\psi\rangle
=\sum_{\mu \in S_{Q,\tau}}a^\mu  \inp{\phi}{\Restau(Q:\mu: a:\dotvar)\psi}
\qquad(a\in A_{\fq}).
\end{equation}
Since $\pi$ is irreducible, $\phi$ is an eigenfunction for $\gD_{G/H} = R_\Omega.$
 Let $c$ be its eigenvalue. Then by Lemma \ref{Lemma H(Delta phi)=(Delta-rho^2)phi}
$$
\big(\tau(\Omega_{\fm_{0}})+\Delta_{A_{\fq}}-\langle\rho_{P_{0}},\rho_{P_{0}}\rangle-c\big)\Ht_{Q,\tau}\phi
=0.
$$
Since $\tau(\Omega_{\fm_{0}})$ is symmetric, we have for $\chi\in\oC(\tau)$
$$
\langle\tau(\Omega_{\fm_{0}})\chi,\psi\rangle
=\langle\chi,\tau(\Omega_{\fm_{0}})\psi\rangle
=c_{\xi}\langle\chi,\psi\rangle.
$$
We thus see that
$$
\big(\Delta_{A_{\fq}}+c_{\xi}-\langle\rho_{P_{0}},\rho_{P_{0}}\rangle-c\big)\langle\Ht_{Q,\tau}\phi,\psi\rangle
=0.
$$
The solutions to this differential equation are either polynomial functions (in case $c_{\xi}-\langle\rho_{P_{0}},\rho_{P_{0}}\rangle-c=0$) or a sum of exponential functions (in case $c_{\xi}-\langle\rho_{P_{0}},\rho_{P_{0}}\rangle-c\neq 0$). By comparing
with (\ref{eq Ht phi_ds =Res phi_ds}) we see that $\langle\Ht_{Q,\tau}\phi,\psi\rangle$
cannot be purely polynomial, hence must be a finite sum of exponentials functions. This establishes (a).

We turn to (b). We define the linear function  $p: \faqdc \to \C$
by $p(\gl) = \inp{\gl + \mu}{\omega},$ where $\omega$ is the unique
unit vector in $\faqd(P_0).$ We note that
$$
p(-\mu + z\omega) = z,\qquad (z \in \C).
$$
The function $E(Q: - \dotvar)$ has a singularity at $\mu,$ by definition
of the set $S_{Q,\tau}.$ Reasoning by contradiction, assume that (b) is not valid. Then
there exists an element $\psi \in \oC(\tau)$ and a $k \geq 1$ such that $p^{k+1} E(Q: \psi: - \dotvar)$ is regular at $\mu$ and has a non-zero value at that point. Then with the notation of (\ref{e: definition of residue}) and Lemma \ref{l: residues one}, it follows
that
\begin{eqnarray*}
\Restau(Q: \mu: a:x)\psi
&=& - \Res_{\gl = - \mu} a^{-\gl - \mu} E(Q: \psi: \gl)(x)\\
&=& - \Res_{z = 0} a^{-z\omega} E(Q: \psi: -\mu + z \omega)(x)\\
&=& - \left. \frac{d^k}{dz^k}\right|_{z = 0} z^{k+1} a^{-z\omega} E(Q: \psi: - \mu + z\omega)(x)\\
&=& \sum_{j=0}^k  \omega(\log a)^j \chi_j(x)
\end{eqnarray*}
with uniquely determined functions $\chi_j \in C^\infty(G/H:\tau).$
By the Leibniz rule it follows in particular that
$$
\chi_k(x) = (-1)^{k+1} \omega(\log a)^k [z^{k+1} E(Q:\psi: -\mu + z\omega)]_{z = 0}.
$$
Since
$$
[z^{k+1} E(Q:\psi: -\mu + z\omega)]_{z = 0} = [p(\gl)E(Q:\psi:\gl]_{\gl = -\mu} \neq 0
$$
this implies that $\Restau(Q: \mu: a:x)(\psi)$ is not constant in $a,$ contradicting (a).
Hence (b) is valid.
\end{proof}

In view of Theorem \ref{t: Res constant in a} we now obtain the following
version of Corollary \ref{c: HQtau as tempered term plus residues for Schwartz}.

\begin{Cor}
Let $\phi\in\cC(G/H:\tau)$. Then for all $\psi \in \oC(\tau)$ and $a \in \Aq,$
\begin{equation}
\label{e: H minus I as pol free residues}
\langle\Ht_{Q,\tau}\phi(a),\psi\rangle
=\langle\It_{Q,\tau}\phi(a),\psi\rangle+\sum_{\mu\in S_{Q,\tau}}a^{\mu}\langle\phi,\Restau(Q,\mu)\psi\rangle.
\end{equation}
\end{Cor}

\subsection{Cusp Forms}
\label{ss: cusp forms}
In this final subsection we keep working under  the
\medno
{\bf Assumption:\ } {\em $G/H$ is of split rank one.}
\medno
The following definition makes use
of the Radon transform introduced in  Definition \ref{Defi Radon and HC-transform}.

\begin{Defi}\label{Defi cusp forms}
A function $\phi\in\cC(G/H)$ is called a cusp form if $\Rt_{Q}\phi=0$ for every $Q\in\cPH(A)$. We write $\cC_{\cusp}(G/H)$ for the subspace of such cusp forms in $\cC(G/H).$
\end{Defi}

\begin{Lemma}
\label{l: space of cusp forms closed}
$\cC_\cusp(G/H)$ is a $G$-invariant closed subspace of $\cC(G/H).$
\end{Lemma}

\begin{proof}
This follows immediately from Theorem \ref{Thm convergence}
and the $G$-equivariance of $\Rt_Q,$ for every $Q \in \cP_\fh(A).$
\end{proof}

Recall that a parabolic subgroup $Q\in\cP(A)$ is said to be
$\fh$-extreme if $\Sigma(Q,\sigma\theta)=\Sigma(Q)\cap\faq^{*}$.

\begin{Lemma}
\label{l: minimal defi cusp form}
Let $\phi \in \cC(G/H).$ Then the following conditions are equivalent.
\begin{enumerate}
\itema $\phi$ is a cusp form.
\itemb $\Rt_Q\phi = 0$ for every $\fh$-extreme parabolic subgroup
$Q \in \cPH(A).$
\end{enumerate}
\end{Lemma}

\begin{proof}
Clearly, (b) follows from (a). For the converse, assume that (b) holds.
Let $P \in \cPH(A).$ There exists a $\fh$-extreme $Q \in \cP(A)$ such that
$P \succeq Q,$ see \cite[Lemma 2.6]{BalibanuVdBan_ConvexityTheorem}. By Lemma
\ref{l: H compatible parabolics} (c)
we see that $Q \in \cPH(A),$
so that $\Rt_{Q}\phi=0$.
Since the integral for $\Rt_P\phi(g)$ is absolutely convergent, for
every $g \in G,$ it now follows by application of Corollary \ref{c: dominant radon integral}
that
$$
\Rt_P\phi(g) = \int_{N_{P} \cap \bar N_Q} \Rt_Q\phi(gn) dn = 0.
$$
Thus, (a) follows.
\end{proof}

\begin{Rem}
\label{r: notion of cusp forms coincides with AFJS}
It follows from this result that for the class of real hyperbolic spaces
$\SO(p,q+1)_e/\SO(p, q)_e$
our notion of cusp form coincides with
the one introduced by
\cite[Eqn. (5)]
{AndersenFlenstedJensenSchlichtkrull_CuspidalDiscreteSerieseForSemisimpleSymmetricSpaces}.
Indeed, the minimal parabolic subgroup mentioned in the text following \cite[Eqn. (5)]{AndersenFlenstedJensenSchlichtkrull_CuspidalDiscreteSerieseForSemisimpleSymmetricSpaces}
is $\fh$-extreme in our sense, and it turns out the condition of $\fh$-compatibility is fulfilled.
In fact, it is easy to see that for this family of symmetric spaces the properties of $\fh$-compatibility and $\fh$-extremeness coincide.
\end{Rem}

\begin{Rem}
\label{r: K finite cusp forms}
Let $\vartheta$ be a finite subset of $\widehat{K}$.
For a representation of $K$ on a vector space $V,$ we denote by $V_\vartheta$  the subspace of $K$-finite vectors with isotypes contained in $\vartheta$.

Let $C(K)_\vartheta$ be the space of $K$-finite continuous functions
on $K,$ whose right $K$-types belong to $\vartheta$
and let $\tau$ denote the right
regular representation of $K$ on $V_{\tau}:= C(K)_{\vartheta}$. Then the canonical map
$$
\varsigma:\cC(G/H)_{\vartheta}\to\cC(G/H:\tau)
$$
given by
$$
\varsigma\phi(x)(k)=\phi(kx)
\qquad\big(\phi\in\cC(G/H)_{\vartheta}, k\in K, x\in G/H\big)
$$
is a linear isomorphism.
Let $\phi\in\cC(G/H)_{\vartheta}$. Then it follows from (\ref{eq relation H_(Q,tau) and H_Q}), see Corollary \ref{Cor H_(Q,tau) extends to Schwartz functions}, that $\Ht_{Q^{v}}\phi=0$ for every $v\in\cW$ if and only if $\Ht_{Q,\tau}(\varsigma\phi)=0$. Hence,
 $\phi\in\cC_{\cusp}(G/H)$ if and only if $\Ht_{Q,\tau}(\varsigma\phi)=0$ for every $Q\in\cPH(A)$.
\end{Rem}

\begin{Ex}[\bf Group case]
We use notation as in Example \ref{Ex Group case - HC-transform def compared to HC's def}.

Every minimal parabolic subgroup is $\fh$-compatible (see Example \ref{Ex Group case - H-compatible parabolic subgroups}); the $\fh$-extreme parabolic subgroups are all of the form $\bp P\times\bp P$ where $\bp P$ is a minimal parabolic subgroup of $\bp G$. As explained in Example \ref{Ex Group case - HC-transform def compared to HC's def}, the Radon transform $\Rt_{\bp P\times\bp P}$ is identified with $\Rt_{\bp P}$ under the identification $(\bp G\times\bp G)/\diag(\bp G)$.
From Lemma \ref{l: minimal defi cusp form}
 it now
 follows that $\phi\in\cC(G/H)$ is a cusp form if and only if $\Rt_{\bp P}\phi=0$ for every minimal parabolic subgroup $\bp P$ of $\bp G$. Using the fact that $\bp G$ acts transitively on the set of minimal parabolic subgroups of $\bp G$, it is easily seen that this in turn is equivalent to the condition that
$$
\int_{N_{\bp P}}\phi(gn)\,dn=0
$$
for every $g\in \bp G$ and every minimal parabolic subgroup $\bp P$.
Thus,
if $\bp G$ is a reductive Lie group of the Harish-Chandra class of split rank $1$, then our definition of cusp forms coincides with the definition of Harish-Chandra; see \cite[Part II, Sect.~12.6, p. 222]{Varadarajan_HarmonicAnalysisOnRealReductiveGroups}.
\end{Ex}

We now move on to investigate the relation between $\cC_{\cusp}(G/H)$ and $\cC_{\ds}(G/H)$.
Our main tool will be the identity
(\ref{e: H minus I as pol free residues}).
The following result is a straightforward consequence of Theorem \ref{t: vanishing Rt Q on Lone}.

\begin{Cor}
\label{c: L one ds are cuspidal}
$\cC_{\ds}(G/H)\cap L^1(G/H)\subseteq   \cC_{\cusp}(G/H)$.
\end{Cor}

\begin{Lemma}\label{Lemma H_(Q,tau)phi=0 implies phi in C_ds}
Let $Q\in\cP_{\fh}(A)$ and $\phi\in\cC(G/H:\tau)$. Then $\Ht_{Q,\tau}\phi=0$ if and only if both (a) and (b) hold,
\begin{enumerate}[(a)]
\item $\It_{Q,\tau}\phi=0$,
\item $ \phi \perp \Restau(Q,\mu)\psi,$ {\ } for every $\mu\in S_{Q,\tau}$ and $\psi\in\oC(\tau)$.
\end{enumerate}
In particular, if $\Ht_{Q,\tau}\phi=0$, then $\phi\in\cC_{\ds}(G/H:\tau)$.
\end{Lemma}

\begin{proof}
Assume $\Ht_{Q,\tau}\phi=0$.
From (\ref{e: H minus I as pol free residues}) and Theorem \ref{t: Res constant in a} it follows
for every $\psi\in\oC(\tau)$ that the function $\langle\Ht_{Q,\tau}\phi(\dotvar),\psi\rangle$ equals a sum of the tempered term $\langle\It_{Q,\tau}\phi(\dotvar),\psi\rangle$ and finitely many exponential terms with non-zero real exponents.  Since $\Ht_{Q,\tau}\phi=0$, it follows that all the mentioned terms vanish. This proves (a) and (b) in the lemma. The converse implication follows directly from (\ref{e: H minus I as pol free residues}).

Finally, if (b) holds, then we infer from Corollary \ref{c: ker It Q tau is ds} that $\phi\in\cC_{\ds}(G/H:\tau)$. This concludes the proof of the lemma.
\end{proof}

\begin{Thm}
\label{t: inclusion cusp in ds}
$\;\;\cC_{\cusp}(G/H)\subseteq   \cC_{\ds}(G/H).$
\end{Thm}

\begin{proof}
In view of Lemma \ref{l: space of cusp forms closed} it suffices
to show that every $K$-finite cusp form is an element of $\cC_{\ds}(G/H)$.
Let $\vartheta\subset\widehat{K}$ be finite and let $\tau$ and $\varsigma$ be as in Remark \ref{r: K finite cusp forms}. Let $\phi\in\cC_{\cusp}(G/H)_{\vartheta}$ and assume that $Q\in\cPH(A)$. Then $\Ht_{Q,\tau}(\varsigma\phi)=0$ by Remark \ref{r: K finite cusp forms} and thus it follows from Lemma \ref{Lemma H_(Q,tau)phi=0 implies phi in C_ds} that $\varsigma\phi\in\cC_{\ds}(G/H:\tau)$.
Hence, $\phi\in\cC_{\ds}(G/H)_{\vartheta}$.
\end{proof}

\begin{Rem}
\label{r: inclusion cusp in ds proper}
There exist
symmetric spaces for which the inclusion of Theorem \ref{t: inclusion cusp in ds} is proper.
Indeed, in \cite[Thm.\ 5.3]{AndersenFlenstedJensenSchlichtkrull_CuspidalDiscreteSerieseForSemisimpleSymmetricSpaces} it has been established that the mentioned inclusion is proper for $G=\SO(p,q+1)_{e}$ and $H=\SO(p,q)_{e}$, with $1\leq p<q-1$.
\end{Rem}

\begin{Thm}
\label{t: K fixed ds cusp forms implies all}
If {\ } $\cC_{\ds}(G/H)^{K}\subseteq   \cC_{\cusp}(G/H)$, then $\cC_{\ds}(G/H)=\cC_{\cusp}(G/H)$.
\end{Thm}

\begin{proof}
Assume $\cC_{\ds}(G/H)^{K}\subseteq  \cC_{\cusp}(G/H)$ and let $Q\in\cPH(A)$.
Let $\findim$ be the subspace of $C^\infty(G/H)^K$ spanned by the functions
$\Resone(Q:\mu)\psi$, for $\mu \in S_{Q,1}$ and $\psi \in \oC(1).$
Then $\findim \subseteqq\cC_\ds(G/H)^K$ by Proposition \ref{p: R Q  in span of finite sum of discrete series}.
We claim that $\findim = 0.$ To see this, let $\chi \in \findim.$ Then by the assumption and Remark \ref{r: K finite cusp forms}
it follows that
$$
\Ht_{Q,\oneK}(\chi) = 0.
$$
By Lemma \ref{Lemma H_(Q,tau)phi=0 implies phi in C_ds}  (b) it follows that
$\chi \perp \findim.$ Hence,  $\chi  = 0$ and the claim is established.

We conclude from the claim and (\ref{e: H minus I as pol free residues})
that $\Ht_{Q,\oneK}=\It_{Q,\oneK}$.
Let $\phi\in\cC(G/H)$. By Proposition \ref{Prop domination by K-inv Schwartz functions} there exists a $\widehat{\phi}\in\cC(G/H)^{K}$ such that $|\phi|\leq\widehat{\phi}$. Now
\begin{equation}
\label{e: estimate H Q phi}
|\Ht_{Q}\phi|
\leq\Ht_{Q}|\phi|
\leq\Ht_{Q}\widehat{\phi}.
\end{equation}
Let $\psi_{0}$ be the element of $\oC(\oneK)\simeq\C^{\cW}$ determined
by $(\psi_{0})_w = \gd_{1,w}.$
Then $$\Ht_{Q}\widehat{\phi}(a)=\langle\Ht_{Q,\oneK}\widehat{\phi}(a),\psi_{0}\rangle$$ by (\ref{eq relation H_(Q,tau) and H_Q}), see Corollary \ref{Cor H_(Q,tau) extends to Schwartz functions}.
It follows from Proposition \ref{Prop JQ: C(G/H:tau) to C^infty(Aq)} that $\langle\Ht_{Q,\oneK}\widehat{\phi}(\dotvar),\psi_{0}\rangle = \langle\It_{Q,\oneK}\widehat{\phi}(\dotvar),\psi_{0}\rangle$ is a tempered function on $A_{\fq}$. In view of the estimate (\ref{e: estimate H Q phi})
we conclude that $\Ht_{Q}$ maps the functions from $\cC(G/H)$
to tempered functions on $L/H_{L}$.
The same holds for $\Ht_{Q^{v}}$ with $v\in\cW$.

Let $(\tau,V_{\tau})$ be a finite dimensional representation of $K$. From (\ref{eq relation H_(Q,tau) and H_Q}) and Corollary \ref{Cor H_(Q,tau) extends to Schwartz functions} we conclude that $\Ht_{Q,\tau}$ maps the functions from $\cC(G/H\colon\tau)$ to tempered $\oC(\tau)$-valued functions on $A_{\fq}$. It follows that the exponential terms in (\ref{e: H minus I as pol free residues}) are all equal to zero. Therefore, $\Ht_{Q,\tau}=\It_{Q,\tau}$ and by Corollary \ref{c: ker It Q tau is ds} we have that
$$
\cC_{\ds}(G/H:\tau)=\ker\Ht_{Q,\tau}.
$$
Since this holds for every finite dimensional representation $\tau$ of $K$, we conclude that $\cC_{\ds}(G/H)\subseteq   \ker(\Rt_{Q})$. As $Q \in \cPH(\Aq)$
was arbitrary, we conclude that $\cC_\ds(G/H) \subseteq \cC_\cusp(G/H).$
The converse inclusion was established in Theorem \ref{t: inclusion cusp in ds}.
\end{proof}

\begin{Rem}
\label{r: non spherical residual part}
At present, we know of no example of a split rank one symmetric space where
the  equality
$
\cC(G/H)^K \cap \cC_\cusp(G/H) = \{0\}
$
is violated.
On the other hand, $\cC_\ds(G/H) \cap \cC_{\cusp}(G/H)^\perp$  may contain irreducible submodules that are not
spherical. An example of a symmetric pair for which this happens is provided by
 $G=\SO(p,q+1)_{e}$ and $H=\SO(p,q)_{e}$, with $1\leq p<q-3$; see
 \cite[Thm.~5.3]{AndersenFlenstedJensenSchlichtkrull_CuspidalDiscreteSerieseForSemisimpleSymmetricSpaces}.
 \end{Rem}

Let $(\tau,\Vtau)$ be a finite dimensional unitary representation of $K.$ We define
$\cC_\cusp(G/H: \tau)$ to be the intersection of $\cC(G/H:\tau)$ with $\cC_\cusp(G/H) \otimes \Vtau.$ Furthermore, we define $\cC_\res(G/H: \tau)$ to be the $L^2$-orthocomplement of $\cC_\cusp(G/H: \tau)$ in $\cC_\ds(G/H:\tau).$ Then by
finite dimensionality of the latter space, we have the following direct sum decomposition,
\begin{equation}
\label{e: ds is res plus cusp}
\cC_\ds(G/H:\tau) = \cC_\res(G/H: \tau) \oplus \cC_\cusp(G/H: \tau).
\end{equation}

\begin{Thm}
\label{t: spherical cCres generated by Res}
Let $\cP_{\fh\fh}(A)$ denote the set of $\fh$-extreme parabolic subgroups in $\cP_\fh(A).$
Then
\begin{enumerate}
\itema $\cC_\cusp(G/H:\tau) = \{ \phi \in \cC(G/H:\tau) \col  \Ht_{Q,\tau} \phi = 0\; (\forall Q \in \cP_{\fh\fh}(A))\};$
\itemb $\cC_\res(G/H:\tau) = \spn\{\Restau(Q:\mu)\psi \col Q \in \cP_{\fh\fh}(A), \,\mu \in S_{Q,\tau},\;\psi \in \oC(\tau)\}.$
\end{enumerate}
\end{Thm}

\begin{proof}
If $\phi \in \cC_\cusp(G/H:\tau)$ and $Q \in \cP_{\fh\fh}(A),$
then in view of (\ref{eq def H_(Q,tau)}), see Corollary \ref{Cor H_(Q,tau) extends to Schwartz functions},
it follows that $\Ht_{Q,\tau}\phi = 0.$ This establishes one
inclusion. For the converse inclusion, assume that $\phi\in \cC(G/H:\tau)$ belongs
to the set on the right-hand side. Let $Q \in \cP_{\fh\fh}(A).$
Then it follows from (\ref{eq def H_(Q,tau)}), see Corollary \ref{Cor H_(Q,tau) extends to Schwartz functions},
that $(\Rt_Q \otimes I_{V_\tau})(\phi)$ vanishes on $L N_Q.$
By sphericality of $\phi$ and $G$-equivariance of $\Rt_Q$ it follows
that $(\Rt_Q \otimes I_{V_\tau})(\phi) = 0.$ By Lemma \ref{l: minimal defi cusp form}
we infer that $\phi \in \cC_\cusp(G/H:\tau).$

We now turn to (b). Let $\phi\in\cC_{\ds}(G/H:\tau)$ and let $Q \in \cP_{\fh\fh}(A)$. As $\It_{Q,\tau}$ vanishes on
$\cC_\ds(G/H:\tau)$ by Corollary \ref{c: ker It Q tau is ds},
it follows from Lemma \ref{Lemma H_(Q,tau)phi=0 implies phi in C_ds} that $\Ht_{Q,\tau}\phi=0$ if and only if $\phi$ is perpendicular to $\Restau(Q,\mu)\psi$ for every $\mu\in S_{Q,\tau}$ and $\psi\in\oC(\tau)$. Therefore the space on the right-hand side of (b) equals the orthocomplement in $\cC_{\ds}(G/H:\tau)$ of the space on the right-hand side of (a). Now (b) follows from (a) by the orthogonality of the direct sum (\ref{e: ds is res plus cusp}).
\end{proof}

\begin{Rem}
\label{r: finitely many non-cuspidal ds in AFJS}
In \cite[Thm.~5.2 \& 5.3]{AndersenFlenstedJensenSchlichtkrull_CuspidalDiscreteSerieseForSemisimpleSymmetricSpaces}
it is shown that for the real hyperbolic spaces $\SO(p, q+1)_e/\SO(p,q)_e,$
the left regular representation on $L^2_\ds(G/H) \cap \cC_\cusp(G/H)^\perp$ is a finite direct sum of discrete series  for $G/H$ and these are explicitly identified.
\end{Rem}

We conclude this article by giving a characterization of $K$-finite functions in $\cC_{\ds}(G/H)$.

\begin{Thm}
\label{t: characterisation ds by Radon}
Let $\phi$ be a $K$-finite function in $\cC(G/H)$.
Then the following assertions are equivalent.
\begin{enumerate}
\itema
$\phi\in\cC_{\ds}(G/H)$
\itemb
For every $Q\in\cPH(A)$ and every $g \in G$ the function
\begin{equation}
\label{e: a mapsto H Q tau phi psi}
A_{\fq}\ni a\mapsto\Rt_{Q}\phi( g a)
\end{equation}
is a finite linear combination of real exponential functions.
\itemc
There exists an $\fh$-extreme  $Q \in\cPH(A)$ such that for
every $v \in \cW$ and every $k \in K$ the function
$$
A_{\fq}\ni a\mapsto \Rt_{Q^v}\phi(k a)
$$
is a finite linear combination of real exponential functions.
\end{enumerate}
\end{Thm}

\begin{proof}
Let $\vartheta$ be a finite subset of $\widehat{K}$ such that
$\phi \in \cC(G/H)_\vartheta$ and let $\tau$ and $\varsigma$ be as in Remark
\ref{r: K finite cusp forms}.

Assume that (a) is valid and let $Q\in\cPH(A)$.
Then $\It_{Q,\tau}(\varsigma\phi)=0$ by Theorem \ref{Thm ker(F^0_(P0,tau))=ker(J_(Q,tau))}. Therefore, only the exponential
 terms on the right-hand side of
 (\ref{e: H minus I as pol free residues})
can be non-zero. From the relation between $\Ht_{Q,\tau}$ and $\Ht_{Q}$ as given in (\ref{eq relation H_(Q,tau) and H_Q}), see Corollary \ref{Cor H_(Q,tau) extends to Schwartz functions}, it follows that the function (\ref{e: a mapsto H Q tau phi psi}) is of
real exponential type, i.e., a finite linear combination of exponential functions with real exponents, if $g =e.$ For $g = k \in K$ the assertion now follows from the $K$-equivariance of $\Rt_Q.$
Let $g \in G$ be general, then $g = k a_0 n_Q$ according to the Iwasawa decomposition $G = K A N_Q.$ Furthermore,
$$
\Rt_Q\phi(ga) = \Rt_Q\phi(ka_0 a (a^{-1} n_Q a)) =
\Rt_Q\phi(k (a_0 a))$$
 and we see that (\ref{e: a mapsto H Q tau phi psi})
is of real exponential type. Hence, (b) follows.

Clearly, (b) implies (c). Now assume (c) and let $Q$ be an $\fh$-extreme
parabolic subgroup in $\cPH(A)$ with the asserted properties.
It follows from (\ref{eq relation H_(Q,tau) and H_Q}), see Corollary \ref{Cor H_(Q,tau) extends to Schwartz functions},
that for every $\psi\in\oC(\tau)$ the function $\langle\Ht_{Q,\tau}(\varsigma\phi)(\dotvar),\psi\rangle$ is of real exponential type.
From (\ref{e: H minus I as pol free residues})
we then read off that $\langle\It_{Q,\tau}(\varsigma\phi)(\dotvar),\psi\rangle$ is of such exponential type as well. It now suffices to prove the claim that $\It_{Q,\tau}(\varsigma\phi)$
is in fact equal to $0$.  Indeed, from the claim it follows that $\phi\in\cC_{\ds}(G/H)_{\vartheta}$ by Corollary \ref{c: ker It Q tau is ds}.
Hence (a).

It remains to prove the above claim.
We established for every $\psi\in\oC(\tau)$ that  the function $\langle\It_{Q,\tau}(\varsigma\phi)(\dotvar),\psi\rangle$ is of real exponential type. Since this function is tempered in view of Proposition \ref{Prop JQ: C(G/H:tau) to C^infty(Aq)}, it has to be constant.
As this is valid for every $\psi\in\oC(\tau)$, the support of  $\eFt_{A_{\fq}}\big(\It_{Q,\tau}(\varsigma\phi)\big)$ is contained in the origin. Now it follows from (\ref{eq J_Q= u*K_Q}) that $\eFt_{A_{\fq}}\big(\Kt_{Q,\tau}(\varsigma\phi)\big)$ is supported in the origin as well. As the latter is a smooth function, it must vanish identically.
It then follows from (\ref{eq J_Q= u*K_Q}) that also $\It_{Q,\tau}(\varsigma\phi)=0$. The validity of the claim follows.
\end{proof}


\def\adritem#1{\hbox{\small #1}}
\def\distance{\hbox{\hspace{3.5cm}}}
\def\apetail{@}
\def\addVdBan{\vbox{
\adritem{E.~P.~van den Ban}
\adritem{Mathematical Institute}
\adritem{Utrecht University}
\adritem{PO Box 80 010}
\adritem{3508 TA Utrecht}
\adritem{The Netherlands}
\adritem{E-mail: E.P.vandenBan{\apetail}uu.nl}
}
}
\def\addKuit{\vbox{
\adritem{J.~J.~Kuit}
\adritem{Institut f\"ur Mathematik}
\adritem{Universit\"at Paderborn}
\adritem{Warburger Stra{\ss}e 100}
\adritem{33089 Paderborn}
\adritem{Germany}
\adritem{E-mail: j.j.kuit{\apetail}gmail.com}
}
}
\mbox{}
\vfill
\hbox{\vbox{\addVdBan}\vbox{\distance}\vbox{\addKuit}}

\end{document}